\documentclass[11pt,reqno]{amsart}
\usepackage{amsmath,amssymb,latexsym,esint,cite,mathrsfs}
\usepackage{verbatim,wasysym}
\usepackage[left=2.6cm,right=2.6cm,top=2.8cm,bottom=2.8cm]{geometry}
\usepackage{tikz,enumitem,graphicx, subfig, microtype, color}
\usepackage{epic,eepic}
\usepackage{amsmath}
\allowdisplaybreaks
\usepackage[colorlinks=true,urlcolor=blue, citecolor=red,linkcolor=blue,
linktocpage,pdfpagelabels, bookmarksnumbered,bookmarksopen]{hyperref}
\usepackage[hyperpageref]{backref}
\usepackage[english]{babel}

\numberwithin{equation}{section}

\newtheorem{thm}{Theorem}[section]
\newtheorem{lem}[thm]{Lemma}

\newtheorem{Prop}[thm]{Proposition}
\newtheorem{Def}[thm]{Definition}
\newtheorem{Rem}[thm]{Remark}

\newcommand{\dist}{{\rm dist}}

\newcommand{\R}{\mathbb{R}}

\def\pa {\partial}

\def\sp {\quad}

\begin{document}
	\baselineskip=14pt
	
	\title[Nondegeneracy]{ Nondegeneracy of positive solutions for a biharmonic hartree equation and its applications}
	
	\author[X. Zhang]{Xinyun Zhang}
	\author[W. Ye]{Weiwei Ye}
	\author[M. Yang]{Minbo Yang}

	\address{Minbo Yang  \newline\indent Department of Mathematics, Zhejiang Normal University, \newline\indent
	Jinhua, Zhejiang, 321004, People's Republic of China}
    \email{mbyang@zjnu.edu.cn}

	\address{Weiwei Ye  \newline\indent Department of Mathematics, Zhejiang Normal University, \newline\indent
		Jinhua, Zhejiang, 321004, People's Republic of China
		\newline\indent Department of Mathematics, Fuyang Normal University, \newline\indent
		Fuyang, Anhui, 236037, People's Republic of China}\email{ yeweiweime@163.com}

	\address{Xinyun Zhang  \newline\indent Department of Mathematics, Zhejiang Normal University, \newline\indent
		Jinhua, Zhejiang, 321004, People's Republic of China}
	\email{xyzhang@zjnu.edu.cn}

	\subjclass[2020]{35J20, 35J60, 35A15}
	\keywords{Critical bi-harmonic Hartree type equation; Hardy-Littlewood-Sobolev inequality; Infinitely many solutions; Poho\v{z}aev identities; Finite dimensional reduction.}
	
	\thanks{$^\ddag$Minbo Yang is the corresponding author who is partially supported by NSFC (11971436, 12011530199) and ZJNSF (LZ22A010001, LD19A010001). Weiwei Ye is partially supported by Natural Science Research key Projects in Universities in Anhui Province of China(2023AH050425).}

	\begin{abstract}
		In this paper, we are interested in some problems related to the following biharmonic hartree equation 
		\begin{equation*}
			\Delta^{2} u=(|x|^{-\alpha}\ast |u|^{p})u^{p-1},\sp  \text{in}\quad\R^N.
		\end{equation*}
		where $p=\frac{2N-\alpha}{N-4}$, $N\geq 9$ and $0<\alpha<N$. First, by using the spherical harmonic decomposition and the Funk-Heck formula of the spherical harmonic functions, we prove the nondegeneracy of the positive solutions of the above biharmonic equation. As applications, we investigate the stability of a version of nonlocal Sobolev inequality
		\begin{equation}
			\int_{\R^N}|\Delta u|^{2} \geq S^{*}\left( \int_{\R^N}\Big(|x|^{-\alpha}\ast |u|^{p}\Big)u^{p} dx\right)^{\frac{1}{p}},
		\end{equation}
		and give a gradient form remainder. Moreover, by applying a finite dimension reduction and local Poho$\check{z}$aev identity, we can also construct multi-bubble solutions for the following equation with potential
		\begin{equation*}
			\Delta^2 u+V(|x'|, x'')u
			=\Big(|x|^{-\alpha}\ast |u|^{p}\Big)u^{p-1}\hspace{4.14mm}x\in \mathbb{R}^N.
		\end{equation*}
		where $N\geq9$, $(x',x'')\in \mathbb{R}^2\times\mathbb{R}^{N-2}$ and $V(|x'|, x'')$ is a bounded and nonnegative function. We will show what is the role of the order of the Riesz potential in proving the existence result. In fact, the existence result is restricted to the range $6-\frac{12}{N-4}\leq\alpha<N$.
	\end{abstract}

	\maketitle
	
	\begin{center}
		\begin{minipage}{8.5cm}
			\small
			\tableofcontents
		\end{minipage}
	\end{center}
	%
	\section{Introduction and main results}
	We study some elliptic variational problems related to the following biharmonic equation with Hartree type nonlinearity
	\begin{equation}\label{he}
		\Delta^{2} u=(|x|^{-\alpha}\ast |u|^{p})u^{p-1},\sp  \text{in}\quad\R^N.
	\end{equation}
	where $p=\frac{2N-\alpha}{N-4}\geq 2$, $N\geq 9$ and $0<\alpha<N$. We are primarily interested in equation\eqref{he} since it is $H^2$ critical which implies that both the equation \eqref{he} and the $H^2$ norm are invariant under the scaling $u_{\rho }(x)=\rho^{\frac{N-4}{2}}u(\rho x)$.
	Equation \eqref{he} is strongly related to the following evolutionary Hartree equation, which includes a fourth-order dispersive term, 
	\begin{equation}\label{he1}
		i\partial_{t}u+\Delta^{2} u=(|x|^{\alpha}\ast |u|^{p})u^{p-1},\sp  (t,x)\in\R\times\R^N,\sp N\geq 9.
	\end{equation}
	The solution for problem \eqref{he} also acts as a stationary solution for equation \eqref{he1}. The Hartree equation with fourth-order dispersive term has numerous fascinating applications in the fields of quantum theory for large systems of non-relativistic bosonic atoms and molecules, as well as in the theory of laser propagation through a medium (\cite{FL,K}). Many qualitative properties of solutions to Hartree equations have been widely researched and applied.

	Before stating our primary motivation, we first recall
	the well-known Hardy-Littlewood-Sobolev(HLS) inequality (see \cite{L,LL}).
	\begin{Prop}\label{pro1.1}
		Let $t,\,r>1$ and $0<\alpha <N$ be such that $\frac{1}{t}+\frac{\alpha}{N}+\frac{1}{r}=2$.
		Then there is a sharp constant $C(N,\alpha,t)$ such that, for $f\in L^{t}(\R^N)$
		and $h\in L^{r}(\R^N)$,
		$$
		\left|\int_{\R^{N}}\int_{\R^{N}}\frac{f(x)h(y)}{|x-y|^{\alpha}}dxdy\right|
		\leq C(N,\alpha,t) |f|_{L^t(\R^N)}|h|_{L^r(\R^N)}.
		$$
		If $t=r=2N/(2N-\alpha)$, then
		$$
		C(t,N,\alpha,r)=C(N,\alpha)=\pi^{\frac{\alpha}{2}}\frac{\Gamma(\frac{N}{2}-\frac{\alpha}{2})}{\Gamma(N-\frac{\alpha}{2})}\left\{\frac{\Gamma(\frac{N}{2})}{\Gamma(N)}\right\}^{-1+\frac{\alpha}{N}}.
		$$
		In this case there is equality is achieved if and only if $f\equiv Ch$ and
		$$
		h(x)=A(\gamma^{2}+|x-a|^{2})^{-(2N-\alpha)/2}
		$$
		for some $A\in \mathbb{C}$, $0\neq\gamma\in\mathbb{R}$ and $a\in \mathbb{R}^{N}$.
	\end{Prop}

From Porposition \ref{pro1.1},  it is easy to see that for any $u\in {{D}^{2,2}(\R^N)}$,
	\begin{equation}\label{NoBi}
		\int_{\R^N}|\Delta u|^{2} \geq S^{*}\left( \int_{\R^N}\Big(|x|^{-\alpha}\ast |u|^{p}\Big)u^{p} dx\right)^{\frac{1}{p}},
	\end{equation}
	where 
	\begin{equation}\label{MinP}
		S^{*}:=\inf_{u\in\mathcal D^{2,2}(\R^N)\backslash\lbrace 0\rbrace}\frac{\int_{\mathbb{R}^{N}}|\Delta u|^2dx}{\left(\int_{\mathbb{R}^{N}}\Big(|x|^{-\alpha}\ast |u|^{p}\Big)u^{p}dx\right)^{\frac{1}{p}}}.
	\end{equation}
And it is obvious that equation \eqref{he} is the Euler-Lagrange equation associated to the minimizing problem \eqref{MinP}.  
	The existence and uniqueness of positive solutions of the critical Hartree equation \eqref{he} with $\alpha=8$ and $p=2$ was proved in \cite{CD}. 
	\begin{lem}(Theorem 1.1, \cite{CD})
		Assume that $N\geq9$. Suppose $u\in D^{2,2}(\R^N)$ is a positive solution of \eqref{he}. Then $u$ is radially symmetric and monotone decreasing about some point $x_0 \in \mathbb{R^N}$, in particular, the positive classical solution $u$ must assume the following form
		$
			W_{x_0,\beta}(x)=\beta^{\frac{N-4}{2}}W(\beta(x-x_0)),
		$
		where $$W(x)=\left( \frac{(N+2)N(N-2)(N-4)!}{\pi^\frac{N}{2}\Gamma(\frac{N-8}{2})}\right) ^{\frac{1}{2}}\left(\frac{1}{1+|x|^2} \right)^{\frac{N-4}{2}}.$$
	\end{lem}
	\begin{Rem} 
		The results in \cite{CD}  tell that all the minimizers for best constant make up a set $\mathcal{M}$ of the form $$\mathcal{M}=\left\lbrace cW_{x_{0},\beta}\; :\:\;c\in \R,\,x_{0}\in R^N,\,\beta\in\R^{+}
		\right\rbrace.$$
		Using the same arguments in \cite{CD}, we omit the proof and conclude for the  general case $0<\alpha<N$ that the unique solution of equation \eqref{he} is of the form
		\begin{equation}\label{REL}	
			W_{y,\beta}(x)=\beta^{\frac{N-4}{2}}W(\beta(x-y)).
		\end{equation}
		and the function $W$ shows in the following
		$$W(x)=C_{N,\alpha}\left(\frac{1}{1+|x|^2} \right)^{\frac{N-4}{2}}:=C_{N,\alpha}U(x),$$
		where $C_{N,\alpha}=\left( \frac{(N+2)N(N-2)(N-4)\Gamma(\frac{2N-\alpha}{2})}{\pi^\frac{N}{2}\Gamma(\frac{N-\alpha}{2})}\right) ^{\frac{N-4}{2(N+4-\alpha)}}$. Moreover, all the minimizers for $S^{*}$  is a set $\mathcal{M}$ of the form $$\mathcal{M}=\left\lbrace cW_{y,\beta}\; :\:\;c\in \R,\,y\in R^N,\,\beta\in\R^{+}
		\right\rbrace.$$
		
	\end{Rem} 

As we all know, the following local Yamabe equation
	\begin{equation}\label{lcritical}
		-\Delta u=u^{\frac{N+2}{N-2}},~~x\in\mathbb{R}^{N},
	\end{equation}
	has a family of solutions of the following form
	\begin{equation}\label{U0}
		U_{\lambda,\xi}(x):=[N(N-2)]^{\frac{N-2}{4}}\Big(\frac{\lambda}{1+\lambda^2|x-\xi|^{2}}\Big)^{\frac{N-2}{2}}.
	\end{equation}
	And equation \eqref{lcritical} has
	an $(N+1)$-dimensional manifold of solutions given by
	$$
	\mathcal{Z}=\left\{z_{\lambda,\xi}=[N(N-2)]^{\frac{N-2}{4}}\Big(\frac{\lambda}{\lambda^2+|x-\xi|^{2}}\Big)^{\frac{N-2}{2}},
	\xi\in\mathbb{R}^{N}, \lambda\in\mathbb{R}^{+}\right\}.
	$$
	We say that every $Z\in\mathcal{Z}$ is nondegenerate if the linearized equation around $Z$
	\begin{equation}\label{Linearized}
		-\Delta v=Z^{\frac{4}{N-2}}v
	\end{equation}
	in $D^{1,2}(\mathbb{R}^N)$ only admits solutions of the form
	$$
	\eta=aD_{\lambda}Z+\mathbf{b}\cdot\nabla Z,
	$$
	where $a\in\mathbb{R},\mathbf{b}\in\mathbb{R}^{N}$. 
	
	The well-known Sobolev inequality for bi-Laplace(see \cite{EFJ}) is
	\begin{equation}\label{SI}
		\left\| u\right\|_{\mathcal D^{2,2}(\R^N)}\geq S_{2}\left\| u\right\| _{L^{\frac{2N}{N-4}}},
	\end{equation}
	where $S_2$ is the best Sobolev constant and  $\mathcal D^{2,2}(\R^N)$ denotes the closure of $C_{0}^{\infty}(\R^N)$ with respect to
	the norm $\left\| u\right\|_{\mathcal D^{2,2}(\R^N)}=\left\|\Delta u\right\|_{ L^{2}(\R^N)}$. P.L.Lions\cite{L3} first obtained the existence of extremal functions of \eqref{SI} and proved the radial symmetry of any extremal function of  \eqref{SI}. Edmunds, Fortunato and Janelli \cite{EFJ} gave
	best constant $S_2$ that
	$$
	S_{2}=\pi^2(N-4)(N-2)N(N+2)\Big(\frac{\Gamma(N/2)}{\Gamma(N)}\Big)^{\frac{4}{N}},
	$$ 
	and showed that all the extremal functions achieved $S_2$ consist a $(N+2)$ dimensional manifold 
	$$\mathcal{M}_2=\left\lbrace cU_{x_{0},\lambda}\; :\:\;c\in \R,\,x_{0}\in R^N,\,\lambda\in\R^{+}
	\right\rbrace,$$
	where $U_{x_{0},\lambda}=\lambda^{\frac{N-4}{2}}U(\lambda(x-x_{0}))$ 
	and
	$$
	U(x)=[(N-4)(N-2)N(N+2)]^{\frac{N-4}{8}}(1+|x|^{2})^{-\frac{N-4}{2}}.
	$$  
	As shown in Lin \cite{L2}, the Euler-Lagrange equation associated to \eqref{SI} 
	\begin{equation}\label{biL}
		\Delta^2 u=u^{\frac{N+4}{N-4}},\hspace{5mm} u>0 \hspace{3mm}\text{in}\hspace{3mm} \mathbb{R}^N.
	\end{equation}
	has the smooth solutions of the form $U_{x_{0},\lambda}$.
	
	Lu and Wei \cite{LW} conducted the linearized equation around $U$ an analysis on the eigenvalues of the fourth-order equation that is mentioned below
	\begin{equation}\nonumber
		\Delta^{2} v-\mu S_{2}^{q+1}U^{\frac{8}{N-4}}v=0
		\ \ v\in \mathcal{D}_{0}^{2,2}(\R^N),
	\end{equation} 
	where $q=\frac{N-4}{N+4}$, $S_2$ is the best Sobolev constant and $U$ is the only radial solution that satisfies the Sobolev equality. The authors showed the corresponding eigenfunction spaces are
	$$
	V_{1}=\left\lbrace U\right\rbrace , V_{q}=\left\lbrace \frac{\pa U}{\pa y_j},j=1,\dots,N, x\cdot\nabla U +\frac{N-4}{2}U\right\rbrace,
	$$
	which implies the non-degeneracy of $U$. Many studies have been conducted on the qualitative characteristics of solutions to fourth-order elliptic equations, including those by Xu \cite{X}, 
	Wei and Xu \cite{WX}, Berchio, Gazzola and Mitidieri \cite{BGM}, Chang and Yang \cite{CY} and the references therein.  Pistoia and Vaira\cite{PV} gave the nondegeneracy result about $p$-Laplace operator. And because of the widely applications of the non-degeneracy, many researchers focus on this meaningful property. Recently, Deng and his collaborators \cite{DGT,DT1,DT2,DT3,DT6} concerned some type of Caffarelli-Kohn-Nirenberg(CKN) inequality and characterized all the solutions to the related linearized problem about radial extremals.

    As far as we know, there are no nondegeneracy results for the bi-harmonic equation with Hartree type nonlinearity yet, and so it is quite natural to ask if the  nondegeneracy result still holds for the nonlocal bi-harmonic Hartree equation.
	To investigate the  nondegeneracy property, the classical method relies on the Sturm-Liouville theory (see i.e. \cite{W}) greatly. However when the equations involving nonlocal terms, the problem will become more complicated. Regarding the Hartree equation, 
	\begin{equation}\label{he3}
		-\Delta u+u=\Big(I_{2}\ast |u|^{2}\Big)u
		\ \ \mbox{in}\ \R^N,
	\end{equation}
	where $I_{2}=\frac{1}{(N-2)\mathbb{S}^{N-1}}\frac{1}{(x-2)^{N-2}}$. When $N=3$, the non-degeneracy of the ground states has been proven through various methods. J. Wei and M. Winter \cite{WW} analyzed the Schr\"{o}dinger-Newton equation \begin{equation}\nonumber
		\left\{\begin{array}{ll}
			-\Delta u+u=vu
			&\mbox{in}\ \R^N,\\[1mm]
			-\Delta v=u^2
			&\mbox{in}\ \R^N,
		\end{array}\right.
	\end{equation}
	which is equivalent to \eqref{he3}, and found that this approach has an advantage due to its locality. This allows the equation to be reduced to a series of ODE systems. E. Lenzmann\cite{L1} found another method that relied on spectral analysis of the linearized operators at the ground states.
	For the critical Hartree equation 
	\begin{equation}\label{he2}
		-\Delta u=\Big(|x|^{-\mu}\ast |u|^{2_{\mu}^{\ast}}\Big)|u|^{2_{\mu}^{\ast}-2}u
		\ \ \mbox{in}\ \R^N,
	\end{equation}
	where $2_{\mu}^{\ast}=\frac{2N-\mu}{N-2}$, the nondegeneracy property is completely acquired. If $\mu$ is close to $N$, the limit equation of \eqref{he2} is the critical Lane-Emden equation whose nondegeneracy property is well known. According to this result, the authors proved  the nondegenerate property by approximation approach in \cite{DY}. In \cite{YYZ}, Yang and Zhao got the nondegeneracy result in $N=6,\mu=4$ case. Li et al.\cite{LXLTX}, using the key spherical harmonic decomposition and the Funk-Hecke formula of the spherical harmonic functions, have recently proven that positive bubble solutions for Hartree equation \eqref{he2} are nondegenerate. 
	In \cite{YGRY}, the authors proved nondegeneracy result for the critical Hartree system  
	\begin{equation}\nonumber
		\left\{\begin{array}{ll}
			-\Delta u=\alpha_1\big(|x|^{-4}\ast u^{2}\big)u+\beta \big(|x|^{-4}\ast v^{2}\big)u
			&\mbox{in}\ \R^6,\\[1mm]
			-\Delta v=\alpha_2\big(|x|^{-4}\ast v^{2}\big)v +\beta\big(|x|^{-4}\ast u^{2}\big)v
			&\mbox{in}\ \R^6.
		\end{array}\right.
	\end{equation}
	By using linearization and spectral analysis, they converted the nondegeneracy for the system 
	into analyzing non-degeneracy result of single equation.

	In this paper, we will prove a nondegeneracy result for the critical bi-harmonic Hartree type equation. 
The main result of this paper can be expressed as follows
	\begin{thm}\label{Non}
		The linearized equation of \eqref{he} around the solution $W(x)$ given by
		\begin{equation}\label{linear}
			\Delta^{2} \varphi=p (|x|^{-\alpha}\ast W^{p-1}\varphi)W^{p-1}+(p-1)(|x|^{-\alpha}\ast |W|^p)W^{p-2}\varphi ,\sp  \text{in}\quad\R^N.
		\end{equation}
		only admits solutions of the form
		\begin{equation}\nonumber
			\varphi=a\left[ \frac{(N-4)}{2}W(x)+x\cdot\nabla W(x)\right] +\sum_{j=1}^{N}b{j}\frac{\partial W}{\partial x_j},	
		\end{equation}
		where $a, b_{j}\in \R$.
	\end{thm}
	
	The nondegeneracy result is interesting in itself and have many applications in different topics of elliptic partial differential equations. For example, it applies to study the stability problem of the Sobolev inequality. This type of question was first raised by Brezis and Lieb in \cite{BL}. They asked whether a remainder term - proportional to the quadratic distance from function $u$ to the manifold $\mathcal{M}_1$ - can be added to the right side of classical Sobolev inequality. That is whether the following refined classical Sobolev inequality holds:
	$$
	\| \nabla \phi\|_{2}^{2}-S_{1}^{2}\| \phi\|_{\frac{2N}{N-2}}^{2}\geq C d^{2}(\phi,\mathcal{M}_1),N>2, \phi\in {{D}_{0}^{1,2}(\R^N)},
	$$
	where ${D}_{0}^{1,2}(\R^N)$ is the completion of $C_{0}^{\infty}(\R^N)$ under the norm of $\phi\|_{2}$ when
	$N\geq 3, C > 0$ and $\mathcal{M}_1$ is the $(N + 2)$-dimensional manifold which consists of all
	solutions attaining the classical Sobolev inequality. Bianchi and Egnell \cite{BE} affirmed the question and their result was later extended to the biharmonic case \cite{LW}, the polyharmonic case \cite{BWM}, and the fractional-order case \cite{CFW}.  Figalli and Neumayer \cite{FN}, Figalli and Zhang \cite{FZ} obtained the remainder terms of $p$-Laplace Sobolev inequality. R$\breve{a}$dulescu et al.\cite{RSW} obtained the remainder terms of Hardy-Sobolev inequality with exponent two. Wang and Willem\cite{WW2} gave the remainder terms for (CKN) inequality. Deng and his collaborators investigated the remainder term of several types of (CKN) inequalities involving weighted  fourth-order equations and weighted $p$-Laplace equations. One can refer \cite{DGT,DT1,DT2,DT3,DT6} for more details.
	By using the finite dimensional reduction method, Ciraolo et al.\cite{CFM}, Figalli and Glaudo\cite{FG}, Deng et al.\cite{DSW} proved the sharp stability of profile decomposition for critical Sobolev inequality along Struwe's fundamental result in \cite{S}. Then Wei and Wu\cite{WW3} generalized the special case $a=b=0$(Sobolev inequality) to (CKN) inequality in a proper parameter region and gave the remainder term. Piccione et al. \cite{PYZ} established the quantitative stability of a nonlocal Sobolev inequality. 
	As shown in recent paper\cite{DTYZ}, Deng et al. proved the existence of the gradient type remainder term for a nonlocal Sobolev inequality and gave a remainder term in the weak $L^{\frac{2N}{N-2}}$-norm of the same inequality in bounded domain. 

	Inspired by \cite{DTYZ}, we give an improved version of the nonlocal Sobolev inequality \eqref{NoBi}. That is
	\begin{thm}\label{INE}
		There exist two constants $A_1\leq A_2 $ which depending only on the dimension, such that
		\begin{equation}\label{INERE}
			A_{1}\dist\left(u, \mathcal{M} \right)^2\geq\int_{\R^N}|\Delta u|^{2} - S^{*}\int_{\mathbb{R}^{N}}\Big(|x|^{-\alpha}\ast |u|^{p}\Big)u^{p} dx\geq A_{2}\dist\left(u, \mathcal{M} \right)^2,
		\end{equation}
		for every $u\in {{D}^{2,2}(\R^N)}$, where $\mathcal{M}=\left\lbrace cW_{y,\beta}\; :\:\;c\in \R,\, y\in R^N,\, \beta\in\R^{+}
		\right\rbrace$ is $(N+2)$-dimensional manifold and $\dist\left(u, \mathcal{M} \right)=\inf\limits_{c\in \R,y\in\R^N,\beta\in\R^+} \left\| u-cW_{y,\beta}\right\|$.
	\end{thm}
	
	The non-degeneracy of positive solutions is also a key ingredient in the Lyapunov-Schmidt reduction method of constructing multi-bubble solutions. In \cite{WY1}, Wei and Yan developed
	the technique that allowed to study the prescribed scalar curvature problem on $\mathbb{S}^N$
	$$
	-\Delta_{\mathbb{S}^N} u +\frac{N(N-2)}{2}u= K(x)u^{\frac{N+2}{N-2}}\hspace{4.14mm}\mbox{on}\hspace{1.14mm}  \mathbb{S}^N.
	$$
	Assuming that $K(x)$ is positive and rotationally symmetric and has a local maximum point between the poles, by using the stero-graphic projection, the prescribed scalar curvature problem can be reduced into
	$$
	-\Delta u = K(x)u^{\frac{N+2}{N-2}}\hspace{4.14mm}\mbox{in}\hspace{1.14mm}  \R^N.
	$$
	The authors took the number of the bubbles of the solutions as parameter and proved the existence of infinitely many non-radial positive solutions whose energy can be made arbitrarily large. We may also turn to the works by Deng, Lin, Yan \cite{DLY}, Guo, Peng, Yan \cite{GPY} and Li, Wei, Xu \cite{LWX} for the existence and local uniqueness of multi-bump solutions.
	For the critical Schr\"{o}dinger equation
	$$
	-\Delta u+V(|x|)u = u^{\frac{N+2}{N-2}}\hspace{4.14mm}\mbox{in}\hspace{1.14mm}  \R^N,
	$$
	Chen, Wei and Yan \cite{CWY}
	applied the reduction argument to study the existence of infinitely many positive solutions when $V$ is radially symmetric and $r^{2}V (r)$ has a local maximum point, or a local
	minimum point $r_{0} > 0$ with $V (r_{0}) > 0$. In \cite{PWY}, Peng, Wang and Yan developed a new idea that allowed also to construct bubbling solutions concentrating at saddle points of some functions. They used the Poho\v{z}aev identities to find algebraic equations which determine the location of the bubbles. 
	In \cite{GMYZ}, Gao, Moroz, Yang and Zhao studied a class of critical Hartree equations with axisymmetric potentials,
	$$
	-\Delta u+ V(|x'|,x'')u
	=\Big(|x|^{-4}\ast |u|^{2}\Big)u\hspace{4.14mm}\mbox{in}\hspace{1.14mm} \mathbb{R}^6,
	$$
	where $(x',x'')\in \mathbb{R}^2\times\mathbb{R}^{4}$, $V(|x'|, x'')$ is a bounded nonnegative function in $\mathbb{R}^{+}\times\mathbb{R}^{4}$. By applying a finite dimensional reduction argument and developing novel local Poho\v{z}aev identities, they proved that if the function $r^2V(r,x'')$ has a topologically nontrivial critical point then the problem admits infinitely many solutions with arbitrary large energies. In \cite{GLN}, Guo et al. considered  the following nonlinear equation with critical exponent
	and polyharmonic operator:
	$$
	(-\Delta)^{m} u+ V(|x'|, x'')u
	=u^{m^{*}-1}\hspace{4.14mm} u>0\hspace{1.14mm}u\in\mathcal{D}^{m,2}(\mathbb{R}^N),
	$$
	where $m^{*}=\frac{2N}{N-2m}$, $N\geq 4m+1$,$m\in \mathbb{N_+}$, $(|x'|,x'')\in \mathbb{R}^{2}\times \mathbb{R}^{N-2}$ and $ V(|x'|,x'')$ is a bounded non-negative function in $\mathbb{R}^{+}\times\mathbb{R}^{N-2}$. 
	Using a finite reduction argument and local Poho$\check{z}$aev type identities, they showed that if $N\geq 4m+1$ and $r^{2m}V(r, x)$ has a stable critical point $(r_0, x_0)$. Then the above problem has infinitely many solutions,  whose energy can be arbitrarily large.  
	
	As far as we know, there are no results concerning multi-bubble solutions to 
	\begin{equation}\label{CFL}
		\Delta^2 u+V(r, x'')u
		=\Big(|x|^{-\alpha}\ast |u|^{p}\Big)u^{p-1}\hspace{4.14mm}x\in \mathbb{R}^N.
	\end{equation} 
	where $N\geq9$, $6-\frac{12}{N-4}<\alpha< N$, $(x',x'')\in \mathbb{R}^2\times\mathbb{R}^{N-2}$ and $V(r, x'')=V(|x'|, x'')$ is a bounded and nonnegative function. In order to constructing solutions concentrated at $(r_0, x_0'')$, we use $W_{z_j,\beta}$ as an approximate solution. Denote
	$$
	Z_{z_j,\beta}(x)=\xi W_{z_j,\beta}(x), \
	Z_{\overline{r},\overline{x}'',\beta}(x)=\sum_{j=1}^{m}Z_{z_j,\beta}(x), \
	Z_{\overline{r},\overline{x}'',\beta}^{\ast}(x)=\sum_{j=1}^{m}W_{z_j,\beta}(x),
	$$
	where $\overline{x}''$ is a vector in $\mathbb{R}^{N-2}$. With Theorem \ref{Non}, we can use the Lyapunov Schmidt argument to construct multi-bubble solutions for \eqref{he} as follows
	\begin{thm}\label{Ms}
		Suppose that $N\geq 9 $, $V>0$ is bounded and belongs to $C^2$ and satisfies the following assumption:
		\begin{itemize}
			\item[$\textbf{(V)}$]  The function $r^4 V (r, x'')$ has a critical point $(r_0, x_0'')$ such that $r_0 > 0$ and $V (r_0, x_0'') >0$, and $$\deg(\nabla(r^4V(r, x'')),(r_0, x_0'')) \neq 0,$$
		\end{itemize}
		then there exists an integer $m_{0}>0$ such that for any integer $m \geq m_{0}$, equation \eqref{CFL} has a solution $u_m$ of the form
		\begin{equation}\label{form}
			u_m=Z_{\overline{r}_{m},\overline{x}_{m}'',\beta_{m}}+
			\phi_{\overline{r}_{m},\overline{x}_{m}'',\beta_{m}}=\sum_{j=1}^{m}\xi W_{z_j,\beta_{m}}+\phi_{\overline{r}_{m},\overline{x}_{m}'',\beta_{m}},
		\end{equation}
		where $\phi_{\overline{r}_{m},\overline{x}_{m}'',\beta_{m}}\in H_{s}$ and $\beta_{m}\in\big[L_0m^{\frac{N-4}{N-8}},L_1m^{\frac{N-4}{N-8}}\big]$. Moreover, as $m\rightarrow\infty$, $(\overline{r}_m, \overline{x}_m'')\rightarrow(r_0, x_0'')$, and $\beta_{m}^{-\frac{N-4}{2}}\|\phi_{\overline{r}_{m},\overline{x}_{m}'',\beta_{m}}\|_{L^{\infty}}\rightarrow0$.
		As a consequence, equation \eqref{CFL} has infinitely many solutions whose energy can be arbitrarily large.
	\end{thm}
	
	This paper is organized as follows: In Section 2, we will prove a nondegeneracy result for the critical bi-harmonic Hartree equation \eqref{he}. In Section 3, we give the behavior
	near minimizers manifold $\mathcal{M}$ and complete the proof of Theorem \ref{INE}. To state another application of our nondegeneracy result, in section 4 we will prove Theorem \ref{Ms} using finite reduction and local Poho$\check{z}$aev identities.

	\section{Nondegeneracy}
	In this section, we are going to prove the nondegeneracy of the unique positive solution for the nonlinear Hartree equation \eqref{he} with upper critical exponent.
	
	Inspired by\cite{LXLTX}, we need to transform equation \eqref{linear} into the sphere by stereographic projection $S: \R^N\mapsto \mathbb{S}^N \setminus \left\lbrace 0,0,\cdots,0,-1\right\rbrace $ which is defined by
	\begin{equation}\label{S}
		S(x)=\left( \frac{2x}{1+|x|^2},\frac{1-|x|^2}{1+|x|^2}\right),
	\end{equation}
	where $\mathbb{S}^N$ is the unite sphere in $\R ^{N+1}$. Denote the corresponding inverse map $S^{-1}$ by  
	\begin{equation}\label{S-1}
		S^{-1}(\xi_{1},\xi_{2},\cdots,\xi_{N+1})=\left( \frac{\xi_{1}}{1+\xi_{N+1}},\frac{\xi_{2}}{1+\xi_{N+1}},\cdots,\frac{\xi_{N}}{1+\xi_{N+1}}\right).
	\end{equation}
	And for any integrable function $\varphi$ on $\mathbb{S}^N$ we have
	\begin{equation}
		\int_{\mathbb{S}^N}\varphi(\xi)d\xi=	\int_{\mathbb{R}^N}\varphi(S(x))J_{S}(x)dx,
	\end{equation}
	where $J_{S}(x)=(\frac{2}{1+|x|^2})^{N}$ is the Jacobian of the transformation.
	
	By simple calculation, we can get
	\begin{equation}\label{Sx}
		|S(x)-S(y)|^2=|x-y|^{2}\frac{2}{1+|x|^2}\frac{2}{1+|y|^2}.
	\end{equation}
	For any $\psi:\R^N \mapsto \R$, define $\Psi:\mathbb{S}^N \setminus \left\lbrace 0,0,\cdots,0,-1\right\rbrace\mapsto\R$ as follow
	\begin{equation}\label{S2}
		\Psi(\psi)(\xi)=J_{S}^\frac{4-N}{2N}(S^{-1}(\xi))\psi(S^{-1}(\xi)).
	\end{equation}
	And for any $\iota:\mathbb{S}^N \setminus \left\lbrace 0,0,\cdots,0,-1\right\rbrace \mapsto \R$, define $\Phi:\R^N\mapsto\R$ as follow
	\begin{equation}\label{S3}
		\Phi(\iota)(x)=J_{S}^\frac{N-4}{2N}(x)\iota(Sx).
	\end{equation}
	By differentiating $W(x)$ with respect to $x$ and $\beta$ at $(x_0,\beta)=(0,1)$, we have
\begin{equation}\label{U1}
	\varphi_{j}:=\frac{\pa W}{\pa x_{j}}=(4-N)W(x)\frac{x_j}{1+|x|^2},\sp 1\leq j\leq N,
\end{equation}
\begin{equation}\label{U2}
	\varphi_{N+1}:=\frac{\pa W}{\pa \beta}=\frac{(N-4)}{2}W(x)+x\cdot\nabla W(x)=\frac{(N-4)}{2}W(x) \Big(\frac{1-|x|^2}{1+|x|^2}\Big).
\end{equation}
	Moreover, we have
	\begin{lem}\label{L1}
		Let $\Psi:\mathbb{S}^N \setminus \left\lbrace 0,0,\cdots,0,-1\right\rbrace\mapsto\R$ be defined by \eqref{S2}, then for any $\xi\in \mathbb{S}^N $, we have
		$$
		\Psi(\varphi_{j})(\xi)=(4-N)2^{\frac{2-N}{2}}C_{N,\alpha}\xi_{j},\sp 1\leq j\leq N,
		$$
		$$
		\Psi(\varphi_{N+1})(\xi)=(N-4)2^{\frac{2-N}{2}}C_{N,\alpha}\xi_{N+1},
		$$
		where $\varphi_{j},1\leq j\leq N+1$ is defined by \eqref{U1} and \eqref{U2}.
	\end{lem}
	\begin{proof}
		An easy computation shows that 
		$$ 
		\left|S^{-1} \xi\right|^2=\frac{1-\xi_{N+1}^2}{(1+\xi_{N+1})^2}, 
		$$
		then 
		\begin{equation}\nonumber
			\begin{split}
				\Psi(\varphi_{j})(\xi)&=(\frac{2}{1+|S^{-1}(\xi)|^2})^{\frac{4-N}{2}}(4-N)C_{N,\alpha}(\frac{1}{1+|S^{-1}(\xi)|^2})^{\frac{N-4}{2}}\frac{(S^{-1}(\xi))_j}{1+|S^{-1}(\xi)|^2}\\
				&=2^{\frac{4-N}{2}}(4-N)C_{N,\alpha}\left(\frac{\xi_{j}}{1+\xi_{N+1}} \right)\left(\frac{1+\xi_{N+1}}{2}\right) \\
				&=(4-N)2^{\frac{2-N}{2}}C_{N,\alpha}\xi_{j},\sp 1\leq j\leq N.
			\end{split}
		\end{equation}
		Similarly, we get $\Psi(\varphi_{N+1})(\xi)=(N-4)2^{\frac{2-N}{2}}C_{N,\alpha}\xi_{N+1}$.
	\end{proof}
	
	Now, we need to introduce some results related to spherical harmonic functions. Let $Y_{k}(x)$
	is the restriction on $\mathbb{S}^N$ of real harmonic $k$-order homogeneous polynomials. The linear space of $\left\lbrace Y_{k,i}(x)|\sp 1\leq i\leq \dim \mathscr {H}_{k}^{N+1} \right\rbrace_{k}$ is denoted by $\mathscr {H}_{k}^{N+1}$. We know the following orthogonal decomposition $$L^2(\mathbb{S}^N)=\bigoplus\limits_{k=0}^{\infty}\mathscr {H}_{k}^{N+1},$$
	and
	\begin{equation}\nonumber
		\dim \mathscr {H}_{k}^{N+1}=\left\lbrace \begin{aligned}
			&1,\sp &\text{if} \sp k=0,\\
			&N+1 ,\sp &\text{if}\sp k=1,\\
			&C^{k}_{k+N}-C^{k-2}_{k+N-2},\sp &\text{if} \sp k\geq2.\\
		\end{aligned}
		\right.
	\end{equation}
	
	The following Funk-Heck formula of the spherical harmonic function plays a key role in the proof of the Theorem \ref{Non}.
	\begin{lem}\label{FH}{\rm{(\cite{AH,DX})}}
		Let $\alpha\in (0,N)$, then for any $Y\in \mathscr {H}_{k}^{N+1}$, we have 
		\begin{equation}\label{FH}
			\int_{\mathbb{S}^N}\frac{Y(\eta)}{|\xi-\eta|^\alpha}d\eta=\lambda_{k}(\alpha)Y(\xi),
		\end{equation}
		where 
		\begin{equation}\label{FHL}
		\lambda_{k}(\alpha)=2^{N-\alpha}\pi^{\frac{N}{2}}\frac{\Gamma\left(k+\frac{\alpha}{2} \right)\Gamma\left(\frac{N-\alpha}{2} \right) }{\Gamma\left(\frac{\alpha}{2} \right)\Gamma\left(k+N-\frac{\alpha}{2} \right)}.
	    \end{equation}
	\end{lem} 
	Here we need to pay attention to $\lambda_{k}$. We check at once that $\lambda_{k+1}(\alpha)<\lambda_{k}(\alpha), (k\geq 0)$
	for every $\alpha\in (0,N)$ and obviously
	$$
	\lambda_{0}(N-4)=\frac{2^6}{N(N+2)}\frac{\pi^{\frac{N}{2}}}{\Gamma(\frac{N}{2})},\sp
	\lambda_{1}(N-4)=\frac{N-4}{N+4}\lambda_{0}(N-4),
	$$
	$$
	\lambda_{0}(\alpha)=2^{N-\alpha}\pi^{\frac{N}{2}}\frac{\Gamma\left(\frac{N-\alpha}{2} \right)}{\Gamma\left(N-\frac{\alpha}{2} \right)},\sp
	\lambda_{1}(\alpha)=\frac{\alpha}{2N-\alpha}\lambda_{0}(\alpha).
	$$
	
	We can give the following lemma which is a direct consequence of the above Funk-Heck formula. 
	\begin{lem}\label{FH1}
		Let $\alpha\in (0,N)$ and $\lambda_{k}(\alpha)$ be defined by \eqref{FHL}, then for any $Y\in \mathscr {H}_{k}^{N+1}$, we have 
		\begin{equation}\label{FH11}
			\int_{\mathbb{S}^N}\int_{\mathbb{S}^N}\frac{1}{|\xi-\eta|^{N-4}}\frac{1}{|\eta-\sigma|^{\alpha}}Y(\eta)d\eta d\sigma=\lambda_{k}(N-4)\lambda_{0}(\alpha)Y(\xi),
		\end{equation}
		\begin{equation}\label{FH12}
			\int_{\mathbb{S}^N}\int_{\mathbb{S}^N}\frac{1}{|\xi-\eta|^{N-4}}\frac{1}{|\eta-\sigma|^{\alpha}}Y(\sigma)d\sigma d\eta =\lambda_{k}(N-4)\lambda_{k}(\alpha)Y(\xi).
		\end{equation}
	\end{lem} 
	
	Next we will give some useful estimates.
	For convenience, write
	\begin{equation}\label{T}
		\begin{aligned}
			T(\varphi)(x)&=p\left( |x|^{-\alpha}\ast U^{p-1}\varphi\right) )U^{p-1}+(p-1)\left(|x|^{-\alpha}\ast |U|^p\right)U^{p-2}\varphi\\
			&:=T_{1}(\varphi)(x)+T_{2}(\varphi)(x),
		\end{aligned}
	\end{equation}
	then we can get an estimate of $T(\varphi)$.
	\begin{lem}\label{Est1}
		Let $\alpha\in (0,N)$ and $\theta\in \left[ 0,N-4\right] $, if $\varphi$ satisfies $|\varphi|\lesssim\frac{1}{\tau(x)^\theta}$ then
		\begin{equation}\nonumber
			\left| T(\varphi)(x)\right| \lesssim\frac{1}{\tau(x)^{\theta+8}}.
		\end{equation}
		where $\tau(x) =\left(1+|x|^2 \right)^\frac{1}{2}. $
	\end{lem} 
	\begin{proof}
		From Lemma A.1 in \cite{LXLTX}, by direct calculation we get
		\begin{equation}
			\begin{split}
				\left| T_{1}(\varphi)(x)\right|=\left|p\left( |x|^{-\alpha}\ast U^{p-1}\varphi\right) )U^{p-1} \right|&\lesssim\int_{\R^N}\frac{1}{|x-y|^{\alpha}}\frac{1}{\left(1+|y|^2 \right)^{\frac{N-\alpha+4}{2}}} \frac{1}{\tau(y) ^{\theta}}dy\frac{1}{\left(1+|x|^2 \right)^{\frac{N-\alpha+4}{2}}}\\
				&=\frac{1}{\tau(x) ^{N-\alpha+4}}\int_{\R^N}\frac{1}{|x-y|^{\alpha}}\frac{1}{\tau(y)^{N-\alpha+4}} \frac{1}{\tau(y) ^{\theta}}dy\\
				&\lesssim\left\lbrace \begin{aligned}
					&\tau(x) ^{-N-\theta-8+\alpha},\sp &if \sp \theta+4<\alpha,\\
					&\tau(x)^{-N-4}\left( 1+\log(\tau(x))\right) ,\sp &if \sp \theta+4=\alpha,\\
					&\tau(x) ^{-N-4},\sp &if \sp \theta+4>\alpha.\\
				\end{aligned}
				\right.
			\end{split}
		\end{equation}
		Since $N+4\geq\theta+8$ and $\frac{1+\log(\tau(x))}{\tau(y) ^{N-\alpha}}\lesssim 1$, we deduce 
		\begin{equation}\label{T1}
			\left| T_{1}(\varphi)(x)\right| \lesssim\frac{1}{\tau(x)^{\theta+8}}.
		\end{equation}
		Using the same method, we can get
		\begin{equation}\label{T2}
			\left| T_{2}(\varphi)(x)\right| \lesssim\frac{1}{\tau(x)^{\theta+8}},
		\end{equation}
		this finish the proof.
	\end{proof}
	
	The equation \eqref{linear} can be rewritten in the following integral form
	\begin{equation}\label{Inte}
		\varphi(x)=C^\frac{2(N+4-\alpha)}{N-4}_{N,\alpha}C_{N}\int_{\R^N}\frac{1}{|x-y|^{N-4}}T(\varphi)(y)dy,
	\end{equation}
	where $C_{N}=\frac{\Gamma(\frac{N}{2})}{4(N-2)(N-4)\pi^{\frac{N}{2}}}$.
	We can improve the estimate of $\varphi(x)$ and have the following lemma.
	
	\begin{lem}\label{Est2}
		Let $\alpha\in (0,N)$, $N\geq9$, if $\varphi\in L^{\infty}(\R^N)$ satisfies \eqref{Inte}, then 
		\begin{equation}
			|\varphi(x)|\lesssim\frac{1}{\tau(x)^{N-4}},
		\end{equation}
		where $\tau(x) =\left(1+|x|^2 \right)^\frac{1}{2}. $
	\end{lem}
	\begin{proof}
		As Lemma \ref{Est1}, for $\varphi\in L^{\infty}(\R^N)$, we have 
		$$\left| T(\varphi)(x)\right| \lesssim\frac{1}{\tau(x)^{2k+8-\varepsilon}},$$
		where $0<\varepsilon\ll 1$ and $1\leq k\leq \left[\frac{N-4}{2} \right]$. Then by Lemma A.1 in \cite{LXLTX} we deduce that 
		\begin{equation}\label{Es}
			|\varphi(x)|\lesssim \left\lbrace \begin{aligned}
				&\frac{1}{\tau(x) ^{N-4}},\sp &if \sp N<2k+8-\varepsilon,\\
				&\frac{1}{\tau(x) ^{2(k+1)-\varepsilon}},\sp &if \sp N\geq 2k+8-\varepsilon.\\
			\end{aligned}
			\right.
		\end{equation}
		On the other hand, for the same reason we derive $\left| T(\varphi)(x)\right| \lesssim\frac{1}{\tau(x)^{8}}$. By the induction argument for $k=1,\cdots, \left[\frac{N-4}{2} \right]$ and $N\geq 9$ we obtain that
		\begin{equation}\nonumber
			|\varphi(x)|\lesssim\frac{1}{\tau(x) ^{N-4}}.
		\end{equation}
	\end{proof}
	
	Let
	\begin{equation}
		\begin{split}
			T_{\mathbb{S}^N}\Psi(\varphi)(\xi)&:=pT_{\mathbb{S}^N, 1}\Psi(\varphi)(\xi)+(p-1)T_{\mathbb{S}^N,2}\Psi(\varphi)(\xi)\\
			&=p\cdot 2^{-(p-1)(N-4)} \int_{\mathbb{S}^N}\int_{\mathbb{S} ^N}\frac{1}{|\xi-\eta|^{N-4}}\frac{1}{|\eta-\sigma|^{\alpha}}\Psi(\varphi)(\sigma)d\sigma d\eta\\
			&+(p-1)2^{-(p-1)(N-4)} \int_{\mathbb{S}^N}\int_{\mathbb{S} ^N}\frac{1}{|\xi-\eta|^{N-4}}\frac{1}{|\eta-\sigma|^{\alpha}}\Psi(\varphi)(\eta) d\eta d\sigma.
		\end{split}
	\end{equation}
	To prove the theorem, we need to transform the integral equation \eqref{Inte} on $\R^N$ to that on $\mathbb{S}^N$ by the stereographic projection. Therefore, we have the following lemma.
	\begin{lem}\label{SN}
		Let $\alpha\in (0,N)$, $N\geq9$, if $\varphi\in L^{\infty}(\R^N)$ satisfies \eqref{Inte}, then $\Psi(\varphi)\in L^2(\mathbb{S}^N)$ and
		\begin{equation}\label{EqTS}
			\Psi(\varphi)(\xi)=C^\frac{2(N+4-\alpha)}{N-4}_{N,\alpha}C_{N}T_{\mathbb{S}^N}\Psi(\varphi)(\xi).
		\end{equation}
	\end{lem}
	\begin{proof}
		Our proof will be divided into two steps. 
		
		\textbf{Step 1}: Let's prove $\Psi(\varphi)\in L^2(\mathbb{S}^N)$ first. According to \eqref{S2} and Lemma \ref{Est2}, we get
		$$
		\int_{\mathbb{S}^N}\left|\Psi(\varphi)(\xi) \right|^{2}d\xi\lesssim\int_{\R^N}\frac{1}{\tau(x)^{2N-8}}J_S^{\frac{4}{N}}(x)dx\lesssim \int_{\R^N}\frac{1}{\tau(x)^{2N}}dx<+\infty.
		$$
		
		\textbf{Step 2}: Now we are going to prove equation \eqref{EqTS}.
		\begin{equation}\nonumber
			\begin{split}
				T_{1}(\varphi)(x)&=p\int_{\R^N}\frac{1}{|x-y|^{\alpha}}U^{p-1}(y)\varphi(y)dyU^{p-1}(x)\\
				&=p\cdot 2^{-(p-1)(N-4)}J_{S}^{\frac{N-\alpha+4}{2N}}(x) \int_{\R^N}\frac{1}{|x-y|^{\alpha}}J_{S}^{\frac{N-\alpha+4}{2N}}(y)\varphi(y)dy\\
				&=p\cdot 2^{-(p-1)(N-4)}J_{S}^{\frac{N+4}{2N}}(x) \int_{\R^N}\frac{1}{|Sx-Sy|^{\alpha}}J_{S}^{\frac{4-N}{2N}}(y)\varphi(y)J_{S}(y)dy\\
				&=p\cdot 2^{-(p-1)(N-4)}J_{S}^{\frac{N+4}{2N}}(x)\int_{\mathbb{S} ^N}\frac{1}{|Sx-\eta|^{\alpha}}\Psi(\varphi)(\eta)d\eta
			\end{split}
		\end{equation}
		\begin{equation}\label{ST1}
			\begin{split}
				&\int_{\R^N}\frac{1}{|x-y|^{N-4}}T_{1}(\varphi)(y)dy\\
				&=p\cdot 2^{-(p-1)(N-4)}\int_{\R^N}\frac{1}{|x-y|^{N-4}}J_{S}^{\frac{N+4}{2N}}(y) \int_{\mathbb{S} ^N}\frac{1}{|Sy-\sigma|^{\alpha}}\Psi(\varphi)(\sigma)d\sigma dy\\
				&=p\cdot 2^{-(p-1)(N-4)}J_{S}^{\frac{N-4}{2N}}(x) \int_{\R^N}\frac{1}{|Sx-Sy|^{N-4}}J_{S}(y) \int_{\mathbb{S} ^N}\frac{1}{|Sy-\sigma|^{\alpha}}\Psi(\varphi)(\sigma)d\sigma dy\\
				&=p\cdot 2^{-(p-1)(N-4)}J_{S}^{\frac{N-4}{2N}}(x) \int_{\mathbb{S}^N}\int_{\mathbb{S} ^N}\frac{1}{|Sx-\eta|^{N-4}}\frac{1}{|\eta-\sigma|^{\alpha}}\Psi(\varphi)(\sigma)d\sigma d\eta.
			\end{split}
		\end{equation}
		Similarly,
		\begin{equation}\label{ST2}
			\int_{\R^N}\frac{1}{|x-y|^{N-4}}T_{2}(\varphi)(y)dy=(p-1)2^{-(p-1)(N-4)}J_{S}^{\frac{N-4}{2N}}(x) \int_{\mathbb{S}^N}\int_{\mathbb{S} ^N}\frac{1}{|Sx-\eta|^{N-4}}\frac{1}{|\eta-\sigma|^{\alpha}}\Psi(\varphi)(\eta) d\eta d\sigma.
		\end{equation}
		Combing \eqref{ST1} with \eqref{ST2}, we have
		\begin{equation}\nonumber
			\begin{split}
				J_{S}^{\frac{4-N}{2N}}(x)\varphi(x)&=C^\frac{2(N+4-\alpha)}{N-4}_{N,\alpha}C_{N}\int_{\R^N}\frac{1}{|x-y|^{N-4}}T(\varphi)(y)dy\\
				&=C^\frac{2(N+4-\alpha)}{N-4}_{N,\alpha}C_{N}[ p\cdot 2^{-(p-1)(N-4)} \int_{\mathbb{S}^N}\int_{\mathbb{S} ^N}\frac{1}{|Sx-\eta|^{N-4}}\frac{1}{|\eta-\sigma|^{\alpha}}\Psi(\varphi)(\sigma)d\sigma d\eta\\
				&+(p-1)2^{-(p-1)(N-4)} \int_{\mathbb{S}^N}\int_{\mathbb{S} ^N}\frac{1}{|Sx-\eta|^{N-4}}\frac{1}{|\eta-\sigma|^{\alpha}}\Psi(\varphi)(\eta) d\eta d\sigma] \\
				&=C^\frac{2(N+4-\alpha)}{N-4}_{N,\alpha}C_{N}[pT_{\mathbb{S}^N, 1}\Psi(\varphi)(\xi)+(p-1)T_{\mathbb{S}^N,2}\Psi(\varphi)(\xi)]
				=C^\frac{2(N+4-\alpha)}{N-4}_{N,\alpha}C_{N}T_{\mathbb{S}^N}\Psi(\varphi)(\xi).
			\end{split}
		\end{equation}
		Therefore, equation \eqref{EqTS} holds.
	\end{proof}

	{\bf Proof of Theorem \ref{Non}.} At the beginning of the proof, by definition of $\Psi$, we need to prove that $\Psi(\varphi)\in \mathbb{Y}^{N+1}_{1}$. As Lemma \ref{SN}, we have $\Psi(\varphi)\in L^2(\mathbb{S}^N)$ and the orthogonal decomposition $L^2(\mathbb{S}^N)=\bigoplus\limits_{k=0}^{\infty}\mathscr {H}_{k}^{N+1}$, we get
	\begin{equation}\label{OD}
		\Psi(\varphi)(\xi)=\sum_{k=0}^{\infty}\sum_{i=1}^{\dim \mathscr {H}_{k}^{N+1}}\Psi(\varphi)_{k,i}Y_{k.i}(\xi),
	\end{equation} 
	where $\Psi(\varphi)_{k,i}=\int_{\mathbb{S}^N}\Psi(\varphi)(\xi)Y_{k.i}(\xi)d\xi$.
	Moreover $\Psi(\varphi)(\xi)$ satisfies \eqref{EqTS} and Lemma \ref{FH1}, we deduce that 
	\begin{equation}\label{OD2}
		\Psi(\varphi)_{k,i}=C^\frac{2(N+4-\alpha)}{N-4}_{N,\alpha}C_{N}2^{-(p-1)(N-4)}\lambda_{k}(N-4)\left[p\lambda_{k}(\alpha)+(p-1)\lambda_{0}(\alpha) \right]	\Psi(\varphi)_{k,i}.
	\end{equation}
	When $k=1$, an easy computation shows that  
	\begin{equation}\nonumber
		C^\frac{2(N+4-\alpha)}{N-4}_{N,\alpha}C_{N}2^{-(p-1)(N-4)}\lambda _{1}(N-4)\left[p\lambda_{1}(\alpha)+(p-1)\lambda_{0}(\alpha) \right]=1.
	\end{equation}
	Similar consideration apply to $k=0$ and $k\geq 2$, then 
	\begin{equation}\nonumber
		C^\frac{2(N+4-\alpha)}{N-4}_{N,\alpha}C_{N}2^{-(p-1)(N-4)}\lambda_{0}(N-4)\left[p\lambda_{1}(\alpha)+(p-1)\lambda_{0}(\alpha) \right]>1.
	\end{equation}
	\begin{equation}\nonumber
		C^\frac{2(N+4-\alpha)}{N-4}_{N,\alpha}C_{N}2^{-(p-1)(N-4)}\lambda_{k}(N-4)\left[p\lambda_{1}(\alpha)+(p-1)\lambda_{0}(\alpha) \right]<1.
	\end{equation}
	We conclude from \eqref{OD2} that
	$$
	\Psi(\varphi)_{k,i}=0,\sp \text{for} \sp k=0 \sp \text{and} \sp k\geq 2,
	$$
	so by \eqref{OD}, we have $\Psi(\varphi)\in \mathbb{Y}^{N+1}_{1}$.

	Since the dimension of $\mathbb{Y}^{N+1}_{1}$ is $N+1$, we have
	$$
	\Psi(\varphi)\in \left\lbrace \xi_{j}|\sp 1\leq j\leq N+1 \sp \right\rbrace . 
	$$
	From Lemma \ref{L1}, we know the map $\Psi$ is one to one map from the subspace $\text{span}\left\lbrace \varphi_{j},\sp 1\leq j\leq N+1 \sp \right\rbrace$ to the subspace $\mathscr {H}_{1}^{N+1}$, then the inverse map $\Phi:\mathscr {H}_{1}^{N+1}\mapsto \text{span}\left\lbrace \varphi_{j},\sp 1\leq j\leq N+1 \right\rbrace$, so
	$$
	\varphi\in \text{span}\left\lbrace \varphi_{j},\sp 1\leq j\leq N+1 \right\rbrace.
	$$
	$\hfill{} \Box$

	\section{A nonlocal Sobolev type inequality}
	In this section, we will prove Theorem \ref{INE}. Our proof is mainly inspired by reference \cite{DTYZ,LW}. A key element is the analysis of the eigenvalues of the following equation:
	\begin{equation}\label{EV}
		\Delta^{2}v+\left( |x|^{-\alpha}\ast |W|^{p}\right)W^{p-2} v=\mu\left[ \left( |x|^{-\alpha}\ast (W^{p-1}v)\right) W^{p-1}+\left( |x|^{-\alpha}\ast |W|^{p}\right)W^{p-2} v \right]. 
	\end{equation}
	The study of the eigenvalues of a linear operator is a classical topic. Here we refer to the work of Servadei and Valdinoci \cite{SV} for nonlocal case to give the
	definitions of eigenvalues of problem \eqref{EV} as follow
	
	\begin{Def}\label{EV1}
		The first eigenvalue of problem \eqref{EV} is positive and that can be defined as
		\begin{equation}\nonumber
			\mu_{1}:=\inf_{v\in\mathcal D^{2,2}(\R^N)\backslash\lbrace 0\rbrace}\frac{\int_{\mathbb{R}^{N}}|\Delta v|^2dx+\int_{\mathbb{R}^{N}}\Big(|x|^{-\alpha}\ast W^{p}\Big)W^{p-2}v^{2}dx}{\int_{\mathbb{R}^{N}}\Big(|x|^{-\alpha}\ast (W^{p-1}v)\Big)W^{p-1}v dx+\int_{\mathbb{R}^{N}}\Big(|x|^{-\alpha}\ast W^{p}\Big)W^{p-2}v^{2}dx}.
		\end{equation}
		Moreover, for any $k\in \mathbb{N^+}$ the eigenvalue can be characterized as follows
		\begin{equation}\nonumber
			\mu_{k+1}:=\inf_{v\in\mathbb P_{k+1}\backslash\lbrace 0\rbrace}\frac{\int_{\mathbb{R}^{N}}|\Delta v|^2dx+\int_{\mathbb{R}^{N}}\Big(|x|^{-\alpha}\ast W^{p}\Big)W^{p-2}v^{2}dx}{\int_{\mathbb{R}^{N}}\Big(|x|^{-\alpha}\ast (W^{p-1}v)\Big)W^{p-1}v dx+\int_{\mathbb{R}^{N}}\Big(|x|^{-\alpha}\ast W^{p}\Big)W^{p-2}v^{2}dx},
		\end{equation}
		where
		$$
		\mathbb P_{k+1}:=\Big\lbrace v\in\mathcal D^{2,2}(\R^N): \int_{\mathbb{R}^{N}}\Delta v \cdot \Delta e_{j}dx=0, \sp for \;all \sp j=1,...,k.\Big\rbrace,
		$$
		and $e_{j}$ is the corresponding eigenfunction to $\lambda_{j}$. 
	\end{Def}
	
	In definition \ref{EV1} we choose $v=W$, so it's easy to see 
	\begin{equation}\nonumber
		\mu_{1}\leq\frac{\int_{\mathbb{R}^{N}}|\Delta W|^2dx+\int_{\mathbb{R}^{N}}\Big(|x|^{-\alpha}\ast W_{\overline{y}_{n},\overline{\beta}_{n}}^{p}\Big)W_{\overline{y}_{n},\overline{\beta}_{n}}^{p} dx}{\int_{\mathbb{R}^{N}}\Big(|x|^{-\alpha}\ast W^{p}\Big)W^{p} dx+\int_{\mathbb{R}^{N}}\Big(|x|^{-\alpha}\ast W^{p}\Big)W^{p}dx}=1.
	\end{equation}
	By HLS inequality,  H\"{o}lder inequality and Sobolev inequality \eqref{SI}, we deduce 
	\begin{equation}\label{INE1}
		\begin{split}
			\int_{\mathbb{R}^{N}}\Big(|x|^{-\alpha}\ast (W^{p-1}v)\Big)W^{p-1}v dx
			&\leq C(N,\alpha)\left\| W^{p-1}v\right\| ^{2}_{L^{\frac{2N}{N-4}}}\\
			&\leq C(N,\alpha)\left\| W\right\| ^{2(p-1)}_{L^{\frac{2N}{N-4}}}\left\| v\right\| ^{2}_{L^{\frac{2N}{N-4}}}\\
			&=S_{2}\left\| v\right\| ^{2}_{L^{\frac{2N}{N-4}}}\leq\left\| v\right\| ^{2}_{D^{2,2}(\R^N)}
		\end{split}
	\end{equation}
	Similarly, we also have
	\begin{equation}\label{INE2}
	\aligned
	\int_{\mathbb{R}^{N}}\Big(|x|^{-\alpha}\ast W^{p}\Big)W^{p-2}v^{2}dx
	\leq\left\| v\right\| ^{2}_{D^{2,2}(\R^N)}
	\endaligned
	\end{equation}
	where $C(N,\alpha)$ is given as in Proposition \ref{pro1.1}. Moreover, 
	equality is achieved if and only if $v=cW$ with $c\in \R$, which implies $\mu_{1}\geq 1$. Then $\mu_{1}=1$ and the corresponding eigenfunction is $cW$ with $c\in \R$. Therefore, because of Theorem \ref{Non} we get
	\begin{lem}\label{ES}
		Let $\mu_{k}$, $k=1,2,\cdots,$ denote the eigenvalues of \eqref{EV} in increasing order. Then $\mu_{1}=1$ is simple and the corresponding eigenfunction is $ cU$ with $c\in \R$ and  $\mu_{2}=\mu_{3}=\dots=\mu_{N+2}=p<\mu_{N+3}\leq \dots$ with the corresponding $N+1$-dimension eigenfunction space spanned by 
		$$
		\left\lbrace \frac{(N-4)}{2}W(x)+x\cdot\nabla W(x), \frac{\partial W}{\partial x_1},\cdots \frac{\partial W}{\partial x_N}\right\rbrace.
		$$
	\end{lem}
	
	The lemma below encompasses the essential components for establishing Theorem \ref{INE}, focusing on the analysis of behavior in proximity to $\mathcal{M}=\left\lbrace cW_{y,\beta}\; :\:\;c\in \R,\,y\in R^N,\,\beta\in\R^{+} 
	\right\rbrace$  which is $(N+2)$-dimensional manifold. 
	\begin{lem}\label{ES1}
		For any sequence $\left\lbrace u_n\right\rbrace \subset \mathcal{D}^{2,2}(\R^N)\setminus\mathcal{M}$ satisfying
		$$
		\inf_{n}\|u_n\|_{D^{2,2}(\R^N)}>0.\sp \dist\left(u_n, \mathcal{M} \right) \rightarrow 0,
		$$
		then 
		\begin{equation}\label{INF}
			\liminf_{n\rightarrow\infty}\frac{\int_{\R^N}|\Delta u_n|^{2}- S^{*}\left( \int_{\R^N}(|x|^{-\alpha}*u_n^{p})u_n^{p}dx\right) ^{\frac{1}{p}}}{\dist\left(u_n, \mathcal{M} \right)^2}\geq 2(\mu_{N+3}-p),
		\end{equation}
		and 
		\begin{equation}\label{SUP}
			\limsup_{n\rightarrow\infty}\frac{\int_{\R^N}|\Delta u_n|^{2}- S^{*}\left( \int_{\R^N}(|x|^{-\alpha}*u_n^{p})u_n^{p}dx\right) ^{\frac{1}{p}}}{\dist\left(u_n, \mathcal{M} \right)^2}\leq 1,
		\end{equation}
		where $\dist\left(u_n, \mathcal{M} \right)=\inf\limits_{c\in \R,y\in\R^N,\beta\in\R^+} \left\| u_n-cW_{y, \beta}\right\|$ and $\mu_{N+3}>p$ is given as in Lemma \ref{ES}.
	\end{lem}
	\begin{proof}
		Since $\mathcal{M}$ is an $(N+2)$-dimensional manifold embedded in $\mathcal{D}^{2,2}(\R^N)$, that is
		$$
		(c,y,\beta)\in \R\times\R^{N}\times\R^{+}\rightarrow cW_{z,\beta}\in \mathcal{D}^{2,2}(\R^N).
		$$
		Denoted $\dist\left(u_n, \mathcal{M} \right)$ by $b_{n}$, we know $b_{n} \rightarrow 0$ as $n\rightarrow\infty$. It is clear that the infimum above is attained at a point $(\overline{c}_{n},\overline{y}_{n},\overline{\beta}_{n})\in \R\times\R^{N}\times\R^{+}$ with $\overline{c}_{n}\neq 0$ for each $u_n\in \mathcal{D}^{2,2}(\R^N)$. As Lemma \ref{ES} and the fact $\mathcal{M}\setminus \left\lbrace 0\right\rbrace $ is a smooth manifold , the tangential space at $(\overline{c}_{n},\overline{y}_{n},\overline{\beta}_{n})$ is spanned by
		$$
		T_{\overline{c}_{n},W_{\overline{y}_{n},\overline{\beta}_{n}}}\mathcal{M}=\text{span}\left\lbrace W_{\overline{y}_{n},\overline{\beta}_{n}}, \partial_{\beta} W_{\overline{y}_{n},\overline{\beta}_{n}}, \nabla_{y}W_{\overline{y}_{n},\overline{\beta}_{n}}\right\rbrace. 
		$$
		We have that $(u_n-\overline{c}_{n}W_{\overline{y}_{n},\overline{\beta}_{n}})$ is perpendicular to $T_{\overline{c}_{n},W_{\overline{y}_{n},\overline{\beta}_{n}}}\mathcal{M}$.
		Let $u_n=\overline{c}_{n}W_{\overline{y}_{n},\overline{\beta}_{n}}+b_{n}k_{n}$, then $k_n$ is perpendicular to $T_{\overline{c}_{n},W_{\overline{y}_{n},\overline{\beta}_{n}}}\mathcal{M}$, $\left\| k_n\right\|_{{D}^{2,2}(\R^N)}=1$ and
		$$
		\left\| k_n\right\|_{{D}^{2,2}(\R^N)}^2=b_{n}^2+\overline{c}_{n}^2
		\left\| W_{\overline{y}_{n},\overline{\beta}_{n}}\right\|^{2}_{{D}^{2,2}(\R^N)}=b_{n}^2+\overline{c}_{n}^2
		\left\| W\right\|^{2}_{{D}^{2,2}(\R^N)}.$$
		By the fact that $k_n$ is perpendicular to $W_{\overline{y}_{n},\overline{\beta}_{n}}$, we have $\int_{\R^N}\Delta k_n\cdot\Delta W_{\overline{y}_{n},\overline{\beta}_{n}}=0$. Hence we get
		\begin{equation}\nonumber
			\begin{split}
				\int_{\mathbb{R}^{N}}\Big(|x|^{-\alpha}\ast (W_{\overline{y}_{n},\overline{\beta}_{n}}^{p-1}k_n)\Big)W_{\overline{y}_{n},\overline{\beta}_{n}}^{p} dx
				&=\int_{\mathbb{R}^{N}}\Big(|x|^{-\alpha}\ast W_{\overline{y}_{n},\overline{\beta}_{n}}^{p}\Big)W_{\overline{y}_{n},\overline{\beta}_{n}}^{p-1}k_n dx\\
				&=\int_{\R^N}\Delta k_n\cdot\Delta W_{\overline{y}_{n},\overline{\beta}_{n}}=0.
			\end{split}
		\end{equation}
		Therefore,
		\begin{equation}
			\begin{split}
				\int_{\mathbb{R}^{N}}\Big(|x|^{-\alpha}\ast u_n^{p}\Big)u_n^{p} dx=&\overline{c}_{n}^{2p}	\int_{\mathbb{R}^{N}}\Big(|x|^{-\alpha}\ast W_{\overline{y}_{n},\overline{\beta}_{n}}^{p}\Big)W_{\overline{y}_{n},\overline{\beta}_{n}}^{p} dx\\
				&+p(p-1)\overline{c}_{n}^{2(p-1)}b_{n}^{2}\int_{\mathbb{R}^{N}}\Big(|x|^{-\alpha}\ast W_{\overline{y}_{n},\overline{\beta}_{n}}^{p}\Big)W_{\overline{y}_{n},\overline{\beta}_{n}}^{p-2}k_{n}^2 dx\\
				&+p^2\overline{c}_{n}^{2(p-1)}b_{n}^{2}\int_{\mathbb{R}^{N}}\Big(|x|^{-\alpha}\ast (W_{\overline{y}_{n},\overline{\beta}_{n}}^{p-1}k_n)\Big)(W_{\overline{y}_{n},\overline{\beta}_{n}}^{p-1}k_n) dx+o(b_{n}^{2}).
			\end{split}
		\end{equation}
		
		Next we will prove \eqref{INF}. By the definition of $\mu_{N+3}$ in Lemma \ref{ES}, we can get
		\begin{equation*}\begin{split}
				\int_{\mathbb{R}^{N}}&|\Delta k_n|^2dx+\int_{\mathbb{R}^{N}}\Big(|x|^{-\alpha}\ast W_{\overline{y}_{n},\overline{\beta}_{n}}^{p}\Big)W_{\overline{y}_{n},\overline{\beta}_{n}}^{p-2}k_n^{2}dx\\
				&\geq \mu_{N+3}\left[\int_{\mathbb{R}^{N}}\Big(|x|^{-\alpha}\ast (W_{\overline{y}_{n},\overline{\beta}_{n}}^{p-1}k_n)\Big)W_{\overline{y}_{n},\overline{\beta}_{n}}^{p-1}k_n dx+\int_{\mathbb{R}^{N}}\Big(|x|^{-\alpha}\ast W_{\overline{y}_{n},\overline{\beta}_{n}}^{p}\Big)W_{\overline{y}_{n},\overline{\beta}_{n}}^{p-2}k_n^{2}dx \right],
			\end{split}
		\end{equation*}
		hence 
		\begin{equation}
			\aligned
			\mu_{N+3}\int_{\mathbb{R}^{N}}&\Big(|x|^{-\alpha}\ast (W_{\overline{y}_{n},\overline{\beta}_{n}}^{p-1}k_n)\Big)W_{\overline{y}_{n},\overline{\beta}_{n}}^{p-1}k_n dx\\
			&+(\mu_{N+3}-1)\int_{\mathbb{R}^{N}}\Big(|x|^{-\alpha}\ast W_{\overline{y}_{n},\overline{\beta}_{n}}^{p}\Big)W_{\overline{y}_{n},\overline{\beta}_{n}}^{p-2}k_n^{2}dx\leq 1
			\endaligned
		\end{equation}
		By \eqref{INE1} and \eqref{INE2}, we have
		\begin{equation}
			\int_{\mathbb{R}^{N}}\Big(|x|^{-\alpha}\ast (W_{\overline{y}_{n},\overline{\beta}_{n}}^{p-1}k_n)\Big)(W_{\overline{y}_{n},\overline{\beta}_{n}}^{p-1}k_n) dx\leq \left\| k_n\right\|_{{D}^{2,2}(\R^N)^2}=1
		\end{equation}
		and 
		\begin{equation}
			\int_{\mathbb{R}^{N}}\Big(|x|^{-\alpha}\ast W_{\overline{y}_{n},\overline{\beta}_{n}}^{p}\Big)W_{\overline{y}_{n},\overline{\beta}_{n}}^{p-2}v_{n}^2 dx\leq \left\| k_n\right\|_{{D}^{2,2}(\R^N)^2}=1.
		\end{equation}
		Combining these, we can deduce that
		\begin{equation}
			\begin{split}
				&\int_{\mathbb{R}^{N}}\Big(|x|^{-\alpha}\ast u_n^{p}\Big)u_n^{p} dx=\overline{c}_{n}^{2p}	\int_{\mathbb{R}^{N}}\Big(|x|^{-\alpha}\ast W_{\overline{y}_{n},\overline{\beta}_{n}}^{p}\Big)W_{\overline{y}_{n},\overline{\beta}_{n}}^{p} dx+o(b_{n}^{2})\\
				&+p(p-\mu_{N+3})\overline{c}_{n}^{2(p-1)}b_{n}^{2}\Big[\int_{\mathbb{R}^{N}}\Big(|x|^{-\alpha}\ast W_{\overline{y}_{n},\overline{\beta}_{n}}^{p}\Big)W_{\overline{y}_{n},\overline{\beta}_{n}}^{p-2}k_{n}^2 dx
				+\int_{\mathbb{R}^{N}}\Big(|x|^{-\alpha}\ast (W_{\overline{y}_{n},\overline{\beta}_{n}}^{p-1}k_n)\Big)(W_{\overline{y}_{n},\overline{\beta}_{n}}^{p-1}k_n) dx\Big]\\
				&+{p}\overline{c}_{n}^{2(p-1)}b_{n}^{2}\Big[(\mu_{N+3}-1)\int_{\mathbb{R}^{N}}\Big[\Big(|x|^{-\alpha}\ast W_{\overline{y}_{n},\overline{\beta}_{n}}^{p}\Big)W_{\overline{y}_{n},\overline{\beta}_{n}}^{p-2}k_{n}^2 dx
				+\mu_{N+3}\int_{\mathbb{R}^{N}}\Big(|x|^{-\alpha}\ast (W_{\overline{y}_{n},\overline{\beta}_{n}}^{p-1}k_n)\Big)(W_{\overline{y}_{n},\overline{\beta}_{n}}^{p-1}k_n) dx\Big]\\
				&\leq \overline{c}_{n}^{2p}\left\| W\right\|_{{D}^{2,2}(\R^N)}^{2}+p\overline{c}_{n}^{2(p-1)}b_{n}^{2}[2(p-\mu_{N+3})+1]+o(b_{n}^{2}),
			\end{split}
		\end{equation}
		and so
		\begin{equation}
			\begin{split}
				\left( \int_{\mathbb{R}^{N}}\Big(|x|^{-\alpha}\ast u_n^{p}\Big)u_n^{p} dx\right) ^{\frac{1}{p}}&\leq \overline{c}_{n}^2\left( \left\| W\right\|_{{D}^{2,2}(\R^N)}^2+2\overline{c}_{n}^{-2}b_{n}^{2}[2(p-\mu_{N+3})+1]\right)^{\frac{1}{p}} +o(b_{n}^{2})\\
				&=\overline{c}_{n}^2\left\| W\right\|_{{D}^{2,2}(\R^N)}^{\frac{2}{p}}+b_{n}^{2}[2(p-\mu_{N+3})+1] \left\| W\right\|_{{D}^{2,2}(\R^N)}^{\frac{2}{p}-2}+o(b_{n}^{2})
			\end{split}
		\end{equation}
		Therefore,
		\begin{equation}
			\begin{split}
				&\int_{\mathbb{R}^{N}}|\Delta u_n|^2dx-S^*\left( \int_{\mathbb{R}^{N}}\Big(|x|^{-\alpha}\ast u_n^{p}\Big)u_n^{p} dx\right) ^{\frac{1}{p}}\\
				&\geq b_{n}^2+\overline{c}_{n}^2
				\left\| W\right\|^{2}_{{D}^{2,2}(\R^N)}-S^{*}\left[\overline{c}_{n}^2\left\| W\right\|_{{D}^{2,2}(\R^N)}^{\frac{2}{p}}+b_{n}^{2}[2(p-\mu_{N+3})+1] \left\| W\right\|_{{D}^{2,2}(\R^N)}^{\frac{2}{p}-2}+o(b_{n}^{2}) \right] \\
				&=b_{n}^{2}\left( 1-S^{*}[2(p-\mu_{N+3})+1] \left\| W\right\|_{{D}^{2,2}(\R^N)}^{\frac{2}{p}-2} \right)+\overline{c}_{n}^2
				\left\| W\right\|^{2}_{{D}^{2,2}(\R^N)}\left( 1-S^{*}\left\| W\right\|_{{D}^{2,2}(\R^N)}^{\frac{2}{p}-2}\right)+o(b_{n}^{2})\\
				&\geq 2(\mu_{N+3}-2) b_{n}^{2}+o(b_{n}^{2}).
			\end{split}
		\end{equation}
		where the equality holds because $S^{*}=\left\| W\right\|_{{D}^{2,2}(\R^N)}^{2-\frac{2}{p}}$ which is the combination of $\|W\|_{{D}^{2,2}(\R^N)}^2=\|W\|_{p}^{p}$ and $\|W\|_{{D}^{2,2}(\R^N)}^2=S^{*}\|W\|_{p}$. Here $\|W\|_{p}=\int_{\R^N}\Big(|x|^{-\alpha}\ast W^{p}\Big)W^{p} dx$. Then as $n$ small enough, inequality \eqref{INF} follows exactly.
		
		In order to prove the second result, we can easily get
		\begin{equation}
			\begin{split}
				\left( \int_{\mathbb{R}^{N}}\Big(|x|^{-\alpha}\ast u_n^{p}\Big)u_n^{p} dx\right) ^{\frac{1}{p}}&\geq \overline{c}_{n}^{2}	\left( \int_{\mathbb{R}^{N}}\Big(|x|^{-\alpha}\ast W_{\overline{y}_{n},\overline{\beta}_{n}}^{p}\Big)W_{\overline{y}_{n},\overline{\beta}_{n}}^{p} dx+o(b_{n}^{2})\right)^{\frac{1}{p}}\\
				&=\overline{c}_{n}^2
				\left\| W\right\|_{{D}^{2,2}(\R^N)}^{\frac{2}{p}}+o(b_{n}^{2}).
			\end{split}
		\end{equation}
		Then using the same equality as above, we have
		\begin{equation}
			\begin{split}
				&\int_{\mathbb{R}^{N}}|\Delta u_n|^2dx-S^*\left( \int_{\mathbb{R}^{N}}\Big(|x|^{-\alpha}\ast u_n^{p}\Big)u_n^{p} dx\right) ^{\frac{1}{p}}\\
				&\leq b_{n}^2+\overline{c}_{n}^2
				\left\| W\right\|^{2}_{{D}^{2,2}(\R^N)}-S^*\left( \overline{c}_{n}^2
				\left\|W\right\|_{{D}^{2,2}(\R^N)}^{\frac{2}{p}}+o(b_{n}^{2})\right)\\ 
				&=b_{n}^2+o(b_{n}^{2})+\overline{c}_{n}^2
				\left\| W\right\|^{2}_{{D}^{2,2}(\R^N)}\left(1-S^*\left\| W_{\overline{y}_{n},\overline{\beta}_{n}}\right\|_{{D}^{2,2}(\R^N)}^{\frac{2}{p}} \right) \\
				&=(1+o(1))b_{n}^2,
			\end{split}
		\end{equation}
		as $b_{n}$ is small enough. So \eqref{SUP} can be obtained immediately.
	\end{proof}
	
	{\bf Proof of Theorem \ref{INE}.} We prove by contradiction. Suppose that \eqref{INERE} is not true, then there exists a sequence $u_n\in {D}^{2,2}(\R^N)\setminus\mathcal{M}$ such that
	\begin{equation}\label{LM1}
		\frac{\int_{\R^N}|\Delta u_n|^{2}- S^{*}\left( \int_{\R^N}(|x|^{-\alpha}*u_n^{p})u_n^{p}dx\right) ^{\frac{1}{p}}}{\dist\left(u_n, \mathcal{M} \right)^2}\rightarrow +\infty,\; \text{as}\; n\rightarrow \infty,
	\end{equation}
	or
	\begin{equation}\label{LM2}
		\frac{\int_{\R^N}|\Delta u_n|^{2}- S^{*}\left( \int_{\R^N}(|x|^{-\alpha}*u_n^{p})u_n^{p}dx\right) ^{\frac{1}{p}}}{\dist\left(u_n, \mathcal{M} \right)^2}\rightarrow 
		0,\; \text{as}\; n\rightarrow \infty.
	\end{equation}
	Let $\|u_n\|_{{D}^{2,2}(\R^N)}=1$ by homogeneity. Since $$\dist\left(u_n, \mathcal{M} \right):=\inf\limits_{c\in \R,y\in\R^N,\beta\in\R^+} \left\| u_n-cW_{y,\beta}\right\|\leq \|u_n\|_{{D}^{2,2}(\R^N)}=1,$$ we can select a subsequence such that $\dist\left(u_n, \mathcal{M} \right)\rightarrow K\in [0,1]$. Observe \eqref{LM1} and we find it holds when $K=0$ which is contradict to Lemma \ref{ES1}. Consider \eqref{LM2} and we also deduce a contradiction if $K=0$ by Lemma \ref{ES1}. Consequently, the only possible case is that $K>0$ and \eqref{LM2} hold which means 
	$$
	\dist\left(u_n, \mathcal{M} \right)\rightarrow K>0,\sp \int_{\R^N}|\Delta u_n|^{2}- S^{*}\left( \int_{\R^N}(|x|^{-\alpha}*u_n^{p})u_n^{p}dx\right) ^{\frac{1}{p}}\rightarrow 0,
	$$
	as $n\rightarrow\infty$. Therefore
	\begin{equation}
		\left( \int_{\R^N}(|x|^{-\alpha}*u_n^{p})u_n^{p}dx\right) ^{\frac{1}{p}}\rightarrow \frac{1}{S^*},\sp \|u_n\|_{{D}^{2,2}(\R^N)}=1.
	\end{equation}
	According to P. L. Lions’ concentration and compactness principle, we have two sequences of $\beta_n$ and $y_n$ such that
	$$
	\beta_n^{\frac{N-4}{2}}u_{n}(\beta_n(x-y_n))\rightarrow W_{0}\in{{D}^{2,2}(\R^N)}\;\text{as}\;n\rightarrow \infty. 
	$$
	It implies that 
	$$
	\dist\left(u_n, \mathcal{M} \right)=\dist\left(\beta_n^{\frac{N-4}{2}}u_{n}(\beta_n(x-y_n)), \mathcal{M} \right)\rightarrow 0\;\text{as}\;n\rightarrow \infty,
	$$
	which is impossible.
	$\hfill{} \Box$

	\section{Construction of multi-bubble solutions}
	In this section, we carry out the finite-dimensional reduction argument and use local Poho$\check{z}$aev inequality to construct the infinitly many solutions of \eqref{CFL}. Define
	$$\aligned
	H_{s}=\Big\{&u\in D^{2,2}(\mathbb{R}^N),u(x_{1},-x_{2},x'')=u(x_{1},x_{2},x''),\\
	&\hspace{4mm}u(r\cos\theta,r\sin\theta,x'')=u\Big(r\cos(\theta+\frac{2j\pi}{m}),r\sin(\theta+\frac{2j\pi}{m}),x''\Big)\Big\}
	\endaligned$$
	and let
	$$
	z_{j}=\Big(\overline{r}\cos\frac{2(j-1)\pi}{m},\overline{r}\sin\frac{2(j-1)\pi}{m},\overline{x}''\Big), \ j=1,\cdot\cdot\cdot,m,
	$$
	where $\overline{x}''$ is a vector in $\mathbb{R}^{N-2}$. By the weak symmetry of $V(x)$, we have $V (z_j ) = V (\overline{r}, \overline{x}'')$, $j = 1,\cdots, m$.
	In order to construct solutions concentrated at $(r_0, x_0'')$, we use $W_{z_j,\beta}$ (see \eqref{REL}) as an approximate solution. Let $\delta> 0$ be a small
	constant, such that $r^4V (r, x'') > 0$ if $|(r, x'')-(r_0, x_0'' )|\leq 10\delta$. Let $\xi(x) = \xi(|x'|, x'')$ be a
	smooth function satisfying $\xi= 1$ if $|(r, x'')-(r_0, x_0'' )|\leq\delta$, $\xi= 0$ if $|(r, x'')-(r_0, x_0'' )| \geq2\delta$,
	and $0\leq \xi\leq 1$. Denote
	$$
	Z_{z_j,\beta}(x)=\xi W_{z_j,\beta}(x), \
	Z_{\overline{r},\overline{x}'',\beta}(x)=\sum_{j=1}^{m}Z_{z_j,\beta}(x), \
	Z_{\overline{r},\overline{x}'',\beta}^{\ast}(x)=\sum_{j=1}^{m}W_{z_j,\beta}(x),
	$$
	and
	$$
	Z_{j,1}=\frac{\partial Z_{z_j,\beta}}{\partial \beta}, Z_{j,2}=\frac{\partial Z_{z_j,\beta}}{\partial \overline{r}},
	Z_{j,k}=\frac{\partial Z_{z_j,\beta}}{\partial \overline{x}_{k}''},
	\ \mbox{for }k=3,\cdot\cdot\cdot,N, \ j=1,\cdot\cdot\cdot,m.
	$$
	
	We assume that $m > 0$ is a large integer, $\beta\in[L_{0}m^{\frac{N-4}{N-8}},L_{1}m^{\frac{N-4}{N-8}}]$
	for some constants $L_1 > L_0 > 0$ and
	\begin{equation}\label{uxi}
		|(\overline{r}, \overline{x}'')-(r_0, x_0'')|\leq \vartheta,
	\end{equation}
	where $\vartheta> 0$ is a small constant. Let
	$$
	\|u\|_{\ast}=\sup_{x\in\mathbb{R}^N}\Big(\sum_{j=1}^{m}
	\frac{\beta^{\frac{N-4}{2}}}{(1+\beta|x-z_{j}|)^{\frac{N-4}{2}+\tau}}\Big)^{-1}|u(x)|,
	$$
	and
	$$
	\|h\|_{\ast\ast}=\sup_{x\in\mathbb{R}^N}\Big(\sum_{j=1}^{m}
	\frac{\beta^{\frac{N+4}{2}}}{(1+\beta|x-z_{j}|)^{\frac{N+4}{2}+\tau}}\Big)^{-1}|h(x)|,
	$$
	where $\tau=\frac{N-8}{N-4}$.
	
	Consider
	\begin{equation}\label{c1}
		\left\{\begin{array}{l}
			\displaystyle {\Delta}^2 \phi+V(r,x'')\phi
			-p (|x|^{-\alpha}\ast Z_{\overline{r},\overline{x}'',\beta}^{p-1}\phi)Z_{\overline{r},\overline{x}'',\beta}^{p-1}-(p-1)(|x|^{-\alpha}\ast |Z_{\overline{r},\overline{x}'',\beta}|^p)Z_{\overline{r},\overline{x}'',\beta}^{p-2}\phi
			\\
			\displaystyle \hspace{10.14mm}=h+
			\sum_{l=1}^{N}c_{l}\sum_{j=1}^{m}\Big[p\Big(|x|^{-\alpha}\ast (Z_{z_j,\beta}^{p-1}
			Z_{j,l})\Big)Z_{z_j,\beta}^{p-1}+(p-1)\Big(|x|^{-\alpha}\ast |Z_{z_j,\beta}|^{p}\Big)Z_{z_j,\beta}^{p-2}Z_{j,l}\Big]\hspace{4.14mm}\mbox{in}\hspace{1.14mm} \mathbb{R}^N,\\
			\displaystyle \phi\in H_{s}, \ \ \sum_{j=1}^{m}\int_{\mathbb{R}^{N}}\Big[p\Big(|x|^{-\alpha}\ast (Z_{z_j,\beta}^{p-1}
			Z_{j,l})\Big)Z_{z_j,\beta}^{p-1}\phi+(p-1)\Big(|x|^{-\alpha}\ast |Z_{z_j,\beta}|^{p}\Big)Z_{z_j,\beta}^{p-2}Z_{j,l}\phi\Big] dx=0,\\
			\displaystyle \hspace{10.14mm}l=1,2,\cdots,N,
		\end{array}
		\right.
	\end{equation}
	for some real numbers $c_{l}$.

	\begin{lem}\label{C2}
		There exist $m_0 > 0$ and a constant $C > 0$ independent of $m$, such that for all
		$m \geq m_0$ and all $h\in L^{\infty}(\R^N)$, problem \eqref{c1} has a unique solution $\phi \equiv L_m(h)$. Besides,
		\begin{equation}\label{c13}
			\|L_m(h)\|_{\ast}\leq C\|h\|_{\ast\ast},\quad |c_l|\leq \frac{C}{\beta^{n_{l}}}\|h\|_{\ast\ast}.
		\end{equation}
	\end{lem}
	
	This result is a direct consequence of Proposition 4.1 in\cite{DFM} with the help of following Lemma \ref{C1}.

	\begin{lem}\label{C1}
		Suppose that $\phi_{m}$ solves \eqref{c1} for $h = h_m$. If $\|h_{m}\|_{\ast\ast}\to 0$ as $m\to \infty$, then $\|\phi_{m}\|_{\ast}\to 0$.
	\end{lem}
	\begin{proof}
		We argue by contradiction. Suppose that there exist $m\rightarrow+\infty$, $\overline{r}_m\rightarrow r_{0}$, $\overline{x}_{m}''\rightarrow x_0''$, $\beta_{m}\in[L_{0}m^{\frac{N-4}{N-8}},L_{1}m^{\frac{N-4}{N-8}}]$ and $\phi_{m}$ solving \eqref{c1} for $h=h_{m}$, $\beta=\beta_{m}$, $\overline{r}=\overline{r}_{m}$, $\overline{x}''=\overline{x}_{m}''$,
		with $\|h_{m}\|_{\ast\ast}\rightarrow0$ and $\|\phi_{m}\|_{\ast}\geq c>0$. We may assume that $\|\phi_{m}\|_{\ast}=1$.
		
		By \eqref{c1}, we have
		\begin{equation}\label{c2}
			\aligned
			&	|\phi_m(x)|\\&\leq
			C\int_{\mathbb{R}^N}\frac{1}{|y-x|^{N-4}} \Big(|y|^{-\alpha}\ast (Z_{\overline{r},\overline{y}'',\beta}^{p-1}
			|\phi_m|)\Big)Z_{\overline{r},\overline{x}'',\beta}^{p-1}(y)dy\\
			&\hspace{4mm}+C	\int_{\mathbb{R}^N}\frac{1}{|y-x|^{N-4}}\Big(|y|^{-\alpha}\ast |Z_{\overline{r},\overline{x}'',\beta}|^{p}\Big)Z_{\overline{r},\overline{x}'',\beta}^{p-2}|\phi_m(y)|dy
			\\
			&\hspace{4mm}+C\sum_{l=1}^{N}|c_{l}|\Bigg[\Big|\sum_{j=1}^{m}\int_{\mathbb{R}^N}\frac{1}{|y-x|^{N-4}}\Big(|y|^{-\alpha}\ast (Z_{z_j,\beta}^{p-1}
			Z_{j,l})\Big)Z_{z_j,\beta}^{p-1}(y)dy\Big|\\
			&\hspace{4mm}+\Big|\sum_{j=1}^{m}\int_{\mathbb{R}^N}\frac{1}{|y-x|^{N-4}}\Big(|y|^{-\alpha}\ast |Z_{z_j,\beta}|^{p}\Big)Z_{z_j,\beta}^{p-2}Z_{j,l}(y)dy\Big|\Bigg]+C\int_{\mathbb{R}^N}\frac{1}{|y-x|^{N-4}}|h_m(y)|dy.
			\endaligned
		\end{equation}

		Notice the fact that
		$$
		Z_{\overline{r},\overline{x}'',\beta}\leq C\sum_{j=1}^{m}\frac{\beta^{\frac{N-4}{2}}}{(1+\beta|y-z_{j}|)^{N-4}}.
		$$
		Define
		$$
		\Omega_{j}=\left\{x=(x',x'')\in\mathbb{R}^{2}\times\mathbb{R}^{N-2}:\Big\langle\frac{x'}{|x'|},
		\frac{z_{j}'}{|z_{j}'|}\Big\rangle\geq\cos\frac{\pi}{m}\right\}, \ j=1,\cdot\cdot\cdot,m.
		$$
		For $y\in\Omega_{1}$, we have $|y- z_{j}|\geq|y- z_{1}|$. According to Lemma \ref{B2}, there exists a $0<\tau_1<\frac{N-4}{2}$ that
		\begin{equation}\label{5141}		
			\aligned
			\sum_{j=2}^{m}\frac{1}{(1+\beta|y- z_{j}|)^{N-4}}&\leq \frac{1}{(1+\beta|y- z_{1}|)^{\frac{N-4}{2}}}\sum_{j=2}^{m}\frac{1}{(1+\beta|y- z_{j}|)^{\frac{N-4}{2}}}\\
			&\leq\frac{C}{(1+\beta|y- z_{1}|)^{N-4-\tau_1}}\sum_{j=2}^{m}\frac{1}{(\beta|z_{1}-z_{j}|)^{\tau_1}}\\
			&\leq\frac{C}{(1+\beta|y- z_{1}|)^{N-4-\tau_1}},
			\endaligned
		\end{equation}
		where
		$$
		\sum_{j=2}^{m}\frac{1}{(\beta|z_{1}-z_{j}|)^{\tau_1}}\leq C,
		$$
		derives from the proof of Lemma B.3 in \cite{WY1}. So we have
		$$
		Z_{\overline{r},\overline{x}'',\beta}^{p-1}\leq C\frac{\beta^{\frac{N+4-\alpha}{2}}}{(1+\beta|y-z_{1}|)^{N+4-\alpha-\tau_1\frac{N+4-\alpha}{N-4}}}.
		$$
        Then using Lemma \ref{B4}, we get
		\begin{equation}\label{51411}
		\aligned
		&|y|^{-\alpha}\ast (Z_{\overline{r},\overline{y}'',\beta}^{p-1}
		|\phi_m|)\\
		&\leq C\|\phi_m\|_{\ast}\sum_{j=1}^{m}\int_{\Omega_{j}}\frac{1}{|x-y|^{\alpha}}\Big(\frac{\beta^{\frac{N+4-\alpha}{2}}}{(1+\beta|y-z_{j}|)^{N+4-\alpha-\tau_1\frac{N+4-\alpha}{N-4}}}\sum_{i=1}^{m}
		\frac{\beta^{\frac{N-4}{2}}}{(1+\beta|y-z_{i}|)^{\frac{N-4}{2}+\tau}}\Big)dy \\
		&\leq C\|\phi_m\|_{\ast}\sum_{j=1}^{m}\int_{\Omega_{j}}\frac{1}{|x-y|^{\alpha}}\frac{\beta^{N-\frac{\alpha}{2}}}{(1+\beta|y-z_{j}|)^{\frac{3N+4}{2}-\alpha-\tau_1\frac{2N-\alpha}{N-4}+\tau}}\\
		&\leq C\|\phi_m\|_{\ast}\sum_{j=1}^{m}\frac{\beta^{\frac{\alpha}{2}}}{(1+\beta|y-z_{j}|)^{\min\{\alpha, \frac{N+4}{2}\}}},
		\endaligned
		\end{equation}
		where in $\Omega_{j}$, we have
		$$\aligned
		&\frac{\beta^{\frac{N+4-\alpha}{2}}}{(1+\beta|y-z_{j}|)^{N+4-\alpha-\tau_1\frac{N+4-\alpha}{N-4}}}\sum_{i=1}^{m}
		\frac{\beta^{\frac{N-4}{2}}}{(1+\beta|y-z_{i}|)^{\frac{N-4}{2}+\tau}}\\
		&\leq \frac{\beta^{N-\frac{\alpha}{2}}}{(1+\beta|y-z_{j}|)^{\frac{3N+4}{2}-\alpha-\tau_1\frac{N+4-\alpha}{N-4}+\tau}}
		+\sum_{i=2}^{m}\frac{C}{(\beta|z_{i}-z_{j}|)^{\tau_1}}\frac{\beta^{N-\frac{\alpha}{2}}}{(1+\beta|y-z_{j}|)^{\frac{3N+4}{2}-\alpha-\tau_1\frac{N+4-\alpha}{N-4}-\tau_1+\tau}}\\
		&\leq\frac{C\beta^{N-\frac{\alpha}{2}}}{(1+\beta|y-z_{j}|)^{\frac{3N+4}{2}-\alpha-\tau_1\frac{2N-\alpha}{N-4}+\tau}}.
		\endaligned$$
		Here $\tau_1$ could be small enough such that $\eta=-\tau_1\frac{2N-\alpha}{N-4}+\tau>0$. Combined this, we can obtain the estimate of first term in the right side of  \eqref{c2},
		\begin{equation}\label{c02}
			\aligned
			&\int_{\mathbb{R}^N}\frac{1}{|y-x|^{N-4}} \Big(|y|^{-\alpha}\ast (Z_{\overline{r},\overline{y}'',\beta}^{p-1}
			|\phi_m|)\Big)Z_{\overline{r},\overline{x}'',\beta}^{p-1}(y)dy\\
			&\leq C\|\phi_m\|_{\ast}\beta^{\frac{N-4}{2}}\int_{\mathbb{R}^N}\frac{1}{|y-\beta x|^{N-4}}
			\sum_{j=1}^{m}\frac{1}{(1+|y-\beta z_{j}|)^{\alpha}}
			\sum_{j=1}^{m}\frac{1}{(1+|y-\beta z_{j}|)^{N+4-\alpha-\tau_1\frac{N+4-\alpha}{N-4}}} dy\\
			&\leq C\|\phi_m\|_{\ast}\beta^{\frac{N-4}{2}}\sum_{j=1}^{m}\int_{\Omega_{j}}\frac{1}{|y-\beta x|^{N-4}}
			\frac{1}{(1+|y-\beta z_{j}|)^{N+4-\tau_1\frac{2N-\alpha}{N-4}}} dy
			\\
			&\leq C\|\phi_m\|_{\ast}\beta^{\frac{N-4}{2}}\sum_{j=1}^{m}
						\frac{1}{(1+\beta|y- z_{j}|)^{\frac{N-4}{2}+\tilde{\tau}+\theta}} dy.
			\endaligned
		\end{equation}
		By setting $\tilde{\tau}=\tau_1\frac{2N-\alpha}{N-4}$ and  $\theta=\frac{N+4}{2}-2\tilde{\tau}>0$, the last inequality comes from Lemma \ref{B3}. And we also have
		$$\aligned
		&\int_{\mathbb{R}^N}\frac{1}{|y-x|^{N-4}} \Big(|y|^{-\alpha}\ast (Z_{\overline{r},\overline{y}'',\beta}^{p-1}
		|\phi_m|)\Big)Z_{\overline{r},\overline{x}'',\beta}^{p-1}(y)dy\\
		&\leq C\|\phi_m\|_{\ast}\beta^{\frac{N-4}{2}}\int_{\mathbb{R}^N}\frac{1}{|y-\beta x|^{N-4}}
		\sum_{j=1}^{m}\frac{1}{(1+|y-\beta z_{j}|)^{\frac{N+4}{2}}}
		\sum_{j=1}^{m}\frac{1}{(1+|y-\beta z_{j}|)^{N+4-\alpha-\tau_1\frac{N+4-\alpha}{N-4}}} dy\\
		&\leq C\|\phi_m\|_{\ast}\beta^{\frac{N-4}{2}}\sum_{j=1}^{m}\int_{\Omega_{j}}\frac{1}{|y-\beta x|^{N-4}}
		\frac{1}{(1+|y-\beta z_{j}|)^{\frac{3(N+4)}{2}-\alpha-\tau_1\frac{2N-\alpha}{N-4}}} dy
		\\
		&\leq C\|\phi_m\|_{\ast}\beta^{\frac{N-4}{2}}\sum_{j=1}^{m}
		\frac{1}{(1+\beta|y- z_{j}|)^{\frac{N-4}{2}+\tilde{\tau}+\theta}} dy.
		\endaligned
		$$
		Here we set $\theta=N+4-\alpha-2\tilde{\tau}>0$. Similarly we obtain
		\begin{equation}\label{5143}
			\int_{\mathbb{R}^N}\frac{1}{|y-x|^{N-4}}\Big(|y|^{-\alpha}\ast |Z_{\overline{r},\overline{x}'',\beta}|^{p}\Big)Z_{\overline{r},\overline{x}'',\beta}^{p-2}|\phi_m(y)| dy\leq C \|\phi_m\|_{\ast}\beta^{\frac{N-4}{2}}\sum_{j=1}^{m}
			\frac{1}{(1+\beta|x-z_{j}|)^{\frac{N-4}{2}+\tilde{\tau}+\theta}}.
		\end{equation}
		
		Using the Lemma \ref{B3} and \ref{P0}, the third term can estimate as follows,
		\begin{equation}\label{5144}
			\aligned &\left|\sum_{j=1}^{m}\int_{\mathbb{R}^N}\frac{1}{|y-x|^{N-4}}\Big(|y|^{-\alpha}\ast (Z_{z_j,\beta}^{p-1}
			Z_{j,l})\Big)Z_{z_j,\beta}^{p-1}(y)dy\right|\\
			&\leq C \beta^{n_{l}}\sum_{j=1}^{m}\int_{\mathbb{R}^N}\frac{1}{|y-x|^{N-4}}			\Big(|y|^{-\alpha}\ast {W^p_{z_j,\beta}})\Big)W^{p-1}_{z_j,\beta}(y)dy\\
			&\leq C \beta^{n_{l}}\sum_{j=1}^{m}\int_{\mathbb{R}^N}\frac{1}{|y-x|^{N-4}}\frac{\beta^\frac{\alpha}{2}}{(1+\beta|y-z_{j}|)^{\alpha}}	\frac{\beta^{\frac{N+4-\alpha}{2}}}{(1+\beta|y-z_{j}|)^{N+4-\alpha}}dy\\
			&\leq C \beta^{n_{l}}\sum_{j=1}^{m}\int_{\mathbb{R}^N}\frac{1}{|y-x|^{N-4}}			\frac{\beta^{\frac{N+4}{2}}}{(1+\beta|y-z_{j}|)^{N+4}}dy\\
			&\leq C \beta^{\frac{N-4}{2}+n_{l}}\sum_{j=1}^{m}			\frac{1}{(1+\beta|x-z_{j}|)^{\frac{N-4}{2}+\tau}},
			\endaligned
		\end{equation}
		where $n_{1}=-1$, $n_{l}=1,l=2,\cdot\cdot\cdot,N$. Similarly we have
		\begin{equation}\label{5145}
			\left|\sum_{j=1}^{m}\int_{\mathbb{R}^N}\frac{1}{|y-x|^{N-4}}\Big(|y|^{-\alpha}\ast |Z_{z_j,\beta}|^{p}\Big)Z_{z_j,\beta}^{p-2}Z_{j,l}(y)dy\right|
			\leq C \beta^{\frac{N-4}{2}+n_{l}}\sum_{j=1}^{m}
			\frac{1}{(1+\beta|x-z_{j}|)^{\frac{N-4}{2}+\tau}}.
		\end{equation}
		
		For the last term,  by Lemma \ref{B3}, we know
		\begin{equation}\label{5146}
			\aligned &\int_{\mathbb{R}^N}\frac{1}{|y-x|^{N-4}}|h_m(y)|dy\\
			&\leq C \|h_m\|_{\ast\ast}\beta^{\frac{N+4}{2}}\int_{\mathbb{R}^N}\frac{1}{|y-x|^{N-4}}\sum_{j=1}^{m}
			\frac{1}{(1+\beta|y-z_{j}|)^{\frac{N+4}{2}+\tau}}dy\\
			&\leq C \|h_m\|_{\ast\ast}\beta^{\frac{N-4}{2}}\int_{\mathbb{R}^N}\frac{1}{|y-\beta x|^{N-4}}\sum_{j=1}^{m}
			\frac{1}{(1+|y-\beta z_{j}|)^{\frac{N+4}{2}+\tau}}dy\\
			&\leq C \|h_m\|_{\ast\ast}\beta^{\frac{N-4}{2}}\sum_{j=1}^{m}
			\frac{1}{(1+\beta|x-z_{j}|)^{\frac{N-4}{2}+\tau}}.
			\endaligned
		\end{equation}

		Next we are going to estimate $c_l$, $l = 1, 2, \cdot\cdot\cdot,N$. Multiplying \eqref{c1} by $Z_{1,t} (t = 1, 2, \cdot\cdot\cdot,N)$ and integrating, we see that $c_l$ satisfies
		\begin{equation}\label{c3}
			\aligned
			\sum_{l=1}^{N}&\sum_{j=1}^{m}	\Big\langle p\Big(|x|^{-\alpha}\ast (Z_{z_j,\beta}^{p-1}
			Z_{j,l})\Big)Z_{z_j,\beta}^{p-1}+(p-1)\Big(|x|^{-\alpha}\ast |Z_{z_j,\beta}|^{p}\Big)Z_{z_j,\beta}^{p-2}Z_{j,l}, Z_{1,t}\Big\rangle c_{l}\\
			=&
			\Big\langle {\Delta}^2 \phi_m
			+ V(r,x'')\phi_m-p\Big(|x|^{-\alpha}\ast (Z_{\overline{r},\overline{x}'',\beta}^{p-1}
			\phi_m)\Big)Z_{\overline{r},\overline{x}'',\beta}^{p-1}-(p-1)\Big(|x|^{-\alpha}\ast |Z_{\overline{r},\overline{x}'',\beta}|^{p}\Big)Z_{\overline{r},\overline{x}'',\beta}^{p-2}\phi_m
			, Z_{1,t}\Big\rangle\\
			&-\langle h_m, Z_{1,t}\rangle.
			\endaligned
		\end{equation}
		
		It follows from Lemma 2.1 in \cite{GLN} that
		\begin{equation}\label{c4}
			\aligned
			\langle {\Delta}^2 \phi_m, Z_{1,t}\rangle= O\Big(\frac{\beta^{n_{t}}\|\phi_m\|_{\ast}}{\beta^{2+\varepsilon}}\Big),\hspace{4.4mm}
			|\langle V(r,x'')\phi_m, Z_{1,t}\rangle|
			\leq O(\frac{\beta^{n_{t}}\|\phi_m\|_{\ast}}{\beta^{3+\varepsilon}})
			\endaligned
		\end{equation}
		and
		\begin{equation}\label{c5}
			\aligned
			|\langle h_m, Z_{1,t}\rangle|
			\leq C\beta^{n_{t}}\|h_m\|_{\ast\ast},
			\endaligned
		\end{equation}
		where $\varepsilon>0$ is a small constant.
		Direct calculations show that
		$$\aligned
		&\Big|\Big(|x|^{-\alpha}\ast (Z_{\overline{r},\overline{x}'',\beta}^{p-1}
		\phi_m)\Big)\Big|\\
		\leq& C\|\phi_m\|_{\ast}\int_{\mathbb{R}^{N}}\frac{1}{|y|^{\alpha}}\sum_{j=1}^{m}\frac{\xi(x-y)\beta^{\frac{N+4-\alpha}{2}}}{(1+\beta|x-y-z_{j}|)^{N+4-\alpha-\tau_1\frac{N+4-\alpha}{N-4}}}\sum_{j=1}^{m}
		\frac{\beta^{\frac{N-4}{2}}}{(1+\beta|x-y-z_{j}|)^{\frac{N-4}{2}+\tau}}dy\\
		\leq& C\|\phi_m\|_{\ast}\Big(\sum_{i=1}^{m}\int_{\mathbb{R}^{N}}\frac{1}{|y|^{\alpha}}\frac{\xi(x-y)\beta^{N-\frac{\alpha}{2}}}{(1+\beta|x-y-z_{i}|)^{\frac{3N+4}{2}-\alpha-\tau_1\frac{N+4-\alpha}{N-4}+\tau}}dy\\
		&+
		\sum_{j\neq i}\int_{\mathbb{R}^{N}}\frac{1}{|y|^{\alpha}}\frac{\xi(x-y)\beta^{\frac{N+4-\alpha}{2}}}{(1+\beta|x-y-z_{i}|)^{N+4-\alpha-\tau_1\frac{N+4-\alpha}{N-4}}}
		\frac{\beta^{\frac{N-4}{2}}}{(1+\beta|x-y-z_{j}|)^{\frac{N-4}{2}+\tau}}dy\Big)\\
		=&O(\frac{m^{2}\|\phi_m\|_{\ast}}{\beta^{\frac{\alpha}{2}-1}}),
		\endaligned$$
		where
		$$\aligned
		&\int_{\mathbb{R}^{N}}\frac{1}{|y|^{\alpha}}\frac{\xi(x-y)\beta^{N-\frac{\alpha}{2}}}{(1+\beta|x-y-z_{i}|)^{\frac{3N+4}{2}-\alpha-\tau_1\frac{N+4-\alpha}{N-4}+\tau}}dy\\
		\leq&\int_{B_{2\delta}(x-(r_0, x_0'' ))}\frac{1}{|y|^{\alpha}}\frac{\beta^{N-\frac{\alpha}{2}}}{(1+\beta|x-z_{i}-y|)^{\frac{3N+4}{2}-\alpha-\tau_1\frac{N+4-\alpha}{N-4}+\tau}}dy\\
		\leq&\int_{B_{2\delta}(x-(r_0, x_0'' ))}\frac{\beta^{N-\frac{\alpha}{2}}}{(\beta|x-y-z_{i}|)^{N-1}}dy\\
		=&O(\frac{1}{\beta^{\frac{\alpha}{2}-1}}),
		\endaligned$$
		and
		$$\aligned
		&\int_{\mathbb{R}^{N}}\frac{1}{|y|^{\alpha}}\frac{\xi(x-y)\beta^{\frac{N+4-\alpha}{2}}}{(1+\beta|x-y-z_{i}|)^{N+4-\alpha-\tau_1\frac{N+4-\alpha}{N-4}}}
		\frac{\beta^{\frac{N-4}{2}}}{(1+\beta|x-y-z_{j}|)^{\frac{N-4}{2}+\tau}}dy\\		\leq&\frac{C}{(\beta|z_{i}-z_{j}|)^{\frac{\tau}{2}}}\int_{\mathbb{R}^{N}}\frac{1}{|y|^{\alpha}}
		\Big(\frac{\xi(x-y)\beta^{N-\frac{\alpha}{2}}}{(1+\beta|x-y-z_{i}|)^{\frac{3N+4}{2}-\alpha-\tau_1\frac{N+4-\alpha}{N-4}+\frac{\tau}{2}}}
		\\&\hspace{5mm}+\frac{\xi(x-y)\beta^{N-\frac{\alpha}{2}}}{(1+\beta|x-y-z_{j}|)^{\frac{3N+4}{2}-\alpha-\tau_1\frac{N+4-\alpha}{N-4}+\frac{\tau}{2}}}\Big)dy\\
		=&O(\frac{1}{\beta^{\frac{\alpha}{2}-1}}),  \ \  j\neq i.
		\endaligned$$
		So we obtain
		$$\aligned
		&\int_{\mathbb{R}^N}\Big(|x|^{-\alpha}\ast (Z_{\overline{r},\overline{x}'',\beta}^{p-1}
		\phi_m)\Big)Z_{\overline{r},\overline{y}'',\beta}^{p-1} Z_{1,t}dx\\
		&\leq C\|\phi_m\|_{\ast}\|Z_{\overline{r},\overline{x}'',\beta}\|_{\ast}\frac{m^{2}}{\beta^{\frac{\alpha}{2}-1}}
		\int_{\mathbb{R}^N}\sum_{j=1}^{m}
		\frac{\beta^{\frac{N+4-\alpha}{2}}}{(1+\beta|x-z_{j}|)^{N+4-\alpha-\tau_1\frac{N+4-\alpha}{N-4}}}
		\frac{\xi\beta^{\frac{N-4}{2}+n_{t}}}{(1+\beta|x-z_{1}|)^{N-4}}dx\\
		&\leq C\|\phi_m\|_{\ast}
		\frac{m^{2}}{\beta^{\frac{\alpha}{2}-1}}\frac{m\beta^{n_{t}}}{\beta^{\frac{\alpha}{2}}}\int_{\mathbb{R}^N}
		\frac{1}{(1+|x-\beta z_{1}|)^{2N-\alpha-\tau_1\frac{2N-\alpha}{N-4}}}dx\\
		&=O(\frac{\beta^{n_{t}}\|\phi_m\|_{\ast}}{\beta^{2+\varepsilon}}),
		\endaligned$$
		where $\alpha>6-\frac{12}{N-4}$ and $\varepsilon>0$. 
		Similarly, we also have
		$$
		\Big\langle \Big(|x|^{-\alpha}\ast (|Z_{\overline{r},\overline{x}'',\beta}|^{p})\Big)Z_{\overline{r},\overline{x}'',\beta}^{p-2}\phi_m, Z_{1,t}\Big\rangle= O\Big(\frac{\beta^{n_{t}}\|\phi_m\|_{\ast}}{\beta^{2+\varepsilon}}\Big).
		$$
		Consequently,
		\begin{equation}\label{c6}
			\Big\langle 
			p\Big(|x|^{-\alpha}\ast (Z_{\overline{r},\overline{x}'',\beta}^{p-1}
			\phi_m)\Big)Z_{\overline{r},\overline{x}'',\beta}^{p-1}+(p-1)\Big(|x|^{-\alpha}\ast (|Z_{\overline{r},\overline{x}'',\beta}|^{p})\Big)Z_{\overline{r},\overline{x}'',\beta}^{p-2}\phi_m, Z_{1,t}\Big\rangle= O(\frac{\beta^{n_{t}}\|\phi_m\|_{\ast}}{\beta^{2+\varepsilon}}).
		\end{equation}
		
		Combining \eqref{c4}-\eqref{c6}, we have
		\begin{equation}\label{c7}
			\aligned
			\Big\langle&{\Delta}^2 \phi_m+V(r,x'')\phi_m
			-p\Big(|x|^{-\alpha}\ast (Z_{\overline{r},\overline{x}'',\beta}^{p-1}
			\phi_m)\Big)Z_{\overline{r},\overline{x}'',\beta}^{p-1}-(p-1)\Big(|x|^{-\alpha}\ast (|Z_{\overline{r},\overline{x}'',\beta}|^{p})\Big)Z_{\overline{r},\overline{x}'',\beta}^{p-2}\phi_m
			, Z_{1,t}\Big\rangle\\&\hspace{5mm}-\langle h_m, Z_{1,t}\rangle
		    =O\Big(\frac{\beta^{n_{t}}\|\phi_m\|_{\ast}}{\beta^{2+\varepsilon}}+\beta^{n_{t}}\|h_m\|_{\ast\ast}\Big).
			\endaligned
		\end{equation}
		We can easily check by the orthogonality that
		\begin{equation}\label{c8}
			\sum_{j=1}^{m}\Big\langle \Big(|x|^{-\alpha}\ast |Z_{z_j,\beta}|^{p}\Big)Z_{z_j,\beta}^{p-2}Z_{j,l}, Z_{1,t}\Big\rangle=(\overline{c}+o(1))\delta_{tl}\beta^{n_{l}}\beta^{n_{t}},
		\end{equation}
		and
		\begin{equation}\label{c81}
			\sum_{j=1}^{m}\Big\langle \Big(|x|^{-\alpha}\ast (Z_{z_j,\beta}^{p-1}
			Z_{j,l})\Big)Z_{z_j,\beta}^{p-1}, Z_{1,t}\Big\rangle=(\overline{c}'+o(1))\delta_{tl}\beta^{n_{l}}\beta^{n_{t}},
		\end{equation}
		for some constant $\overline{c} > 0$ and $\overline{c}' > 0$.
		Then we substitute \eqref{c7}, \eqref{c8} and \eqref{c81} into \eqref{c3} and find that
		\begin{equation}\label{c9}
			c_{l}=\frac{1}{\beta^{n_{l}}}(o(\|\phi_m\|_{\ast})+O(\|h_m\|_{\ast\ast})).
		\end{equation}
		Thus,
		\begin{equation}\label{c10}
			\|\phi_m\|_{\ast}\leq o(1)+\|h_m\|_{\ast\ast}+\frac{\sum_{j=1}^{N}\frac{1}{(1+\beta|x-z_{j}|)^{\frac{N-4}{2}+\tilde{\tau}+\theta}}}
			{\sum_{j=1}^{N}\frac{1}{(1+\beta|x-z_{j}|)^{\frac{N-4}{2}+\tau}}}.
		\end{equation}
		Since $\|\phi_m\|_{\ast}= 1$, from \eqref{c10} we can deduce that there is $R > 0$ such that
		\begin{equation}\label{c11}
			\|\beta^{-\frac{N-4}{2}}\phi_m\|_{L^{\infty}(B_{\frac{R}{\beta}}(z_{j}))}\geq a>0,
		\end{equation}
		for some $j$. However, $\widetilde{\phi}_m (x)=\beta^{-\frac{N-4}{2}}\phi_m(\beta(x-z_{j}))$ 
		converges to  $v\in D^{2,2}(\mathbb{R}^{N})$ which is a solution of the equation
		\begin{equation}\label{ib5}
			{\Delta}^2 v
			=p\Big(\int_{\mathbb{R}^{N}}\frac{W_{0,\beta}^{p-1}(y)v(y)}{|x-y|^{\alpha}}dy\Big)W_{0,\beta}^{p-1}+(p-1)\Big(\int_{\mathbb{R}^{N}}\frac{|W_{0,\beta}(y)|^{p}}{|x-y|^{\alpha}}dy\Big)W_{0,\beta}^{p-2}v\hspace{4.14mm}\mbox{in}\hspace{1.14mm} \mathbb{R}^N,
		\end{equation}
		for some $\beta\in[\beta_{1}, \beta_{1}]$. Since $v$ is perpendicular to the kernel of \eqref{ib5}, we conclude that $v = 0$ by the non-degeneracy of $W_{0,1}$. This contradicts to \eqref{c11}.
	\end{proof}
	
	Next we consider:
	\begin{equation}\label{c14}
		\left\{\begin{array}{l}
			\displaystyle {\Delta}^2 (Z_{\overline{r},\overline{x}'',\beta}+\phi)+V(r,x'')(Z_{\overline{r},\overline{x}'',\beta}+\phi)
			-\Big(|x|^{-\alpha}\ast (|(Z_{\overline{r},\overline{x}'',\beta}+\phi)|^{p})\Big)(Z_{\overline{r},\overline{x}'',\beta}+\phi)^{p-1}\\
			\displaystyle \hspace{10.14mm}=\sum_{l=1}^{N}c_{l}\sum_{j=1}^{m}\Big[p\Big(|x|^{-\alpha}\ast (Z_{z_j,\beta}^{p-1}
			Z_{j,l})\Big)Z_{z_j,\beta}^{p-1}+(p-1)\Big(|x|^{-\alpha}\ast |Z_{z_j,\beta}|^{p}\Big)Z_{z_j,\beta}^{p-2}Z_{j,l}\Big]\hspace{4.14mm}\mbox{in}\hspace{1.14mm} \mathbb{R}^N,\\
			\displaystyle \phi\in H_{s}, \ \ \sum_{j=1}^{m}\int_{\mathbb{R}^{N}}\Big[ p\Big(|x|^{-\alpha}\ast |Z_{z_j,\beta}^{p-1}Z_{j,l}|\Big)Z_{z_j,\beta}^{p-1}\phi+(p-1)\Big(|x|^{-\alpha}\ast |Z_{z_j,\beta}|^{p}\Big)Z_{z_j,\beta}^{p-2}Z_{j,l}\phi\Big] dx=0,\\
			\displaystyle \hspace{10.14mm}l=1,2,\cdots,N.
		\end{array}
		\right.
	\end{equation}
	We can rewrite \eqref{c14} as
	\begin{equation}\label{c16}
		\aligned
		{\Delta}^2 &\phi+V(r,x'')\phi
		-p\Big(|x|^{-\alpha}\ast (Z_{\overline{r},\overline{x}'',\beta}^{p-1}\phi)\Big)Z_{\overline{r},\overline{x}'',\beta}^{p-1}-(p-1)\Big(|x|^{-\alpha}\ast (|Z_{\overline{r},\overline{x}'',\beta}|^{p})\Big)Z_{\overline{r},\overline{x}'',\beta}^{p-2}\phi\\
		& =N(\phi)+l_{m}+\sum_{l=1}^{N}c_{l}\sum_{j=1}^{m}\Big[p\Big(|x|^{-\alpha}\ast (Z_{z_j,\beta}^{p-1}
		Z_{j,l})\Big)Z_{z_j,\beta}^{p-1}+(p-1)\Big(|x|^{-\alpha}\ast |Z_{z_j,\beta}|^{p}\Big)Z_{z_j,\beta}^{p-2}Z_{j,l}\Big]\hspace{4.14mm}\mbox{in}\hspace{1.14mm} \mathbb{R}^N,
		\endaligned
	\end{equation}
	where
	$$\aligned
	N(\phi)=&\Big(|x|^{-\alpha}\ast (|(Z_{\overline{r},\overline{x}'',\beta}+\phi)|^{p})\Big)(Z_{\overline{r},\overline{x}'',\beta}+\phi)^{p-1}-\Big(|x|^{-\alpha}\ast (|Z_{\overline{r},\overline{x}'',\beta}|^{p})\Big)Z_{\overline{r},\overline{x}'',\beta}^{p-1}\\
	&-p\Big(|x|^{-\alpha}\ast (Z_{\overline{r},\overline{x}'',\beta}^{p-1}\phi)\Big)Z_{\overline{r},\overline{x}'',\beta}^{p-1}-(p-1)\Big(|x|^{-\alpha}\ast (|Z_{\overline{r},\overline{x}'',\beta}|^{p})\Big)Z_{\overline{r},\overline{x}'',\beta}^{p-2}\phi
	\endaligned$$
	and
	$$\aligned
	l_{m}=&\Big(|x|^{-\alpha}\ast (|Z_{\overline{r},\overline{x}'',\beta}|^{p})\Big)Z_{\overline{r},\overline{x}'',\beta}^{p-1}
	-\sum_{j=1}^{m}\xi\Big(|x|^{-\alpha}\ast |W_{z_j,\beta}|^{p}\Big)W_{z_j,\beta}^{p-1}
	-V(r,x'')Z_{\overline{r},\overline{x}'',\beta}\\
	&-\Delta^2 Z_{\overline{r},\overline{x}'',\beta}^{\ast}\xi-2\Delta Z_{\overline{r},\overline{x}'',\beta}^{\ast}\Delta\xi- Z_{\overline{r},\overline{x}'',\beta}^{\ast}\Delta^2\xi+4\nabla(-\Delta \xi)\nabla  Z_{\overline{r},\overline{x}'',\beta}^{\ast}\\&+4\nabla\xi\nabla (-\Delta Z_{\overline{r},\overline{x}'',\beta}^{\ast})-4\nabla^2\xi\nabla^2 Z_{\overline{r},\overline{x}'',\beta}^{\ast}.
	\endaligned$$
	
	In order to apply the Contraction Mapping Theorem to prove that \eqref{c16} is uniquely solvable, we need to estimate $N(\phi)$ and $l_m$ respectively.

	\begin{lem}\label{C4}
		There is a constant $C> 0$, such that
		\begin{equation}\label{c17}
			\|N(\phi)\|_{\ast\ast}\leq C\|\phi\|_{\ast}^{2}.
		\end{equation}
	\end{lem}
	\begin{proof}
		We find that $N(\phi)$ satisfies the following inequality
		$$\aligned
		|N(\phi)|\leq &C\left|\Big(|x|^{-\alpha}\ast (Z_{\overline{r},\overline{x}'',\beta}^{p-1}\phi)\Big)Z_{\overline{r},\overline{x}'',\beta}^{p-2}\phi+\Big(|x|^{-\alpha}\ast (Z_{\overline{r},\overline{x}'',\beta}^{p-2}\phi^2)\Big)Z_{\overline{r},\overline{x}'',\beta}^{p-1}\right|\\
		&+C\left|\Big(|x|^{-\alpha}\ast (Z_{\overline{r},\overline{x}'',\beta}^{p-2}\phi^2)\Big)Z_{\overline{r},\overline{x}'',\beta}^{p-2}\phi+\Big(|x|^{-\alpha}\ast \phi^p\Big)\phi^{p-1}\right|.
		\endaligned$$
		Using \eqref{51411} and Lemma \ref{B4}, we have
		$$\aligned
		&\Big(|x|^{-\alpha}\ast (Z_{\overline{r},\overline{x}'',\beta}^{p-1}\phi)\Big)Z_{\overline{r},\overline{x}'',\beta}^{p-2}|\phi|\\
		\leq &C \|\phi\|^2_{\ast}\beta^{\frac{N+4}{2}}\sum_{j=1}^{m}\frac{1}{(1+\beta|x-z_{j}|)^{\min\{\alpha,\frac{N+4}{2}\}}}\Big(\sum_{j=1}^{m}
		\frac{1}{(1+\beta|x-z_{j}|)^{\frac{N-4}{2}+\tau}}\Big)^{p-1}\\
		\leq &C \|\phi\|_{\ast}^{2}
		\beta^{\frac{N+4}{2}}
		\sum_{j=1}^{m}\frac{1}{(1+\beta|x-z_{j}|)^{\min\{\alpha,\frac{N+4}{2}\}}}\sum_{j=1}^{m}
		\frac{1}{(1+\beta|x-z_{j}|)^{\frac{N+4-\alpha}{2}+\tau}}\Big(\sum_{j=1}^{m}\frac{1}{(1+\beta|x-z_{j}|)^{\tau}}\Big)^{p-2}\\
		\leq &C \|\phi\|_{\ast}^{2}
		\beta^{\frac{N+4}{2}}\sum_{j=1}^{m}
		\frac{1}{(1+\beta|x-z_{j}|)^{\frac{N+4}{2}+\tau}}.
		\endaligned$$
		By the similar argument in \eqref{51411} and Lemma \ref{B4} we have
		$$\aligned
		&\Big(|x|^{-\alpha}\ast (Z_{\overline{r},\overline{x}'',\beta}^{p-2}\phi^2)\Big)Z_{\overline{r},\overline{x}'',\beta}^{p-1}\\
		\leq &C \|\phi\|^2_{\ast}\sum_{j=1}^{m}\int_{\Omega_{j}}\frac{1}{|x-y|^{\alpha}}\frac{\beta^{N-\frac{\alpha}{2}}}{(1+\beta|y-z_{j}|)^{N+4-\alpha-\tau_1\frac{2N-\alpha}{N-4}+\tau}}dy\Big(\sum_{j=1}^{m}\frac{\beta^{\frac{N-4}{2}}}{(1+\beta|x-z_{j}|)^{\frac{N-4}{2}+\tau}}\Big)^{p-1}\\
		\leq &C \|\phi\|^2_{\ast}\beta^{\frac{N+4}{2}}\sum_{j=1}^{m}\frac{1}{(1+\beta|x-z_{j}|)^{\min\{\alpha,4\}+\tau}}\sum_{j=1}^{m}
		\frac{1}{(1+\beta|x-z_{j}|)^{\frac{N+4-\alpha}{2}+\tau}}\Big(\sum_{j=1}^{m}\frac{1}{(1+\beta|x-z_{j}|)^{\tau}}\Big)^{p-2}\\
		\leq &C \|\phi\|_{\ast}^{2}
		\beta^{\frac{N+4}{2}}\sum_{j=1}^{m}
		\frac{1}{(1+\beta|x-z_{j}|)^{\frac{N+4}{2}+\tau}}.
		\endaligned$$
		Similarly, we also have
		$$
		\Big(|x|^{-\alpha}\ast (Z_{\overline{r},\overline{x}'',\beta}^{p-2}\phi^2)\Big)Z_{\overline{r},\overline{x}'',\beta}^{p-2}\phi
		\leq C \|\phi\|_{\ast}^{2}
		\beta^{\frac{N+4}{2}}\sum_{j=1}^{m}
		\frac{1}{(1+\beta|x-z_{j}|)^{\frac{N+4}{2}+\tau}}.
		$$
		For the last term, we also have
		$$
		\aligned
		|x|^{-\alpha}\ast |\phi|^{p}&\leq C\|\phi\|^{p}_{\ast}\Big(|x|^{-\alpha}\ast\Big(\sum_{j=1}^{m}\frac{\beta^{\frac{N-4}{2}}}{(1+\beta|x-z_{j}|)^{\frac{N-4}{2}+\tau}}\Big)^p\Big)\\
		&\leq C\|\phi\|^{p}_{\ast}\bigg(|x|^{-\alpha}\ast\Big(\sum_{j=1}^{m}
		\frac{\beta^{N-\frac{\alpha}{2}}}{(1+\beta|x-z_{j}|)^{\frac{2N-\alpha}{2}+\tau}}\Big(\sum_{j=1}^{m}\frac{1}{(1+\beta|x-z_{j}|)^{\tau}}\Big)^{p-1}\Big) \bigg)\\
		&\leq C\|\phi\|^{p}_{\ast}\sum_{j=1}^{m}\Big(|x|^{-\alpha}\ast\frac{\beta^{N-\frac{\alpha}{2}}}{(1+\beta|x-z_{j}|)^{N-\frac{\alpha}{2}+\tau}} \Big)\\
		&\leq C\|\phi\|^{p}_{\ast}\sum_{j=1}^{m}\frac{\beta^{\frac{\alpha}{2}}}{(1+\beta|x-z_{j}|)^{\frac{\alpha}{2}}},
		\endaligned
		$$
		and we find that
		$$\aligned
		&\Big(|x|^{-\alpha}\ast |\phi|^{p}\Big)|\phi|^{p-1}\\
		\leq &C \|\phi\|^{2p-1}_{\ast}\beta^{\frac{N+4}{2}}\sum_{j=1}^{m}\frac{1}{(1+\beta|x-z_{j}|)^{\frac{\alpha}{2}}}\sum_{j=1}^{m}
		\frac{1}{(1+\beta|x-z_{j}|)^{\frac{N+4-\alpha}{2}+\tau}}\Big(\sum_{j=1}^{m}\frac{1}{(1+\beta|x-z_{j}|)^{\tau}}\Big)^{p-2}\\
		\leq &C \|\phi\|_{\ast}^{2p-1}
		\beta^{\frac{N+4}{2}}\sum_{j=1}^{m}
		\frac{1}{(1+\beta|x-z_{j}|)^{\frac{N+4}{2}+\tau}}.
		\endaligned$$
		All the above estimates show that 
		$$
		\|N(\phi)\|_{\ast\ast}\leq C\|\phi\|_{\ast}^{\min\{2p-1,2\}}\leq C\|\phi\|_{\ast}^{2}.
		$$
	\end{proof}

	\begin{lem}\label{C5}
		There is a small constant $\varepsilon> 0$, such that
		\begin{equation}\label{c18}
			\|l_{m}\|_{\ast\ast}\leq C(\frac{1}{\beta})^{2+\varepsilon}.
		\end{equation}
	\end{lem}
	\begin{proof}
		Observe that
		$$\aligned
		l_{m}=&\bigg[\Big(|x|^{-\alpha}\ast (|Z_{\overline{r},\overline{x}'',\beta}|^{p})\Big)Z_{\overline{r},\overline{x}'',\beta}^{p-1}
		-\sum_{j=1}^{m}\xi\Big(|x|^{-\alpha}\ast |W_{z_j,\beta}|^{p}\Big)W_{z_j,\beta}^{p-1}\bigg]
		-V(r,x'')Z_{\overline{r},\overline{x}'',\beta}
		+\Big(-\Delta^2 Z_{\overline{r},\overline{x}'',\beta}^{\ast}\xi\\
		&-2\Delta Z_{\overline{r},\overline{x}'',\beta}^{\ast}\Delta\xi- Z_{\overline{r},\overline{x}'',\beta}^{\ast}\Delta^2\xi+4\nabla(-\Delta \xi)\nabla  Z_{\overline{r},\overline{x}'',\beta}^{\ast}+4\nabla\xi\nabla (-\Delta Z_{\overline{r},\overline{x}'',\beta}^{\ast})-4\nabla^2\xi\nabla^2 Z_{\overline{r},\overline{x}'',\beta}^{\ast}\Big)\\
		:=&J_1+J_2+J_3.
		\endaligned$$
		We assume that $|x-z_{j}|\geq|x-z_{1}|, \ \forall x \in\Omega_{1}$. For the first term  $J_1$, we have
		$$\aligned
		\left|J_1\right|\leq& C\Big(|x|^{-\alpha}\ast |W_{z_1,\beta}|^{p}\Big)W_{z_1,\beta}^{p-2}\sum_{j=2}^{m}W_{z_j,\beta}+C\Big(|x|^{-\alpha}\ast (W_{z_1,\beta}^{p-1}\sum_{j=2}^{m}W_{z_j,\beta})\Big)W_{z_1,\beta}^{p-1}\\
		&+C\Big(|x|^{-\alpha}\ast \Big(\sum_{j=2}^{m}W_{z_j,\beta}\Big)^{p}\Big)\Big(\sum_{j=2}^{m}W_{z_j,\beta}\Big)^{p-1}:=J_{11}+J_{12}+J_{13}.
		\endaligned$$
		Then by Lemma \ref{B2}, Lemma \ref{P0} and taking $0 < \gamma \leq \min\{8,N-4\}$,  we obtain that for any $x\in\Omega_{1}$ and $j>1$
		$$\aligned
		\Big(|x|^{-\alpha}\ast |W_{z_1,\beta}|^{p}\Big)W_{z_1,\beta}^{p-2}\sum_{j=2}^{m}W_{z_j,\beta}\leq& C \frac{\beta^{\frac{\alpha}{2}}}{(1+\beta^2|x-z_{1}|^2)^{\frac{\alpha}{2}}}\frac{\beta^{\frac{8-\alpha}{2}}}{(1+\beta^2|x-z_{1}|^2)^{\frac{8-\alpha}{2}}}\sum_{j=2}^{m}\frac{\beta^{\frac{N-4}{2}}}{(1+\beta^2|x-z_{j}|^2)^{\frac{N-4}{2}}}\\
		\leq&\frac{\beta^{4}}{(1+\beta|x-z_{1}|)^{8}}\sum_{j=2}^{m}\frac{\beta^{\frac{N-4}{2}}}{(1+\beta|x-z_{j}|)^{N-4}}\\
		\leq&C\frac{\beta^{\frac{N+4}{2}}}{(1+\beta|x-z_{1}|)^{\frac{N+4}{2}+\tau}}\sum_{j=2}^{m}\frac{1}{(\beta|z_{1}-z_{j}|)^{\gamma}}\\
		\leq&C\frac{\beta^{\frac{N+4}{2}}}{(1+\beta|x-z_{1}|)^{\frac{N+4}{2}+\tau}}\Big(\frac{1}{\beta}\Big)^{2+\varepsilon}.
		\endaligned$$
		Here we choose $\gamma> \frac{N-4}{2}$ satisfying $\frac{N+4}{2}-\gamma\geq \tau$. 
		
		For the term $J_{12}$, by taking $0 < \gamma \leq \min\{\alpha,N+4-\alpha\}$ and applying Lemma \ref{B2} again, we obtain that for any $x\in\Omega_{1}$
		$$\aligned
		&\Big(|x|^{-\alpha}\ast (W_{z_1,\beta}^{p-1}\sum_{j=2}^{m}W_{z_j,\beta})\Big)W_{z_1,\beta}^{p-1}\\
		\leq&C\Big[|x|^{-\alpha}\ast\Big(\frac{\beta^\frac{N+4-\alpha}{2}}{(1+\beta|x-z_{1}|)^{N+4-\alpha}}\sum_{j=2}^{m}\frac{\beta^{\frac{N-4}{2}}}{(1+\beta|x-z_{j}|)^{N-4}}\Big)\Big]\frac{\beta^{\frac{N+4-\alpha}{2}}}{(1+\beta|x-z_{1}|)^{N+4-\alpha}}\\
		\leq&C\sum_{j=1}^{m}\int_{\Omega_{j}}\frac{1}{|x-y|^{\alpha}}\frac{\beta^{N-\frac{\alpha}{2}}}{(1+\beta|y-z_{j}|)^{2N-\alpha-\tau_1}}dy\frac{\beta^{\frac{N+4-\alpha}{2}}}{(1+\beta|x-z_{1}|)^{N+4-\alpha}}\\
		\leq&C\sum_{j=1}^{m}\frac{\beta^{\frac{\alpha}{2}}}{(1+\beta|x-z_{j}|)^{\alpha}}\frac{\beta^{\frac{N+4-\alpha}{2}}}{(1+\beta|x-z_{1}|)^{N+4-\alpha}}\\
		\leq&C\frac{\beta^{\frac{N+4}{2}}}{(1+\beta|x-z_{1}|)^{\frac{N+4}{2}+\tau}}\sum_{j=2}^{m}\frac{1}{(\beta|z_{1}-z_{j}|)^{\gamma}}
		\leq C\Big(\frac{1}{\beta}\Big)^{2+\varepsilon}\frac{\beta^{\frac{N+4}{2}}}{(1+\beta|x-z_{1}|)^{\frac{N+4}{2}+\tau}},
		\endaligned$$
		where we choose $\gamma> \frac{N-4}{2}$ satisfying $\frac{N+4}{2}-\gamma\geq\tau$.
		By H$\ddot{o}$lder inequality, for the term $J_{13}$ we have
		$$\aligned
		&\Big(|x|^{-\alpha}\ast \Big(\sum_{j=2}^{m}W_{z_j,\beta}\Big)^{p}\Big)\Big(\sum_{j=2}^{m}W_{z_j,\beta}\Big)^{p-1}\\
		\leq& C\sum_{j=2}^{m}\int_{\Omega_{j}}\frac{1}{|x-y|^{\alpha}} \frac{\beta^{N-\frac{\alpha}{2}}}{(1+\beta|y-z_{j}|)^{2N-\alpha-\tau_1\frac{2N-\alpha}{N-4}}}dy\Big(\sum_{j=2}^{m}W_{z_j,\beta}\Big)^{p-1}\\
		\leq& C\sum_{j=2}^{m}\frac{\beta^{\frac{\alpha}{2}}}{(1+\beta|x-z_{j}|)^{\alpha}}\sum_{j=2}^{m}
		\frac{\beta^{\frac{N+4-\alpha}{2}}}{(1+\beta|x-z_{j}|)^{\frac{N+4-\alpha}{2}+\tau}}\Big(\sum_{j=2}^{m}\frac{1}{(1+\beta|x-z_{j}|)^{\frac{N+4-\alpha}{8-\alpha}(\frac{N-4}{2}-\frac{N-4}{N+4-\alpha}\tau)}}\Big)^{\frac{8-\alpha}{N-4}}\\
		\leq& C\Big(\frac{m}{\beta}\Big)^{\frac{N+4-\alpha}{2}-\tau}\sum_{j=2}^{m}\frac{\beta^{\frac{\alpha}{2}}}{(1+\beta|x-z_{j}|)^{\alpha}}\sum_{j=2}^{m}
		\frac{\beta^{\frac{N+4-\alpha}{2}}}{(1+\beta|x-z_{j}|)^{\frac{N+4-\alpha}{2}+\tau}}\\
		\leq& C\Big(\frac{m}{\beta}\Big)^{\frac{N+4}{2}-\tau} \sum_{j=2}^{m}\frac{\beta^{\frac{N+4}{2}}}{(1+\beta|x-z_{j}|)^{\frac{N+4}{2}+\tau}}
		\leq C (\frac{1}{\beta})^{2+\varepsilon}\sum_{j=2}^{m}\frac{\beta^{\frac{N+4}{2}}}{(1+\beta|x-z_{j}|)^{\frac{N+4}{2}+\tau}}.
		\endaligned$$
	    So we conclude that
		\begin{equation}\label{J1}
			\|J_{1}\|_{\ast\ast}\leq C(\frac{1}{\beta})^{2+\varepsilon}.
		\end{equation}
		Lemma 2.5 in \cite{GLN} shows that term $J_{2}$ and $J_{3}$ have the following estimates:
		\begin{equation*}
			\aligned
			|J_{2}| 
			&\leq C (\frac{1}{\beta})^{2+\varepsilon}\sum_{i=1}^{m}\frac{\beta^{\frac{N+4}{2}}}{(1+\beta|x-z_{i}|)^{\frac{N+4}{2}+\tau}},
			\endaligned
			\hspace{0.1cm}
			|J_{3}|\leq C (\frac{1}{\beta})^{2+\varepsilon}\sum_{j=1}^{m}\frac{\beta^{\frac{N+4}{2}}}{(1+\beta|x-z_{j}|)^{\frac{N+4}{2}+\tau}}.
		\end{equation*}
		So we get
		\begin{equation}\label{J2J3}
			\|J_{2}\|_{\ast\ast}\leq C(\frac{1}{\beta})^{2+\varepsilon},\hspace{0.4cm}\|J_3\|_{\ast\ast}\leq C(\frac{1}{\beta})^{2+\varepsilon}.
		\end{equation}
		As a result of \eqref{J1} and \eqref{J2J3} 
		, we obtain
		$$
		\|l_{m}\|_{\ast\ast}\leq C(\frac{1}{\beta})^{2+\varepsilon}.
		$$
	\end{proof}

	Now we will carry out the contraction mapping argument and prove that 
	\begin{lem}\label{C3}
		There is an integer $m_0 > 0$, such that for each $m \geq m_0$, $\beta\in[L_0m^{\frac{N-4}{N-8}},L_1m^{\frac{N-4}{N-8}}]$,
		$\overline{r}\in[r_{0}-\theta,r_{0}+\theta]$, $\overline{x}''\in B_{\theta}(x_{0}'')$, where $\theta> 0$ is a fixed small constant, \eqref{c14} has a unique solution $\phi= \phi_{\overline{r},\overline{x}'',\beta}\in H_{s}$ satisfying
		\begin{equation}\label{c15}
			\|\phi\|_{\ast}\leq C(\frac{1}{\beta})^{2+\varepsilon}, \ \ |c_{l}|\leq C(\frac{1}{\beta})^{2+n_{l}+\varepsilon},
		\end{equation}
		where $\varepsilon> 0$ is a small constant.
	\end{lem}
	\begin{proof}
		Recall that $\beta\in [L_0m^{\frac{N-4}{N-8}},L_1m^{\frac{N-4}{N-8}}]$. We set
		$$\aligned
		\mathcal{N}=&\Big\{w:w\in C(\mathbb{R}^N)\cap H_{s},\|w\|_{\ast}\leq\frac{1}{\beta^2},\\
		&\sum_{j=1}^{m}\int_{\mathbb{R}^{N}}\Big[p\Big(|x|^{-\alpha}\ast |Z_{z_j,\beta}^{p-1}Z_{j,l}|\Big)Z_{z_j,\beta}^{p-1}w+(p-1)\Big(|x|^{-\alpha}\ast |Z_{z_j,\beta}|^{p}\Big)Z_{z_j,\beta}^{p-2}Z_{j,l}w \Big] dx=0.\Big\},
		\endaligned$$
		where $l=1,2,\cdots, N$.
		Then by Lemma \ref{C2}, problem \eqref{c16} is equivalent to
		\begin{equation}\label{c19}
			\phi=\mathcal{T}(\phi)=:L_{m}(N(\phi))+L_{m}(l_{m}),
		\end{equation}
		where $L_m$ is defined in Lemma \ref{C2}. We will prove that $\mathcal{T}$ is a contraction map from $\mathcal{N}$ to $\mathcal{N}$.
		
		Firstly, for any $\phi \in\mathcal{N}$,
		$$
		\|\mathcal{T}(\phi)\|_{\ast}\leq C(\|N(\phi)\|_{\ast\ast}+\|l_{m}\|_{\ast\ast})\leq C(\|\phi\|_{\ast}^{2}+(\frac{1}{\beta})^{2+\varepsilon})\leq
		\frac{1}{\beta^2}.
		$$
		Hence, $\mathcal{T}$ maps $\mathcal{N}$ to $\mathcal{N}$.
		
		And for any $\phi_1,\phi_2 \in \mathcal{N}$,
		$$
		\|\mathcal{T}(\phi_{1})-\mathcal{T}(\phi_{2})\|_{\ast}
		=\|L_{m}(N(\phi_{1}))-L_{m}(N(\phi_{2}))\|_{\ast}\leq C\|N(\phi_{1})-N(\phi_{2})\|_{\ast\ast}.
		$$
		It is easy to check that
		$$\aligned
		\left|N(\phi_{1})-N(\phi_{2})\right|
		\leq& N'(\phi_1+\theta(\phi_2-\phi_1))|\phi_2-\phi_1|\\
		\leq& C(G(\phi_1)+G(\phi_2))\left|\phi_2-\phi_1\right|.
		\endaligned$$
		where
		$$\aligned
		G(\phi)=&\Big(|x|^{-\alpha}\ast Z_{\overline{r},\overline{x}'',\beta}^{p-1}\Big)Z_{\overline{r},\overline{x}'',\beta}^{p-2}\phi+\Big(|x|^{-\alpha}\ast (Z_{\overline{r},\overline{x}'',\beta}^{p-1}\phi)\Big)Z_{\overline{r},\overline{x}'',\beta}^{p-2}
		\\
		&+\Big(|x|^{-\alpha}\ast (Z_{\overline{r},\overline{x}'',\beta}^{p-2}\phi)\Big)(Z_{\overline{r},\overline{x}'',\beta}^{p-1}+Z_{\overline{r},\overline{x}'',\beta}^{p-2}\phi)+\Big(|x|^{-\alpha}\ast (Z_{\overline{r},\overline{x}'',\beta}^{p-2}\phi^{2})\Big)Z_{\overline{r},\overline{x}'',\beta}^{p-2}\\
		&+\Big(|x|^{-\alpha}\ast \phi^{p-1}\Big)\phi^{p-1}+\Big(|x|^{-\alpha}\ast \phi^{p}\Big)\phi^{p-2}.
		\endaligned$$
		According to the proof of Lemma \ref{C4}, we have
		$$
		\|\mathcal{T}(\phi_{1})-\mathcal{T}(\phi_{2})\|_{\ast}\leq C\|N(\phi_{1})-N(\phi_{2})\|_{\ast\ast}\leq C(\|\phi_1\|_{\ast}+\|\phi_2\|_{\ast})\|\phi_2-\phi_1\|_{\ast}\leq \frac{1}{2}\|\phi_{1}-\phi_{2}\|_{\ast},
		$$
		which means that $\mathcal{T}$ is a contraction map. Thus by contraction mapping theorem, there exists a unique $\phi\in\mathcal{N}$ such that \eqref{c19} holds. Moreover, by Lemmas \ref{C2}, \ref{C4} and \ref{C5}, we obtain
		$$
		\|\phi\|_{\ast}\leq C(\frac{1}{\beta})^{2+\varepsilon}
		$$
		and the estimate of $c_l$ from \eqref{c13}.
	\end{proof}

	%
	
	After the above finite reduction arguments, now we will look for suitable $(\overline{r},\overline{x}'',\beta)$ such that the function $Z_{\overline{r},\overline{x}'',\beta}+
	\phi_{\overline{r},\overline{x}'',\beta}$ is a solution of \eqref{CFL}. For this purpose, we establish the new local Poho\v{z}aev identities to determine the location of the bubbles.
	
	\begin{lem}\label{D1}
		Suppose that $(\overline{r},\overline{x}'',\beta)$ satisfies
		\begin{equation}\label{d1}
			\int_{D_{\rho}}\Big(\Delta^2 u_{m}+V(r,x'')u_{m}
			-\Big(|x|^{-\alpha}\ast (|u_{m}|^{p})\Big)u_{m}^{p-1}\Big)\langle x,\nabla u_{m}\rangle dx=0,
		\end{equation}
		\begin{equation}\label{d2}
			\int_{D_{\rho}}\Big(\Delta^2 u_{m}+V(r,x'')u_{m}
			-\Big(|x|^{-\alpha}\ast (|u_{m}|^{p})\Big)u_{m}^{p-1}\Big)\frac{\partial u_{m}}{\partial x_{i}} dx=0, i=3,\cdots,N,
		\end{equation}
		and
		\begin{equation}\label{d3}
			\int_{\mathbb{R}^{N}}(\Delta^2u_{m}+V(r,x'')u_{m}
			-\Big(|x|^{-\alpha}\ast (|u_{m}|^{p})\Big)u_{m}^{p-1})\frac{\partial Z_{\overline{r},\overline{x}'',\beta}}{\partial \beta} dx=0,
		\end{equation}
		where $u_{m}=Z_{\overline{r},\overline{x}'',\beta}+
		\phi_{\overline{r},\overline{x}'',\beta}$ and $D_{\rho}=\{(r, x''):|(r, x'')-(r_0, x_0'' )|\leq\rho\}$ with $\rho\in(2\delta, 5\delta)$. Then $c_{i}=0, i=1,\cdot\cdot\cdot,N$.
	\end{lem}
	\begin{proof}
		Since $Z_{\overline{r},\overline{x}'',\beta}=0$ in $\mathbb{R}^{N}\backslash D_{\rho}$, we see that if \eqref{d1}-\eqref{d3} hold, then
		\begin{equation}\label{d4}
			\sum_{l=1}^{N}c_{l}\sum_{j=1}^{m}\int_{\mathbb{R}^{N}}\Big[(p-1)\Big(|x|^{-\alpha}\ast |Z_{z_j,\beta}|^{p}\Big)Z_{z_j,\beta}^{p-2}Z_{j,l}+p\Big(|x|^{-\alpha}\ast (Z_{z_j,\beta}^{p-1}
			Z_{j,l})\Big)Z_{z_j,\beta}^{p-1}\Big]vdx=0,
		\end{equation}
		for $v=\langle x,\nabla u_{m}\rangle$, $\frac{\partial u_{m}}{\partial x_{i}}, i=3,\cdot\cdot\cdot,N$ and $\frac{\partial Z_{\overline{r},\overline{x}'',\beta}}{\partial \beta}$.
		
		Direct calculations show that
		$$\aligned
		&\left|\int_{\mathbb{R}^{N}}\Big(|x|^{-\alpha}\ast |Z_{z_j,\beta}|^{p}\Big)Z_{z_j,\beta}^{p-2}Z_{j,i}
		\frac{\partial Z_{z_l,\beta}}{\partial x_{i}}dx\right|
		\leq C\beta^{2}\int_{\mathbb{R}^{N}}\frac{(x_{i}-\beta\overline{x}_{i})^{2}}{(1+|x-\beta z_{j} |^{2})^{\frac{N+6}{2}}}
		\frac{1}{(1+|x-\beta z_{l} |^{2})^{\frac{N-2}{2}}}dx,
		\endaligned$$
		where $ i=3,\cdot\cdot\cdot,N$ and $(\overline{x}_{3}, \overline{x}_{4}, \cdot\cdot\cdot,\overline{x}_{N})=\overline{x}$.
		If $l=j$, we have
		$$
		\left|\int_{\mathbb{R}^{N}}\Big(|x|^{-\alpha}\ast |Z_{z_j,\beta}|^{p}\Big)Z_{z_j,\beta}^{p-2}Z_{j,i}
		\frac{\partial Z_{z_l,\beta}}{\partial x_{i}}dx\right|=O (\beta^{2}).
		$$
		If $l\neq j$, similar to the arguments in Lemma \ref{C5}, we can prove:
		$$
		\left|\int_{\mathbb{R}^{N}}\Big(|x|^{-\alpha}\ast |Z_{z_j,\beta}|^{p}\Big)Z_{z_j,\beta}^{p-2}Z_{j,i}
		\left(\sum_{l=1,\neq j}^{m}\frac{\partial Z_{z_l,\beta}}{\partial x_{i}}\right)dx\right|=O (\beta^{2-\varepsilon}),
		$$
		for some $\varepsilon> 0$. Then, we can get
		\begin{equation}\label{d6}
			\sum_{j=1}^{m}\int_{\mathbb{R}^{N}}\Big[(p-1)\Big(|x|^{-\alpha}\ast |Z_{z_j,\beta}|^{p}\Big)Z_{z_j,\beta}^{p-2}Z_{j,i}
			+p\Big(|x|^{-\alpha}\ast (Z_{z_j,\beta}^{p-1}
			Z_{j,i})\Big)Z_{z_j,\beta}^{p-1}\Big]\frac{\partial Z_{\overline{r},\overline{x}'',\beta}}{\partial x_{i}}dx=m\beta^{2}(a_{1}+o(1)), 
		\end{equation}
		where $i=3,\cdot\cdot\cdot, N$, for some constants $a_{1}\neq0$. Similarly, we also have
		\begin{equation}\label{d5}
			\sum_{j=1}^{m}\int_{\mathbb{R}^{N}}\Big[(p-1)\Big(|x|^{-\alpha}\ast |Z_{z_j,\beta}|^{p}\Big)Z_{z_j,\beta}^{p-2}Z_{j,2}
			+p\Big(|x|^{-\alpha}\ast (Z_{z_j,\beta}^{p-1}
			Z_{j,2})\Big)Z_{z_j,\beta}^{p-1}\Big]\langle x',\nabla_{x'} Z_{\overline{r},\overline{x}'',\beta}\rangle dx=m\beta^{2}(a_{2}+o(1)),
		\end{equation}
		and
		\begin{equation}\label{d7}
			\sum_{j=1}^{m}\int_{\mathbb{R}^{N}}\Big[(p-1)\Big(|x|^{-\alpha}\ast |Z_{z_j,\beta}|^{p}\Big)Z_{z_j,\beta}^{p-2}Z_{j,1}
			+p\Big(|x|^{-\alpha}\ast (Z_{z_j,\beta}^{p-1}
			Z_{j,1})\Big)Z_{z_j,\beta}^{p-1}\Big]\frac{\partial Z_{\overline{r},\overline{y}'',\beta}}{\partial \beta}dx=\frac{m}{\beta^{2}}(a_{3}+o(1)),
		\end{equation}
		for some constants $a_{2}\neq0$ and $a_3>0$.
		
		It is easy to check that
		$$\aligned
		&\left|\int_{\mathbb{R}^{N}}\Big(|x|^{-\alpha}\ast |Z_{z_j,\beta}|^{p}\Big)Z_{z_j,\beta}^{p-2}Z_{j,1}\frac{\partial \phi_{\overline{r},\overline{x}'',\beta}}{\partial x_{i}}dx\right|\\
		=&\int_{\mathbb{R}^N}\int_{\mathbb{R}^N}
		\frac{|Z_{z_j,\beta}(y)|^{p}
			\Big(\frac{\partial Z_{z_j,\beta}^{p-2}(x)}{\partial x_{i}}Z_{j,1}(x)+ Z_{z_j,\beta}^{p-2}(x)\frac{\partial Z_{j,1}}{\partial x_{i}}(x)\Big)\phi_{\overline{r},\overline{x}'',\beta}(x)}{|x-y|^{\alpha}}dxdy\\
		&+\int_{\mathbb{R}^N}\int_{\mathbb{R}^N}
		\frac{|Z_{z_j,\beta}(y)|^{p}
			(x_{i}-y_{i}) Z_{z_j,\beta}^{p-2}(x)Z_{j,1}(x)\phi_{\overline{r},\overline{x}'',\beta}(x)}{|x-y|^{\alpha+2}}dxdy,
		\endaligned$$
		where $j=1,2,\cdot\cdot\cdot,m$ and $i=3,\cdot\cdot\cdot,N$. Integrating by parts, by \eqref{c15}, we can obtain
				$$\aligned
		&\left|\int_{\mathbb{R}^N}\int_{\mathbb{R}^N}
		\frac{|Z_{z_j,\beta}(y)|^{p}
			\frac{\partial Z_{z_j,\beta}^{p-2}(x)}{\partial x_{i}}Z_{j,1}(x)\phi_{\overline{r},\overline{x}'',\beta}(x)}{|x-y|^{\alpha}}dxdy\right|\\
		\leq &
		C\|\phi\|_{\ast}\int_{\mathbb{R}^N}\frac{\beta^{\frac{\alpha}{2}}}{(1+\beta^{2}|x-z_{j} |^{2})^{\frac{\alpha}{2}}}\frac{\beta^{\frac{8-\alpha}{2}}\beta^2|x-z_{j}|}{(1+\beta^{2}|x-z_{j} |^{2})^{\frac{8-\alpha}{2}+1}}\frac{\beta^{\frac{N-6}{2}}}{(1+\beta^{2}|x-z_{j} |^{2})^{\frac{N-4}{2}}}\sum_{k=1}^{m} \frac{\beta^{\frac{N-4}{2}}}{(1+\beta|x-z_{k}|)^{\frac{N-4}{2}+\tau}}dx\\
		\leq & 
		C\|\phi\|_{\ast}\int_{\mathbb{R}^N}\frac{|x-\beta z_{j} |}{(1+|x-\beta z_{j} |^{2})^{\frac{N}{2}+3}}\sum_{k=1}^{m}\frac{1}{(1+|x-\beta z_{k} |)^{\frac{N-4}{2}+\tau}}dx\\
		\leq &
		C\|\phi\|_{\ast}\sum_{j=1}^{m}\int_{\mathbb{R}^N}\frac{1}{(1+|x-\beta z_{j} |)^{\frac{3N}{2}+3}}dx\\
		=&
		O(\frac{1}{\beta^{\varepsilon}}).
		\endaligned$$
		Similarly, we have
		$$\aligned
		&\left|\int_{\mathbb{R}^N}\int_{\mathbb{R}^N}
		\frac{|Z_{z_j,\beta}(y)|^{p}
		 Z_{z_j,\beta}^{p-2}(x)\frac{\partial Z_{j,1}}{\partial x_{i}}(x)\phi_{\overline{r},\overline{x}'',\beta}(x)}{|x-y|^{\alpha}}dxdy\right|
		=O(\frac{1}{\beta^{\varepsilon}}),
		\endaligned$$
		and
		$$
		\int_{\mathbb{R}^N}\int_{\mathbb{R}^N}
		\frac{|Z_{z_j,\beta}(y)|^{p}
			(x_{i}-y_{i})Z_{z_j,\beta}^{p-2}(x)Z_{j,1}(x)\phi_{\overline{r},\overline{x}'',\beta}(x)}{|x-y|^{\alpha+2}}dxdy=O(\frac{1}{\beta^{\varepsilon}}).
		$$
		So we can conclude that 
		$$
		\int_{\mathbb{R}^{N}}\Big(|x|^{-\alpha}\ast |Z_{z_j,\beta}|^{p}\Big)Z_{z_j,\beta}^{p-2}Z_{j,1}\frac{\partial \phi_{\overline{r},\overline{x}'',\beta}}{\partial x_{i}}dx=O(\frac{1}{\beta^{\varepsilon}}).
		$$
		Similarly, we can also get
		$$
		\int_{\mathbb{R}^{N}}\Big(|x|^{-\alpha}\ast (Z_{z_j,\beta}^{p-1}
		Z_{j,1})\Big)Z_{z_j,\beta}^{p-1}\frac{\partial \phi_{\overline{r},\overline{x}'',\beta}}{\partial x_{i}}dx=O(\frac{1}{\beta^{\varepsilon}}).
		$$
		Therefore
		$$
		c_{1}\sum_{j=1}^{m}\int_{\mathbb{R}^{N}}\Big[(p-1)\Big(|x|^{-\alpha}\ast |Z_{z_j,\beta}|^{p}\Big)Z_{z_j,\beta}^{p-2}Z_{j,1}+p\Big(|x|^{-\alpha}\ast (Z_{z_j,\beta}^{p-1}
		Z_{j,1})\Big)Z_{z_j,\beta}^{p-1}\Big]\frac{\partial \phi_{\overline{r},\overline{x}'',\beta}}{\partial x_{i}}dx=o(m|c_{1}|).
		$$
		Using the same arguments as above, we can get
		$$
		\sum_{l=2}^{N}c_{l}\sum_{j=1}^{m}\int_{\mathbb{R}^{N}}\Big[(p-1)\Big(|x|^{-\alpha}\ast |Z_{z_j,\beta}|^{p}\Big)Z_{z_j,\beta}^{p-2}Z_{j,l}+p\Big(|x|^{-\alpha}\ast (Z_{z_j,\beta}^{p-1}
		Z_{j,l})\Big)Z_{z_j,\beta}^{p-1}\Big]\frac{\partial \phi_{\overline{r},\overline{x}'',\beta}}{\partial x_{i}}dx=o(m\beta^{2})\sum_{l=2}^{N}|c_{l}|.
		$$
		So we prove that
		$$\aligned
		\sum_{l=1}^{N}c_{l}\sum_{j=1}^{m}\int_{\mathbb{R}^{N}}\Big[(p-1)\Big(|x|^{-\alpha}\ast |Z_{z_j,\beta}|^{p}\Big)Z_{z_j,\beta}^{p-2}Z_{j,l}+p&\Big(|x|^{-\alpha}\ast (Z_{z_j,\beta}^{p-1}
		Z_{j,l})\Big)Z_{z_j,\beta}^{p-1}\Big]\frac{\partial \phi_{\overline{r},\overline{x}'',\beta}}{\partial x_{i}}dx\\
		&=o(m\beta^{2})\sum_{l=2}^{N}|c_{l}|+o(m|c_{1}|),
		\endaligned$$
		where $i=3,\cdot\cdot\cdot,N$. Analogously, we can also obtain
		$$\aligned
		\sum_{l=1}^{N}c_{l}\sum_{j=1}^{m}\int_{\mathbb{R}^{N}}\Big[(p-1)\Big(|x|^{-\alpha}\ast |Z_{z_j,\beta}|^{p}\Big)Z_{z_j,\beta}^{p-2}Z_{j,l}+p&\Big(|x|^{-\alpha}\ast (Z_{z_j,\beta}^{p-1}
		Z_{j,l})\Big)Z_{z_j,\beta}^{p-1}\Big]\langle x,\nabla \phi_{\overline{r},\overline{x}'',\beta}\rangle dx\\
		&=o(m\beta^{2})\sum_{l=2}^{N}|c_{l}|+o(m|c_{1}|).
		\endaligned$$
		These two equalities together with \eqref{d4} can lead to 
		\begin{equation}\label{d8}
			\aligned
			\sum_{l=1}^{N}c_{l}\sum_{j=1}^{m}\int_{\mathbb{R}^{N}}\Big[(p-1)\Big(|x|^{-\alpha}\ast |Z_{z_j,\beta}|^{p}\Big)_{z_j,\beta}^{p-2}Z_{j,l}+p&\Big(|x|^{-\alpha}\ast (Z_{z_j,\beta}^{p-1}
			Z_{j,l})\Big)Z_{z_j,\beta}^{p-1}\Big]vdx\\
			&=o(m\beta^{2})\sum_{l=2}^{N}|c_{l}|+o(m|c_{1}|),
			\endaligned
		\end{equation}
		for $v=\langle x,\nabla Z_{\overline{r},\overline{x}'',\beta}\rangle$, $\frac{\partial Z_{\overline{r},\overline{x}'',\beta}}{\partial x_{i}}$, $ i=3,\cdot\cdot\cdot,N$.
		
		From
		$$
		\langle x,\nabla Z_{\overline{r},\overline{x}'',\beta}\rangle=\langle x',\nabla_{x'} Z_{\overline{r},\overline{x}'',\beta}\rangle+\langle x'',\nabla_{x''} Z_{\overline{r},\overline{x}'',\beta}\rangle,
		$$
		we find
		\begin{equation}\label{d9}
			\aligned
			&\sum_{l=1}^{N}c_{l}\sum_{j=1}^{m}\int_{\mathbb{R}^{N}}\Big[(p-1)\Big(|x|^{-\alpha}\ast |Z_{z_j,\beta}|^{p}\Big)Z_{z_j,\beta}^{p-2}Z_{j,l} +p\Big(|x|^{-\alpha}\ast (Z_{z_j,\beta}^{p-1}
			Z_{j,l})\Big)Z_{z_j,\beta}^{p-1}\Big]\langle x,\nabla Z_{\overline{r},\overline{x}'',\beta}\rangle dx\\
			&\hspace{2mm}=c_{2}\sum_{j=1}^{m}\int_{\mathbb{R}^{N}}\Big[(p-1)\Big(|x|^{-\alpha}\ast |Z_{z_j,\beta}|^{p}\Big)Z_{z_j,\beta}^{p-2}Z_{j,2}+p\Big(|x|^{-\alpha}\ast (Z_{z_j,\beta}^{p-1}
			Z_{j,2})\Big)Z_{z_j,\beta}^{p-1}\Big]\langle x',\nabla_{x'} Z_{\overline{r},\overline{x}'',\beta}\rangle dx\\
			&\hspace{4mm}+o(m\beta^{2})\sum_{l=3}^{N}|c_{l}|+o(m|c_{1}|),
			\endaligned
		\end{equation}
		and
		\begin{equation}\label{d10}
			\aligned
			&\sum_{l=1}^{N}c_{l}\sum_{j=1}^{m}\int_{\mathbb{R}^{N}}\Big[(p-1)\Big(|x|^{-\alpha}\ast |Z_{z_j,\beta}|^{p}\Big)Z_{z_j,\beta}^{p-2}Z_{j,l} +p\Big(|x|^{-\alpha}\ast (Z_{z_j,\beta}^{p-1}
			Z_{j,l})\Big)Z_{z_j,\beta}^{p-1}\Big]\frac{\partial Z_{\overline{r},\overline{x}'',\beta}}{\partial x_{i}}dx\\
			&\hspace{2mm}=c_{i}\sum_{j=1}^{m}\int_{\mathbb{R}^{N}}\Big[(p-1)\Big(|x|^{-\alpha}\ast |Z_{z_j,\beta}|^{p}\Big)Z_{z_j,\beta}^{p-2}Z_{j,i} +p\Big(|x|^{-\alpha}\ast (Z_{z_j,\beta}^{p-1}
			Z_{j,i})\Big)Z_{z_j,\beta}^{p-1}\Big]\frac{\partial Z_{\overline{r},\overline{x}'',\beta}}{\partial x_{i}}dx\\
			&\hspace{4mm}+o(m\beta^{2})\sum_{l\neq1,i}^{N}|c_{l}|+o(m|c_{1}|), \ i=3,\cdot\cdot\cdot, N.
			\endaligned
		\end{equation}
		
		Combining \eqref{d6},\eqref{d5} and \eqref{d8}-\eqref{d10}, we have
		\begin{equation}\label{ci}
			c_{i}=o(\frac{1}{\beta^{2}})c_{1}, \ i=2,\cdot\cdot\cdot, N.
		\end{equation}
		
		Together with \eqref{d7}, we obtain
		$$\aligned
		0=&\sum_{l=1}^{N}c_{l}\sum_{j=1}^{m}\int_{\mathbb{R}^{N}}\Big[(p-1)\Big(|x|^{-\alpha}\ast |Z_{z_j,\beta}|^{p}\Big)Z_{z_j,\beta}^{p-2}Z_{j,l}+p\Big(|x|^{-\alpha}\ast (Z_{z_j,\beta}^{p-1}
		Z_{j,l})\Big)Z_{z_j,\beta}^{p-1}\Big]\frac{\partial Z_{\overline{r},\overline{x}'',\beta}}{\partial \beta} dx\\
		=&c_{1}\sum_{j=1}^{m}\int_{\mathbb{R}^{N}}\Big[(p-1)\Big(|x|^{-\alpha}\ast |Z_{z_j,\beta}|^{p}\Big)Z_{z_j,\beta}^{p-2}Z_{j,1}+p\Big(|x|^{-\alpha}\ast (Z_{z_j,\beta}^{p-1}
		Z_{j,1})\Big)Z_{z_j,\beta}^{p-1}\Big]\frac{\partial Z_{\overline{r},\overline{x}'',\beta}}{\partial \beta} dx+o(\frac{m}{\beta^{2}})c_{1}\\
		=&\frac{m}{\beta^{2}}(a_{3}+o(1))c_{1}+o\big(\frac{m}{\beta^{2}}\big)c_{1}.
		\endaligned$$
		So $c_1=0$.
	\end{proof}

	\begin{lem}\label{D3}
		We have
		\begin{equation}\label{main1}
			\aligned
			&\int_{\mathbb{R}^{N}}\Big(\Delta^2 u_{m}+V(r,x'')u_{m}
			-\Big(|x|^{-\alpha}\ast |u_{m}|^{p}\Big)u_{m}^{p-1}\Big)\frac{\partial Z_{\overline{r},\overline{x}'',\beta}}{\partial \beta}dx\\
			=&m\Big(-\frac{A_{1}}{\beta^{5}}V(\overline{r},\overline{x}'')+\sum_{j=2}^{m}
			\frac{A_{2}}{\beta^{N-3}|z_{1}-z_{j}|^{N-4}}+O(\frac{1}{\beta^{5+\varepsilon}})\Big)\\
			=&m\Big(
			-\frac{A_1}{\beta^{5}}V(\overline{r},\overline{x}'')
			+\frac{A_3m^{N-4}}{\beta^{N-3}}
			+O(\frac{1}{\beta^{5+\varepsilon}})\Big).
			\endaligned
		\end{equation}
	\end{lem}
	\begin{proof}
		Direct calculation shows that
		$$\aligned
		&\int_{\mathbb{R}^{N}}\Big(\Delta^2 u_{m}+V(r,x'')u_{m} -\Big(|x|^{-\alpha}\ast |u_{m}|^{p}\Big)u_{m}^{p-1}\Big)\frac{\partial Z_{\overline{r},\overline{x}'',\beta}}{\partial \beta} dx
		=\langle J'(Z_{\overline{r},\overline{x}'',\beta}),\frac{\partial Z_{\overline{r},\overline{x}'',\beta}}{\partial \beta}\rangle\\&+m\Big\langle \Delta^2\phi +V(r,x'')\phi-p\Big(|x|^{-\alpha}\ast (Z_{\overline{r},\overline{x}'',\beta}^{p-1}\phi)\Big)Z_{\overline{r},\overline{x}'',\beta}^{p-1}
		-(p-1)\Big(|x|^{-\alpha}\ast |Z_{\overline{r},\overline{x}'',\beta}|^{p}\Big)Z_{\overline{r},\overline{x}'',\beta}^{p-2}\phi,\frac{\partial Z_{z_1,\beta}}{\partial \beta}\Big\rangle\\
		&-\int_{\mathbb{R}^{N}}\Big(|x|^{-\alpha}\ast |Z_{\overline{r},\overline{x}'',\beta}+\phi|^{p}\Big)(Z_{\overline{r},\overline{x}'',\beta}+\phi)^{p-1}\frac{\partial Z_{\overline{r},\overline{x}'',\beta}}{\partial \beta} dx
		+\int_{\mathbb{R}^{N}}\Big(|x|^{-\alpha}\ast |Z_{\overline{r},\overline{x}'',\beta}|^{p}\Big)Z_{\overline{r},\overline{x}'',\beta}^{p-1}\frac{\partial Z_{\overline{r},\overline{x}'',\beta}}{\partial \beta} dx\\
		&+p\int_{\mathbb{R}^{N}}\Big(|x|^{-\alpha}\ast (Z_{\overline{r},\overline{x}'',\beta}^{p-1}\phi)\Big)Z_{\overline{r},\overline{x}'',\beta}^{p-1}\frac{\partial Z_{\overline{r},\overline{x}'',\beta}}{\partial \beta} dx
		+(p-1)\int_{\mathbb{R}^{N}}\Big(|x|^{-\alpha}\ast |Z_{\overline{r},\overline{x}'',\beta}|^{p}\Big)Z_{\overline{r},\overline{x}'',\beta}^{p-2}\phi\frac{\partial Z_{\overline{r},\overline{x}'',\beta}}{\partial \beta} dx\\
		:&=\langle J'(Z_{\overline{r},\overline{x}'',\beta}),\frac{\partial Z_{\overline{r},\overline{x}'',\beta}}{\partial \beta}\rangle+mI_{1}-I_{2}.
		\endaligned$$
		
		Using \eqref{c7}, we obtain
		$$
		I_{1}=O(\frac{\|\phi\|_{\ast}}{\beta^{3+\varepsilon}})
		=O(\frac{1}{\beta^{5+\varepsilon}}).
		$$
		
		Recalling the arguments in Lemma \ref{C4} and \eqref{c15}, it is easy to see
		$$\aligned
		I_{2}=&\int_{\mathbb{R}^{N}}\Big(|x|^{-\alpha}\ast |Z_{\overline{r},\overline{x}'',\beta}+\phi|^{p}\Big)(Z_{\overline{r},\overline{x}'',\beta}+\phi)^{p-1}\frac{\partial Z_{\overline{r},\overline{x}'',\beta}}{\partial \beta} dx
		-\int_{\mathbb{R}^{N}}\Big(|x|^{-\alpha}\ast |Z_{\overline{r},\overline{x}'',\beta}|^{p}\Big)Z_{\overline{r},\overline{x}'',\beta}^{p-1}\frac{\partial Z_{\overline{r},\overline{x}'',\beta}}{\partial \beta} dx\\
		&-p\int_{\mathbb{R}^{N}}\Big(|x|^{-\alpha}\ast (Z_{\overline{r},\overline{x}'',\beta}^{p-1}\phi)\Big)Z_{\overline{r},\overline{x}'',\beta}^{p-1}\frac{\partial Z_{\overline{r},\overline{x}'',\beta}}{\partial \beta} dx
		-(p-1)\int_{\mathbb{R}^{N}}\Big(|x|^{-\alpha}\ast |Z_{\overline{r},\overline{x}'',\beta}|^{p}\Big)Z_{\overline{r},\overline{x}'',\beta}^{p-2}\phi\frac{\partial Z_{\overline{r},\overline{x}'',\beta}}{\partial \beta} dx\\
		\leq&C\int_{\mathbb{R}^{N}}|N(\phi)|\left|\frac{\partial Z_{\overline{r},\overline{x}'',\beta}}{\partial \beta}\right| dx\\
		\leq&C\|\phi\|_{\ast}^{2}\int_{\mathbb{R}^{N}}\sum_{j=1}^{m}\frac{\beta^{\frac{N+4}{2}}}{(1+\beta|x-z_{j}|)^{\frac{N+4}{2}+\tau}}\sum_{j=1}^{m}\frac{\beta^{\frac{N-6}{2}}}{(1+\beta|x-z_{j}|)^{\frac{N-2}{2}}} dx\\
		\leq&\frac{C\|\phi\|_{\ast}^2}{\beta}\int_{\mathbb{R}^{N}}\sum_{j=1}^{m}\frac{1}{(1+\beta|x-z_{j}|)^{N+1}} dx=O(\frac{m}{\beta^{5+\varepsilon}}) .
		\endaligned$$
		
		So, we have proved that
		$$
		\Big\langle J'(Z_{\overline{r},\overline{x}'',\beta}+\phi),\frac{\partial Z_{\overline{r},\overline{x}'',\beta}}{\partial \beta}\Big\rangle
		=\Big\langle J'(Z_{\overline{r},\overline{x}'',\beta}),\frac{\partial Z_{\overline{r},\overline{x}'',\beta}}{\partial \beta}\Big\rangle+O(\frac{m}{\beta^{5+\varepsilon}}).
		$$
		Combined this with Lemma \ref{D2}, we finish our proof.
	\end{proof}	
	Using the same arguments in Lemma \ref{D3} and Lemma \ref{D2}, we can also prove
	\begin{lem}
		We have
		\begin{equation}\aligned\label{main2}
			\Big\langle J'(Z_{\overline{r},\overline{x}'',\beta}+\phi),\frac{\partial Z_{\overline{r},\overline{x}'',\beta}}{\partial \overline{r}}\Big\rangle
			=&m\Big(\frac{A_1}{\beta^{4}}\frac{\partial V(\overline{r},\overline{x}'')}{\partial \overline{r}}+\sum_{j=2}^{m}
			\frac{A_{2}}{\overline{r}\beta^{N-4}|z_{1}-z_{j}|^{N-4}}+O(\frac{1}{\beta^{3+\varepsilon}})\Big),
			\endaligned\end{equation}
		and
		\begin{equation}\label{main3}
			\Big\langle J'(Z_{\overline{r},\overline{x}'',\beta}+\phi),\frac{\partial Z_{\overline{r},\overline{x}'',\beta}}{\partial \overline{x}_{j}''}\Big\rangle
			=m\Big(\frac{A_1}{\beta^{4}}\frac{\partial V(\overline{r},\overline{x}'')}{\partial \overline{x}_{j}''}+O(\frac{1}{\beta^{3+\varepsilon}})\Big), \ \ j=3,\cdot\cdot\cdot,N,
		\end{equation}
		where $A_i$, $i=1, 2$ are some positive constants.
	\end{lem}

	We observe from the above lemma that the error term in \eqref{main2} and \eqref{main3} may destroy the main term which causes the difficulty to determine the location of the bubbles. So we will follow the idea from \cite{PWY} and consider the local Poho$\check{z}$aev type identities \eqref{d1} and \eqref{d2}. We know that \eqref{d1} is equivalent to
	\begin{equation}\label{d12}
		\begin{split}
			&\frac{4-N}{2}\int_{D_{\rho}}\Delta^2 u_{m}u_{m}dx-\frac{1}{2}\int_{D_{\rho}}\big(NV(x)+\langle x,\nabla V(x)\rangle\big) u_{m}^{2}dx
			+\frac{N}{p}\int_{D_{\rho}}\int_{\R^N}\frac{|u_{m}(y)|^{p}|u_{m}(x)|^{p}}{|x-y|^\alpha}dxdy\\&-\frac{\alpha}{p}
			\int_{D_{\rho}}\int_{\R^N}x(x-y)\frac{|u_{m}(y)|^{p}|u_{m}(x)|^{p}}{|x-y|^{\alpha+2}}dxdy
			=O\Big(\int_{\partial D_{\rho}}\Big(\int_{\R^N} \frac{|\phi(y)|^{p}}{|x-y|^\alpha}dy\Big)|\phi|^{p}ds\\
			&\hspace{6cm}+\int_{\partial D_{\rho}}\Big(\phi^{2}+\sum_{j=1}^{3}|\nabla^{j}\phi||\nabla^{4-j}\phi|+\sum_{j=0}^{3}|\nabla^{j}\phi||\nabla^{3-j}\phi|\Big)ds\Big).
		\end{split}
	\end{equation}
		
	Since
	$
	\sum_{j=1}^{m}\int_{\mathbb{R}^{N}}\Big[p\Big(|x|^{-\alpha}\ast |Z_{z_j,\beta}^{p-1}Z_{j,l}|\Big)Z_{z_j,\beta}^{p-1}\phi+(p-1)\Big(|x|^{-\alpha}\ast |Z_{z_j,\beta}|^{p}\Big)Z_{z_j,\beta}^{p-2}Z_{j,l}\phi \Big] dx=0,
	$
	we obtain from \eqref{c14} that
	\begin{equation}\label{d13}
		\aligned
		\int_{D_{\rho}}&\Delta^{2} u_{m}u_{m}dx+\int_{D_{\rho}}V(x) u_{m}^{2}dx=\int_{D_{\rho}}\Big(|x|^{-\alpha}\ast |u_{m}|^{p}\Big)|u_{m}|^{p} dx\\
		&+\sum_{l=1}^{N}c_{l}\sum_{j=1}^{m}\int_{D_{\rho}}\Big[p\Big(|x|^{-\alpha}\ast (Z_{z_j,\beta}^{p-1}
		Z_{j,l})\Big)Z_{z_j,\beta}^{p-1}+(p-1)\Big(|x|^{-\alpha}\ast |Z_{z_j,\beta}|^{p}\Big)Z_{z_j,\beta}^{p-2}Z_{j,l}\Big]Z_{\overline{r},\overline{x}'',\beta}dx.
		\endaligned\end{equation}
	Inserting \eqref{d13} into \eqref{d12}, we obtain
	\begin{equation}\label{d14}
		\aligned
		&\int_{D_{\rho}}\big(2V(x)+\frac{1}{2}\langle x,\nabla V(x)\rangle\big) u_{m}^{2}dx\\
		=&(\frac{4-N}{2}+\frac{N}{p})\int_{D_{\rho}}\Big(|x|^{-\alpha}\ast |u_{m}|^{p}\Big)|u_{m}|^{p} dx-\frac{\alpha}{p}\int_{ D_{\rho}}\int_{\mathbb{R}^{N}}x(x-y)\frac{|\phi(x)|^{p}|\phi(y)|^{p}}{|x-y|^{\alpha+2}}dxdy\\&-\frac{N-4}{2}\sum_{l=1}^{N}c_{l}\sum_{j=1}^{m}\int_{\mathbb{R}^{N}}\Big[p\Big(|x|^{-\alpha}\ast Z_{z_j,\beta}^{p-1}Z_{j,l}\Big)Z_{z_j,\beta}^{p-1}+(p-1)\Big(|x|^{-\alpha}\ast |Z_{z_j,\beta}|^{p}\Big)Z_{z_j,\beta}^{p-2}Z_{j,l}\Big]Z_{\overline{r},\overline{x}'',\beta}dx\\
		&+O\Big(\int_{\partial D_{\rho}}\Big(\phi^{2}+\sum_{j=1}^{3}|\nabla^{j}\phi||\nabla^{4-j}\phi|+\sum_{j=0}^{3}|\nabla^{j}\phi||\nabla^{3-j}\phi|\Big)ds
		+\int_{\partial D_{\rho}}\Big(\int_{\R^N} \frac{|\phi(y)|^{p}}{|x-y|^\alpha}dy\Big)|\phi|^{p}ds\Big)\\
		=&-\frac{N-4}{2}\sum_{l=1}^{N}c_{l}\sum_{j=1}^{m}\int_{\mathbb{R}^{N}}\Big[p\Big(|x|^{-\alpha}\ast Z_{z_j,\beta}^{p-1}Z_{j,l}\Big)Z_{z_j,\beta}^{p-1}+(p-1)\Big(|x|^{-\alpha}\ast |Z_{z_j,\beta}|^{p}\Big)Z_{z_j,\beta}^{p-2}Z_{j,l}\Big]Z_{\overline{r},\overline{x}'',\beta}dx\\
		&+O\Big(\int_{\partial D_{\rho}}\Big(\phi^{2}+\sum_{j=1}^{3}|\nabla^{j}\phi||\nabla^{4-j}\phi|+\sum_{j=0}^{3}|\nabla^{j}\phi||\nabla^{3-j}\phi|\Big)ds
		+\int_{\partial D_{\rho}}\Big(\int_{\R^N} \frac{|\phi(y)|^{p}}{|x-y|^\alpha}dy\Big)|\phi|^{p}ds\Big)\\
		&+O\Big(\int_{ D_{\rho}}\int_{\mathbb{R}^{N}\backslash D_{\rho}}x(x-y)\frac{|\phi(x)|^{p}|\phi(y)|^{p}}{|x-y|^{\alpha+2}}dxdy\Big)+O(\frac{1}{\beta^{4+\varepsilon}}), \ i=3,\cdot\cdot\cdot,N.
		\endaligned\end{equation}

	In order to estimate the right side of the above equality, we investigate that
	$$\aligned
	\int_{\mathbb{R}^{N}}&\Big[\Big(|x|^{-\alpha}\ast |Z_{z_j,\beta}|^{p}\Big)Z_{z_j,\beta}^{p-2}Z_{j,2}Z_{z_j,\beta}dx\\
	&\leq C\int_{\mathbb{R}^{N}}\frac{\beta^{\frac{\alpha}{2}}}{(1+\beta^{2}|x-z_{j} |^{2})^{\frac{\alpha}{2}}}\frac{\beta^{\frac{N+4-\alpha}{2}}}{(1+\beta^{2}|x-z_{j} |^{2})^{\frac{N+4-\alpha}{2}}}\frac{\beta^{\frac{N-4}{2}}\beta^2|x-z_{j} |}
	{(1+\beta^{2}|y-z_{j} |^{2})^{\frac{N-2}{2}}}dx
	=O(\beta),
	\endaligned$$
	where $j=1,\cdot\cdot\cdot,m,$ and if $i\neq j$,
	$$\aligned
	\int_{\mathbb{R}^{N}}&\Big[\Big(|x|^{-\alpha}\ast |Z_{z_j,\beta}|^{p}\Big)Z_{z_j,\beta}^{p-2}Z_{j,2}Z_{z_i,\beta}dx\\
	&\leq C\int_{\mathbb{R}^{N}}\frac{\beta^{\frac{\alpha}{2}}}{(1+\beta^{2}|x-z_{j} |^{2})^{\frac{\alpha}{2}}}\frac{\beta^{\frac{8-\alpha}{2}}}{(1+\beta^{2}|x-z_{j} |^{2})^{\frac{8-\alpha}{2}}}\frac{\beta^{\frac{N-4}{2}}}{(1+\beta^{2}|x-z_{i} |^{2})^{\frac{N-4}{2}}}\frac{\beta^{\frac{N-4}{2}}\beta^2|x-z_{j} |}
	{(1+\beta^{2}|x-z_{j}|^{2})^{\frac{N-2}{2}}}dx\\
	&\leq C\int_{\mathbb{R}^{N}}\frac{\beta}
	{(1+|x-\beta z_{j}|)^{N+5}}\frac{1}{(1+|x-\beta z_{i} |)^{N-4}}dy\\
	&\leq\frac{ C}{\beta^{N-5}|z_{i}-z_{j}|^{N-4}},
	\endaligned$$
	So, for $l=2$, we have
	$$
	\sum_{j=1}^{m}\int_{\mathbb{R}^{N}}\Big(|x|^{-\alpha}\ast |Z_{z_j,\beta}|^{p}\Big)Z_{z_j,\beta}^{p-2}Z_{j,2}Z_{\overline{r},\overline{x}'',\beta}dx
	\leq O(m\beta)+m\sum_{j=2}^{m}\frac{C}{\beta^{N-5}|z_{1}-z_{j}|^{N-4}}
	=O(m\beta).
	$$
	Similarly, we have,
	$$\aligned
	\sum_{j=1}^{m}\int_{\mathbb{R}^{N}}\Big(|x|^{-\alpha}\ast Z_{z_j,\beta}^{p-1}Z_{j,2}\Big)Z_{z_j,\beta}^{p-1}Z_{\overline{r},\overline{x}'',\beta}dx
	=O(m\beta).
	\endaligned$$
	Consequently,
	\begin{equation}\label{eq1}
		\sum_{j=1}^{m}\int_{\mathbb{R}^{N}}\Big[p\Big(|x|^{-\alpha}\ast Z_{z_j,\beta}^{p-1}Z_{j,2}\Big)Z_{z_j,\beta}^{p-1}+(p-1)\Big(|x|^{-\alpha}\ast |Z_{z_j,\beta}|^{p}\Big)Z_{z_j,\beta}^{p-2}Z_{j,2}\Big]Z_{\overline{r},\overline{x}'',\beta}dx=O(m\beta).
	\end{equation}
	
	Using the same arguments as above, we have,
	\begin{equation}\label{eq2}
		\sum_{j=1}^{m}\int_{\mathbb{R}^{N}}\Big[p\Big(|x|^{-\alpha}\ast Z_{z_j,\beta}^{p-1}Z_{j,l}\Big)Z_{z_j,\beta}^{p-1}+(p-1)\Big(|x|^{-\alpha}\ast |Z_{z_j,\beta}|^{p}\Big)Z_{z_j,\beta}^{p-2}Z_{j,l}\Big]Z_{\overline{r},\overline{x}'',\beta}dx=O(m\beta), \ \ l=3,\cdot\cdot\cdot, N,
	\end{equation}
	and
	\begin{equation}\label{eq3}
		\sum_{j=1}^{m}\int_{\mathbb{R}^{N}}\Big[p\Big(|x|^{-\alpha}\ast Z_{z_j,\beta}^{p-1}Z_{j,l}\Big)Z_{z_j,\beta}^{p-1}+(p-1)\Big(|x|^{-\alpha}\ast |Z_{z_j,\beta}|^{p}\Big)Z_{z_j,\beta}^{p-2}Z_{j,l}\Big]Z_{\overline{r},\overline{x}'',\beta}dx=O(\frac{m}{\beta}), \ \ l=1.
	\end{equation}
	
	By some direct calculations, these three identities lead to
	$$
	\sum_{j=1}^{m}\int_{\mathbb{R}^{N}}\Big[p\Big(|x|^{-\alpha}\ast (Z_{z_j,\beta}^{p-1}
	Z_{j,l})\Big)Z_{z_j,\beta}^{p-1}+(p-1)\Big(|x|^{-\alpha}\ast |Z_{z_j,\beta}|^{p}\Big)Z_{z_j,\beta}^{p-2}Z_{j,l}\Big]\frac{\partial Z_{\overline{r},\overline{x}'',\beta}}{\partial \beta} dx=O(m), \ \ l=2,\cdot\cdot\cdot,N,
	$$
	and
	$$
	\sum_{j=1}^{m}\int_{\mathbb{R}^{N}}\Big[p\Big(|x|^{-\alpha}\ast (Z_{z_j,\beta}^{p-1}
	Z_{j,1})\Big)Z_{z_j,\beta}^{p-1}+(p-1)\Big(|x|^{-\alpha}\ast |Z_{z_j,\beta}|^{p}\Big)Z_{z_j,\beta}^{p-2}Z_{j,1}\Big]\frac{\partial Z_{\overline{r},\overline{x}'',\beta}}{\partial \beta} dx=O(\frac{m}{\beta^2}).
	$$
	Then by \eqref{ci}-\eqref{main3} together with \eqref{c14}, we obtain the estimates for $c_l$:
	\begin{equation}\label{d15}
		c_1=O(\frac{1}{\beta^{3+\varepsilon}}), \hspace{0.2cm}c_i=O(\frac{1}{\beta^{5+\varepsilon}}),  i=2,\cdot\cdot\cdot,N.
	\end{equation}

	As a result of \eqref{eq1}-\eqref{d15}, we find that \eqref{d14} is equivalent to
	\begin{equation}\label{d17}
		\aligned
		&\int_{D_{\rho}}(2V(x)+\frac{1}{2}\langle x,\nabla V(x)\rangle) u_{m}^{2}dx\\
		=&O(\frac{m}{\beta^{4+\varepsilon}})+O\Big(\int_{\partial D_{\rho}}\Big(\phi^{2}+\sum_{j=1}^{3}|\nabla^{j}\phi||\nabla^{4-j}\phi|+\sum_{j=0}^{3}|\nabla^{j}\phi||\nabla^{3-j}\phi|\Big)ds
		+\int_{\partial D_{\rho}}\Big(\int_{\R^N} \frac{|\phi(y)|^{p}}{|x-y|^\alpha}dy\Big)|\phi|^{p}ds\Big)\\
		&+O\Big(\int_{ D_{\rho}}\int_{\mathbb{R}^{N}\backslash D_{\rho}}x(x-y)\frac{|\phi(x)|^{p}|\phi(y)|^{p}}{|x-y|^{\alpha+2}}dxdy\Big)+O(\frac{1}{\beta^{4+\varepsilon}}),i=3,\cdot\cdot\cdot,N
		\endaligned\end{equation}
	for some small $\varepsilon>0$.
	
	Integrating by parts to \eqref{d2}, we find it is equivalent to
	\begin{equation}\label{d11}
		\begin{split}
			\int_{D_{\rho}}
			&\frac{\partial V(r,x'')}{\partial x_{i}} u_{m}^{2} dx\\
			=&O\Big(\int_{\partial D_{\rho}}\Big(\phi^{2}+\sum_{j=1}^{3}|\nabla^{j}\phi||\nabla^{4-j}\phi|+\sum_{j=0}^{3}|\nabla^{j}\phi||\nabla^{3-j}\phi|\Big)ds
			+\int_{\partial D_{\rho}}\Big(\int_{\R^N} \frac{|\phi(y)|^{p}}{|x-y|^\alpha}dy\Big)|\phi|^{p}ds\Big)\\
			&+O\Big(\int_{ D_{\rho}}\int_{\mathbb{R}^{N}\backslash D_{\rho}}x(x-y)\frac{|\phi(x)|^{p}|\phi(y)|^{p}}{|x-y|^{\alpha+2}}dxdy\Big)+O(\frac{1}{\beta^{4+\varepsilon}}),
		\end{split}
	\end{equation}
	where $i=3,\cdots,N$. From \eqref{d11}, we can rewrite \eqref{d17} as
	\begin{equation}\label{d18}
		\begin{split}
			&\int_{D_{\rho}}\Big(2V(r,x'')+\frac{r}{2}\frac{\partial V(r,x'')}{\partial r}\Big) u_{m}^{2}dx\\
			=&o(\frac{m}{\beta^{4}})+O\Big(\int_{\partial D_{\rho}}\Big(\phi^{2}+\sum_{j=1}^{3}|\nabla^{j}\phi||\nabla^{4-j}\phi|+\sum_{j=0}^{3}|\nabla^{j}\phi||\nabla^{3-j}\phi|\Big)ds
			+\int_{\partial D_{\rho}}\Big(\int_{\R^N} \frac{|\phi(y)|^{p}}{|x-y|^\alpha}dy\Big)|\phi|^{p}ds\Big)\\
			&+O\Big(\int_{ D_{\rho}}\int_{\mathbb{R}^{N}\backslash D_{\rho}}x(x-y)\frac{|\phi(x)|^{p}|\phi(y)|^{p}}{|x-y|^{\alpha+2}}dxdy\Big)+O(\frac{1}{\beta^{4+\varepsilon}}),
		\end{split}
	\end{equation}
    that is
	\begin{equation}\label{d19}
		\begin{split}
			&\int_{D_{\rho}}\Big(\frac{1}{2r^3}\frac{\partial (r^4V(r,x''))}{\partial r}\Big) u_{m}^{2}dx\\
			=&o(\frac{m}{\beta^{4}})+O\Big(\int_{\partial D_{\rho}}\Big(\phi^{2}+\sum_{j=1}^{3}|\nabla^{j}\phi||\nabla^{4-j}\phi|+\sum_{j=0}^{3}|\nabla^{j}\phi||\nabla^{3-j}\phi|\Big)ds
			+\int_{\partial D_{\rho}}\Big(\int_{\R^N} \frac{|\phi(y)|^{p}}{|x-y|^\alpha}dy\Big)|\phi|^{p}ds\Big)\\
			&+O\Big(\int_{ D_{\rho}}\int_{\mathbb{R}^{N}\backslash D_{\rho}}x(x-y)\frac{|\phi(x)|^{p}|\phi(y)|^{p}}{|x-y|^{\alpha+2}}dxdy\Big)+O(\frac{1}{\beta^{4+\varepsilon}}),
		\end{split}
	\end{equation}
	From (4.39) in\cite{GLN}, we know 
	$$
	\int_{D_{4\delta}\backslash D_{3\delta}}\Big(|\phi|^{2}+\sum_{j=1}^{3}|\nabla^{j}\phi||\nabla^{4-j}\phi|+\sum_{j=0}^{3}|\nabla^{j}\phi||\nabla^{3-j}\phi|\Big)=O(\frac{m}{\beta^{4+\varepsilon}}).
	$$
	Together with Lemma \ref{D4}, we obtain
	$$\aligned
	\int_{D_{4\delta}\backslash D_{3\delta}}&
	\Big(|\phi|^{2}+\sum_{j=1}^{3}|\nabla^{j}\phi||\nabla^{4-j}\phi|+\sum_{j=0}^{3}|\nabla^{j}\phi||\nabla^{3-j}\phi|\Big)dx
	+\int_{D_{4\delta}\backslash D_{3\delta}}\Big(\int_{D_{\rho}} \frac{|\phi(y)|^{p}}{|x-y|^\alpha}dy\Big)|\phi|^{p}dx\\
	&+\int_{ D_{4\delta}\backslash D_{3\delta}}\int_{\R^N\backslash D_{\rho}}x(x-y)\frac{|\phi(x)|^{p}|\phi(y)|^{p}}{|x-y|^{\alpha+2}}dxdy
	=O(\frac{m}{\beta^{4+\varepsilon}}),
	\endaligned$$
	where $i=3,\cdot\cdot\cdot,N$. As a result, there exists a $\rho\in(3\delta,4\delta)$ such that
	$$\aligned
	\int_{\partial D_{\rho}}&\Big(|\phi|^{2}+\sum_{j=1}^{3}|\nabla^{j}\phi||\nabla^{4-j}\phi|+\sum_{j=0}^{3}|\nabla^{j}\phi||\nabla^{3-j}\phi|\Big)ds
	+\int_{\partial D_{\rho}}\Big(\int_{D_{\rho}} \frac{|\phi(y)|^{p}}{|x-y|^\alpha}dy\Big)|\phi|^{p}ds\\
	&+\int_{ D_{\rho}}\int_{\R^N\backslash D_{\rho}}x(x-y)\frac{|\phi(x)|^{p}|\phi(y)|^{p}}{|x-y|^{\alpha+2}}dxdy=O(\frac{m}{\beta^{4+\varepsilon}}),
	\endaligned$$
	where $i=3,\cdot\cdot\cdot,N$. By Lemma 4.7 in \cite{GLN}, for any function $g(r, x'')\in C^1(\mathbb{R^N})$, it holds
	$$
	\int_{D_{\rho}}g(r,x'') u_{m}^{2}dx=m(\frac{1}{\beta^{4}}g(\overline{r},\overline{x}'')\int_{\mathbb{R}^{N}}W_{0,1}^{2}dx+o(\frac{1}{\beta^{4}})).
	$$
	So, we can obtain from \eqref{d11} and \eqref{d19} that
	$$
	m\Big(\frac{1}{\beta^4}\frac{\partial  V(\overline{r},\overline{x}'')}{\partial \overline{x}_{i}}\int_{\mathbb{R}^{N}}W_{0,1}^{2}dx+o(\frac{1}{\beta^{4}})\Big)=o(\frac{m}{\beta^{4}}),
	$$
	and
	$$
	m\Big(\frac{1}{\beta^4}\frac{1}{2\overline{r}^3}\frac{\partial (\overline{r}^4V(\overline{r},\overline{x}''))}{\partial \overline{r}}\int_{\mathbb{R}^{N}}W_{0,1}^{2}dx+o(\frac{1}{\beta^{4}})\Big)=o(\frac{m}{\beta^{4}}).
	$$
	While the function $V$ is bounded, we find the equations to determine $(\overline{r},\overline{x}'')$ are
	\begin{equation}\label{d21}
		\frac{\partial (\overline{r}^4V(\overline{r},\overline{x}''))}{\partial \overline{r}}=o(1).
	\end{equation}
	and
	\begin{equation}\label{d22}
		\frac{\partial (\overline{r}^4V(\overline{r},\overline{x}''))}{\partial \overline{x}_{i}}=o(1),  i=3,\cdot\cdot\cdot,N,
	\end{equation}
	
	\noindent
	{\bf Proof of Theorem \ref{Ms}.}  We have proved that \eqref{d1}, \eqref{d2} and \eqref{d3} are equivalent to \begin{equation}\label{d21}
		\frac{\partial (\overline{r}^4V(\overline{r},\overline{x}''))}{\partial \overline{r}}=o(1).
	\end{equation}
	
	\begin{equation}\label{d22}
		\frac{\partial (\overline{r}^4V(\overline{r},\overline{x}''))}{\partial \overline{x}_{i}}=o(1),  i=3,\cdot\cdot\cdot,N,
	\end{equation}
	and
	$$
	-\frac{A_1}{\beta^{5}}V(\overline{r},\overline{x}'')+\frac{A_3 m^{N-4}}{\beta^{N-3}}=O(\frac{1}{\beta^{5+\varepsilon}}).
	$$
	Let $\beta=tm^{\frac{N-4}{N-8}}$, then $t\in[L_0, L_1]$. Therefore, we can get
	\begin{equation}\label{d23}
		-\frac{A_1}{t^{5}}V(\overline{r},\overline{x}'')+\frac{A_3}{t^{N-3}}=o(1), t\in[L_0, L_1].
	\end{equation}
	Let
	$$
	F(t, \overline{r},\overline{x}'')=\big(\nabla_{\overline{r},\overline{x}''}(\overline{r}^4V(\overline{r},\overline{x}''),
	-\frac{A_1}{t^{5}}+\frac{A_3}{t^{N-3}}\big).
	$$
	Then
	$$
	\mbox{deg}(F(t, \overline{r},\overline{x}''), [L_0, L_1]\times B_{\vartheta}((r_{0},x_{0}'')))
	=-\mbox{deg}(\nabla_{\overline{r},\overline{x}''}(\overline{r}^4V(\overline{r},\overline{x}''),B_{\vartheta}((r_{0},x_{0}'')))\neq0.
	$$
	So, \eqref{d21}, \eqref{d22} and \eqref{d23} have a solution $t_{m}\in[L_0, L_1]$, $(\overline{r}_{m},\overline{x}_{m}'')\in B_{\vartheta}((r_{0},x_{0}''))$.
	$\hfill{} \Box$

	\vspace{1cm}

	\begin{appendix}
		\section{The basic estimates}
		\begin{lem}\label{B2} (Lemma B.1, \cite{WY1}) For each fixed $k$ and $j$, $k\neq j$,  let
			$$
			g_{k,j}(x)=\frac{1}{(1+|x-z_{j}|)^{\alpha}}\frac{1}{(1+|x-z_{k}|)^{\beta}},
			$$
			where $\alpha\geq 1$ and $\beta\geq1$ are two constants.
			Then, for any constants $0<\delta\leq\min\{\alpha,\beta\}$, there is a constant $C>0$, such that
			$$
			g_{k,j}(x)\leq\frac{C}{|z_{k}-z_{j}|^{\delta}}\Big(\frac{1}{(1+|x-z_{j}|)^{\alpha+\beta-\delta}}+\frac{1}{(1+|x-z_{k}|)^{\alpha+\beta-\delta}}\Big).
			$$
		\end{lem}
		
		\begin{lem}\label{B3} (Lemma 2.2, \cite{GL}) For any constant $0<\delta<N-4$, $N\geq9$, there is a constant $C>0$, such that
			$$
			\int_{\mathbb{R}^{N}}\frac{1}{|x-y|^{N-4}}\frac{1}{(1+|y|)^{4+\delta}}dy\leq \frac{C}{(1+|x|)^{\delta}}.
			$$
		\end{lem}
		
		\begin{lem}\label{B4}
			For $N\geq 9$ and $1\leq i\leq m$, there is a constant $C>0$, such that
			$$
			|x|^{-\alpha}\ast \frac{\beta^{N-\frac{\alpha}{2}}}{(1+\beta|x-z_{i}|)^{\frac{3N+4}{2}-\alpha+\eta}}\leq \frac{C\beta^{\frac{\alpha}{2}}}{(1+\beta|x-z_{i}|)^{\min\{\alpha,\frac{N+4}{2}\}}},
			$$
			where $\eta>0$.
			
		\end{lem}
		\begin{proof}
			Firstly, we notice that
			$$\aligned
			|x|^{-\alpha}\ast \frac{\beta^{N-\frac{\alpha}{2}}}{(1+\beta|x-z_{i}|)^{\frac{3N+4}{2}-\alpha+\eta}}=&
			\int_{\mathbb{R}^{N}}\frac{1}{|y|^{\alpha}}\frac{\beta^{N-\frac{\alpha}{2}}}{(1+\beta|x-z_{i}-y|)^{\frac{3N+4}{2}-\alpha+\eta}}dy\\
			=&
			\int_{\mathbb{R}^{N}}\frac{1}{|y|^{\alpha}}\frac{\beta^{\frac{\alpha}{2}}}{(1+|\beta x-\beta z_{i}-y|)^{\frac{3N+4}{2}-\alpha+\eta}}dy.
			\endaligned$$
			
			Let $d=\frac{\beta}{2}|x-z_{i}|>1$. Then, we have
			$$\aligned
			\int_{B_{d}(0)}\frac{1}{|y|^{\alpha}}\frac{1}{(1+|\beta x-\beta z_{i}-y|)^{\frac{3N+4}{2}-\alpha+\eta}}dy&\leq \frac{C}{(1+ d)^{\frac{3N+4}{2}-\alpha+\eta}}\int_{B_{d}(0)}\frac{1}{|y|^{\alpha}}dy\\
			&\leq \frac{C d^{N-\alpha}}{(1+ d)^{\frac{3N+4}{2}-\alpha+\eta}}\leq\frac{ C}{(1+ d)^{\frac{N+4}{2}+\eta}},
			\endaligned$$
			and
			$$\aligned
			&\int_{B_{d}(\beta x-\beta z_{i})}\frac{1}{|y|^{\alpha}}\frac{1}{(1+|\beta x-\beta z_{i}-y|)^{\frac{3N+4}{2}-\alpha+\eta}}dy\\
			&\hspace{5mm}\leq \frac{1}{d^{\alpha}}\int_{B_{d}(0)}\frac{1}{(1+|y|)^{\frac{3N+4}{2}-\alpha+\eta}}dy
			\leq \frac{C}{(1+ d)^{\min\{\alpha, \frac{N+4}{2}\}}}.
			\endaligned$$

			Suppose that $y\in\mathbb{R}^{N}\backslash (B_{d}(0)\cup B_{d}(\beta x-\beta z_{i}))$. Then
			$$
			|\beta x-\beta z_{i}-y|\geq\frac{1}{2}|\beta x-\beta z_{i}|, |y|\geq\frac{1}{2}|\beta x-\beta z_{i}|,
			$$
			and we have
			$$
			\frac{1}{|y|^{\alpha}}\frac{1}{(1+|\beta x-\beta z_{i}-y|)^{\frac{3N+4}{2}-\alpha+\eta}}\leq \frac{ C}{(1+ d)^{\frac{N+4}{2}}}\frac{1}{|y|^{\alpha}}\frac{1}{(1+|\beta x-\beta z_{i}-y|)^{N-\alpha+\eta}}.
			$$
			If $|y|\leq2|\beta x-\beta z_{i}|$, then
			$$
			\frac{1}{|y|^{\alpha}}\frac{1}{(1+|\beta x-\beta z_{i}-y|)^{N-\alpha+\eta}}\leq
			\frac{C}{|y|^{\alpha}(1+|\beta x-\beta z_{i}|)^{N-\alpha+\eta}}\leq
			\frac{C_{1}}{|y|^{\alpha}(1+|y|)^{N-\alpha+\eta}}.
			$$
			If $|y|\geq 2|\beta x-\beta z_{i}|$, then $|\beta x-\beta z_{i}-y|\geq|y|-|\beta x-\beta z_{i}|\geq\frac{1}{2}|y|$. As a result,
			$$
			\frac{1}{|y|^{\alpha}}\frac{1}{(1+|\beta x-\beta z_{i}-y|)^{N-\alpha+\eta}}\leq
			\frac{C}{|y|^{\alpha}(1+|y|)^{N-\alpha+\eta}}.
			$$
			Thus, we have
			$$\aligned
			&\int_{\mathbb{R}^{N}\backslash (B_{d}(0)\cup B_{d}(\beta x-\beta z_{i}))}\frac{1}{|y|^{\alpha}}\frac{1}{(1+|\beta x-\beta z_{i}-y|)^{\frac{3N+4}{2}-\alpha+\eta}}dy\\
			&\leq \frac{ C}{(1+ d)^{\frac{N+4}{2}}}\int_{\mathbb{R}^{N}}\frac{1}{|y|^{\alpha}}\frac{1}{(1+|y|)^{N-\alpha+\eta}}dy
			\leq \frac{ C}{(1+ d)^{\frac{N+4}{2}}}.
			\endaligned$$
		\end{proof}
		Using \eqref{REL} and the identity (see (37) in \cite{DHQWF} for example)
		\begin{equation}\label{eq3.18}
			\int_{\R^N}\frac{1}{|x-y|^{2s}}\Big(\frac{1}{1+|y|^{2}}\Big)^{N-s}dy
			=I(s)\Big(\frac{1}{1+|x|^{2}}\Big)^{s},\ \ 0 < s < \frac{N}{2},
		\end{equation}
		where
		$$
		I(s)=\frac{\pi^{\frac{N}{2}}\Gamma(\frac{N-2s}{2})}{\Gamma(N-s)}, \ \ \mbox{and }\Gamma(s)=\int_0^{+\infty} x^{s-1}e^{-x}\,dx, s>0,
		$$
		we have
		$$
		|x|^{-\alpha}\ast |W_{z,\beta}(x)|^{p}
		=\int_{\R^N}\frac{W_{z,\beta}^p(y)}{|x-y|^{\alpha}}dy
		=C\Big(\frac{\beta}{1+\beta^2|x-z|^{2}}\Big)^{\frac{\alpha}{2}},
		$$
		where $N\geq9$ and \text{ $C=I(\frac{\alpha}{2})C^{p}(N,\alpha)$}. So, we have the following lemma.
		\begin{lem}\label{P0}
			For $N\geq 9$ and $1\leq i\leq m$, there is a constant $C>0$, such that
			$$
			|x|^{-\alpha}\ast |W_{z_{i},\beta}(x)|^{p}
			=C\Big(\frac{\beta}{1+\beta^2|x-z|^{2}}\Big)^{\frac{\alpha}{2}}.
			$$
		\end{lem}

		\vspace{1cm}
		\section{The energy expansion}
		Here, we will give energy estimates in this section. The nonlocal interaction parts cause some difficulties. So we need to establish some estimates which give suitable order of small terms. The functional corresponding to \eqref{eq1} is
		$$
		J(u)=\frac{1}{2}\int_{\mathbb{R}^{N}}\Big(\Delta^2 u+V(r,x'')u^2\Big)dx
		-\frac{1}{2p}
		\int_{\mathbb{R}^N}\int_{\mathbb{R}^N}
		\frac{|u(x)|^{p}|u(y)|^{p}}{|x-y|^{\alpha}}dxdy.
		$$
		
		\begin{lem}\label{P3}
			We have
			$$
			\frac{\partial J(Z_{\overline{r},\overline{x}'',\beta})}{\partial \beta}=\frac{\partial J(Z_{\overline{r},\overline{x}'',\beta}^{\ast})}{\partial \beta}+O(\frac{m}{\beta^{5+\varepsilon}}),
			$$
		\end{lem}
		\begin{proof}
			We observe that
			$$\aligned
			\frac{\partial J(Z_{\overline{r},\overline{x}'',\beta}^{\ast})}{\partial \beta}-\frac{\partial J(Z_{\overline{r},\overline{x}'',\beta})}{\partial \beta}
			=&\int_{\mathbb{R}^{N}}\Delta^2 Z_{\overline{r},\overline{x}'',\beta}^{\ast}\frac{\partial Z_{\overline{r},\overline{x}'',\beta}^{\ast}}{\partial \beta}dx-\int_{\mathbb{R}^{N}}\Delta^2 Z_{\overline{r},\overline{x}'',\beta}\frac{\partial Z_{\overline{r},\overline{x}'',\beta}}{\partial \beta}dx\\
			&+\int_{\mathbb{R}^{N}}V(x)\Big(Z_{\overline{r},\overline{x}'',\beta}^{\ast}\frac{\partial Z_{\overline{r},\overline{x}'',\beta}^{\ast}}{\partial \overline{r}}-Z_{\overline{r},\overline{x}'',\beta}^{\ast}\frac{\partial Z_{\overline{r},\overline{x}'',\beta}^{\ast}}{\partial \overline{r}}\Big)dx\\
			&+\Big(\int_{\mathbb{R}^N}\int_{\mathbb{R}^N}
			\frac{|Z_{\overline{r},\overline{x}'',\beta}^{\ast}(x)|^{p}
				|Z_{\overline{r},\overline{x}'',\beta}^{\ast}(y)|^{p-1}\frac{\partial Z_{\overline{r},\overline{x}'',\beta}^{\ast}}{\partial \beta}(y)}{|x-y|^{\alpha}}dxdy\\
				&\hspace{5mm}-\int_{\mathbb{R}^N}\int_{\mathbb{R}^N}
			\frac{|Z_{\overline{r},\overline{x}'',\beta}(x)|^{p}
				Z^{p-1}_{\overline{r},\overline{x}'',\beta}(y)\frac{\partial Z_{\overline{r},\overline{x}'',\beta}}{\partial \beta}(y)}{|x-y|^{\alpha}}dxdy\Big)
			\endaligned$$
						It is easy to check that
			$$\aligned
			&\int_{\mathbb{R}^N}\int_{\mathbb{R}^N}
			\frac{|W_{z_j,\beta}(x)|^{p}
				W^{p-1}_{z_j,\beta}(y)\frac{\partial W_{z_j,\beta}}{\partial \beta}(y)}{|x-y|^{\alpha}}dxdy-\int_{\mathbb{R}^N}\int_{\mathbb{R}^N}
			\frac{|Z_{z_j,\beta}(x)|^{p}
				Z^{p-1}_{z_j,\beta}(y)\frac{\partial Z_{z_j,\beta}}{\partial \beta}(y)}{|x-y|^{\alpha}}dxdy\\
			=&\int_{\mathbb{R}^N}\int_{\mathbb{R}^N}
			\frac{(1-\xi^{p}(x))|W_{z_j,\beta}(x)|^{p}
				W^{p-1}_{z_j,\beta}(y)\frac{\partial W_{z_j,\beta}}{\partial \beta}(y)}{|x-y|^{\alpha}}dxdy\\
				&+\int_{\mathbb{R}^N}\int_{\mathbb{R}^N}
			\frac{|\xi W_{z_j,\beta}(x)|^{p}
				(1-\xi^{p})W^{p-1}_{z_j,\beta}(y)\frac{\partial W_{z_j,\beta}}{\partial \beta}(y)}{|x-y|^{\alpha}}dxdy,
			\endaligned$$
			where $j=1,2, \cdot\cdot\cdot,m$. By the Hardy-Littlewood-Sobolev inequality and a direct calculation, we have
			$$\aligned
			&\left|\int_{\mathbb{R}^N}\int_{\mathbb{R}^N}
			\frac{(1-\xi^{p}(x))|W_{z_j,\beta}(x)|^{p}
			W^{p-1}_{z_j,\beta}(y)\frac{\partial W_{z_j,\beta}}{\partial \beta}(y)}{|x-y|^{\alpha}}dxdy\right|\\ \leq&C\int_{\mathbb{R}^{N}}\int_{\mathbb{R}^{N}}\frac{(1-\xi^{p}(x+ z_{j}))\beta^{N-\frac{\alpha}{2}}}{(1+\beta^{2}|x |^{2})^{N-\frac{\alpha}{2}}}\frac{1}{|x-y|^{\alpha}}\frac{\beta^{N-\frac{\alpha}{2}-1}}
			{(1+\beta^{2}|y|^{2})^{N-\frac{\alpha}{2}}}dxdy\\
			\leq &C\frac{1}{\beta} \left(\int_{\mathbb{R}^{N}}\left[\frac{(1-\xi^{p}(x+z_{j}))^{\frac{1}{2}}\beta^{N-\frac{\alpha}{2}}}{(1+\beta^{2}|x |^{2})^{N-\frac{\alpha}{2}}}\right]^{\frac{N}{N-\frac{\alpha}{2}}}dx\right)^{\frac{2N-\alpha}{N}}
			=O(\frac{1}{\beta^{2N+1-\alpha}}),
			\endaligned$$
			where the integral comes from the direct calculation
			\begin{equation}\label{z1}
				\aligned
				&\int_{\mathbb{R}^{N}}\left[\frac{(1-\xi^{p}(x+z_{j}))^{\frac{1}{2}}\beta^{N-\frac{\alpha}{2}}}{(1+\beta^{2}|x |^{2})^{N-\frac{\alpha}{2}}}\right]^{\frac{N}{N-\frac{\alpha}{2}}}dx\leq C \int_{\mathbb{R}^{N}\setminus B_{\delta-\vartheta} (0)}\frac{\beta^{N}}{(1+\beta^{2}|x |^{2})^{N}}dx
				=O(\frac{1}{\beta^{N}})
				\endaligned
			\end{equation}
			And we can also check that
			$$
			\int_{\mathbb{R}^N}\int_{\mathbb{R}^N}
			\frac{|\xi W_{z_j,\beta}(x)|^{p}
				(1-\xi^{p}(y))W_{z_j,\beta}(y)\frac{\partial W_{z_j,\beta}}{\partial \beta}(y)}{|x-y|^{\alpha}}dxdy=O(\frac{1}{\beta^{2N+1-\alpha}}).
			$$
			Thus, we have
			$$\aligned
			\int_{\mathbb{R}^N}\int_{\mathbb{R}^N}
			&\frac{|W_{z_j,\beta}(x)|^{p}
				W^{p-1}_{z_j,\beta}(y)\frac{\partial W_{z_j,\beta}}{\partial \beta}(y)}{|x-y|^{\alpha}}dxdy-\int_{\mathbb{R}^N}\int_{\mathbb{R}^N}
			\frac{|Z_{z_j,\beta}(x)|^{p}
				Z^{p-1}_{z_j,\beta}(y)\frac{\partial Z_{z_j,\beta}}{\partial \beta}(y)}{|x-y|^{\alpha}}dxdy\\
			&=O(\frac{1}{\beta^{2N+1-\alpha}}),
			\endaligned$$
			where $j=1,2, \cdot\cdot\cdot,m$.  
		    Thus, we have
			\begin{equation}\label{BB2}
				\aligned
				&\int_{\mathbb{R}^N}\int_{\mathbb{R}^N}
				\frac{|Z_{\overline{r},\overline{x}'',\beta}^{\ast}(x)|^{p}
					|Z_{\overline{r},\overline{x}'',\beta}^{\ast}(y)|^{p-1}\frac{\partial Z_{\overline{r},\overline{x}'',\beta}^{\ast}}{\partial \beta}(y)}{|x-y|^{\alpha}}dxdy-\int_{\mathbb{R}^N}\int_{\mathbb{R}^N}
				\frac{|Z_{\overline{r},\overline{x}'',\beta}(x)|^{p}
					Z^{p-1}_{\overline{r},\overline{x}'',\beta}(y)\frac{\partial Z_{\overline{r},\overline{x}'',\beta}}{\partial \beta}(y)}{|x-y|^{\alpha}}dxdy\\
				&=O(\frac{m^{2p}}{\beta^{2N+1-\alpha}})
				=O(\frac{m}{\beta^{5+\varepsilon}}).
				\endaligned
			\end{equation}
			Refer Lemma 3.1 in \cite{GLN}, we have 
			\begin{equation}\label{BB3}
				\aligned
				\int_{\mathbb{R}^{N}}&\Delta^2 Z_{\overline{r},\overline{x}'',\beta}^{\ast}\frac{\partial Z_{\overline{r},\overline{x}'',\beta}^{\ast}}{\partial \beta}dx-\int_{\mathbb{R}^{N}}\Delta^2 Z_{\overline{r},\overline{x}'',\beta}\frac{\partial Z_{\overline{r},\overline{x}'',\beta}}{\partial \beta}dx\\
				&+\int_{\mathbb{R}^{N}}V(x)\Big(Z_{\overline{r},\overline{x}'',\beta}^{\ast}\frac{\partial Z_{\overline{r},\overline{x}'',\beta}^{\ast}}{\partial \overline{r}}-Z_{\overline{r},\overline{x}'',\beta}^{\ast}\frac{\partial Z_{\overline{r},\overline{x}'',\beta}^{\ast}}{\partial \overline{r}}\Big)dx=O(\frac{m}{\beta^{N-3}})=O(\frac{m}{\beta^{5+\varepsilon}})
				\endaligned
			\end{equation}
			Combined \eqref{BB2} and \eqref{BB3}, we finish the proof.
		\end{proof}

		\begin{lem}\label{D2}
			We have
			$$\aligned
			\frac{\partial J(Z_{\overline{r},\overline{x}'',\beta})}{\partial \beta}=&m\Big(-\frac{A_{1}}{\beta^{5}}V(\overline{r},\overline{x}'')+\sum_{j=2}^{m}
			\frac{A_{2}}{\beta^{N-3}|z_{1}-z_{j}|^{N-4}}+O(\frac{1}{\beta^{5+\varepsilon}})\Big)\\
			=&m\Big(-\frac{A_{1}}{\beta^{5}}V(\overline{r},\overline{x}'')+
			\frac{A_{3}m^{N-4}}{\beta^{N-3}}+O(\frac{1}{\beta^{5+\varepsilon}})\Big),
			\endaligned$$
			where $A_j$, $j=1, 2$ are some positive constants.
		\end{lem}
		\begin{proof}
			Direct calculation shows
			\begin{equation}\label{ExpenL4.4}
				\aligned
				\frac{\partial J(Z_{\overline{r},\overline{x}'',\beta}^{\ast})}{\partial \beta}
				=&\int_{\mathbb{R}^{N}}V(x)Z_{\overline{r},\overline{x}'',\beta}^{\ast}\frac{\partial Z_{\overline{r},\overline{x}'',\beta}^{\ast}}{\partial \beta}dx-\Big[\int_{\mathbb{R}^{N}}\Big(|x|^{-\alpha}\ast |Z_{\overline{r},\overline{x}'',\beta}^{\ast}|^{p}\Big)|Z_{\overline{r},\overline{x}'',\beta}^{\ast}|^{p-1}
				\frac{\partial Z_{\overline{r},\overline{x}'',\beta}^{\ast}}{\partial \beta}dx\\
				&-\sum_{j=1}^{m}\int_{\mathbb{R}^{N}}\Big(|x|^{-\alpha}\ast |W_{z_j,\beta}|^{p}\Big)W_{z_j,\beta}^{p-1}\frac{\partial Z_{\overline{r},\overline{x}'',\beta}^{\ast}}{\partial \beta}dx\Big].
				\endaligned
			\end{equation}
			For the first part, by Lemma 3.1 in\cite{GLN}, we have
			\begin{equation}\label{5.0}
				\aligned
				\int_{\mathbb{R}^{N}}V(x)Z_{\overline{r},\overline{x}'',\beta}^{\ast}\frac{\partial Z_{\overline{r},\overline{x}'',\beta}^{\ast}}{\partial \beta}dx=m\Big(-\frac{A_{1}}{\beta^{5}}V(\overline{r},\overline{x}'')+O(\frac{1}{\beta^{5+\varepsilon}})\Big).
				\endaligned
			\end{equation}
			We observe that
			\begin{equation}\label{5.1}
				\aligned
				&\int_{\mathbb{R}^{N}}\Bigg[\Big(|x|^{-\alpha}\ast |Z_{\overline{r},\overline{x}'',\beta}^{\ast}|^{p}\Big)|Z_{\overline{r},\overline{x}'',\beta}^{\ast}|^{p-1}-\sum_{j=1}^{m}\Big(|x|^{-\alpha}\ast |W_{z_j,\beta}|^{p}\Big)W_{z_j,\beta}^{p-1}\Bigg] \frac{\partial Z_{\overline{r},\overline{x}'',\beta}^{\ast}}{\partial \beta}dx\\
				=&m\int_{\Omega_{1}}\Bigg[\Big(|x|^{-\alpha}\ast |Z_{\overline{r},\overline{x}'',\beta}^{\ast}|^{p}\Big)|Z_{\overline{r},\overline{x}'',\beta}^{\ast}|^{p-1}-\sum_{j=1}^{m}\Big(|x|^{-\alpha}\ast |W_{z_j,\beta}|^{p}\Big)W_{z_j,\beta}^{p-1}\Bigg]\frac{\partial Z_{\overline{r},\overline{x}'',\beta}^{\ast}}{\partial \beta}dx\\
				=& m\Bigg(\int_{\Omega_{1}}\Bigg[p\Big(|x|^{-\alpha}\ast |W_{z_1,\beta}^{p-1}\sum_{i=2}^{m}W_{z_i,\beta}|\Big)W_{z_1,\beta}^{p-1}+(p-1)\Big(|x|^{-\alpha}\ast |W_{z_1,\beta}|^{p}\Big)W_{z_1,\beta}^{p-2}\sum_{i=2}^{m}W_{z_i,\beta}\Bigg]\frac{\partial W_{z_1,\beta}}{\partial \beta}dx\\
				&+O(\frac{1}{\beta^{5+\varepsilon}})\Bigg)\\
				\endaligned\end{equation}

			Applying Lemmas \ref{P0}, we estimate the second term as follows
			$$\aligned
			&\int_{\Omega_{1}}\Big[\Big(|x|^{-\alpha}\ast |W_{z_1,\beta}|^{p}\Big)W_{z_1,\beta}^{p-2}W_{z_i,\beta}\Big]\frac{\partial W_{z_1,\beta}}{\partial \beta}dx\\
			=&-C\Big[\int_{\Omega_{1}}\frac{\beta^{\frac{\alpha}{2}}}{(1+\beta^{2}|x-z_{1} |^{2})^{\frac{\alpha}{2}}}\frac{\beta^{\frac{8-\alpha}{2}}}{(1+\beta^{2}|x-z_{1} |^{2})^{\frac{8-\alpha}{2}}}\frac{\beta^{\frac{N-4}{2}}}{(1+\beta^{2}|x-z_{i} |^{2})^{\frac{N-4}{2}}}\frac{\beta^\frac{N-6}{2}}{(1+\beta^{2}|x-z_{1} |^{2})^{\frac{N-4}{2}}}dx\\
			&-\int_{\Omega_{1}}\frac{\beta^{\frac{\alpha}{2}}}{(1+\beta^{2}|x-z_{1} |^{2})^{\frac{\alpha}{2}}}\frac{\beta^{\frac{8-\alpha}{2}}}{(1+\beta^{2}|x-z_{1} |^{2})^{\frac{8-\alpha}{2}}}\frac{\beta^{\frac{N-4}{2}}}{(1+\beta^{2}|x-z_{i} |^{2})^{\frac{N-4}{2}}}\frac{2\beta^\frac{N-6}{2}}{(1+\beta^{2}|x-z_{1} |^{2})^{\frac{N-2}{2}}}dx\Big]\\
			=&-C\Big[\int_{\Omega_{1}}\frac{\beta^{\frac{N+2}{2}}}{(1+\beta^{2}|x-z_{1} |^{2})^{\frac{N+4}{2}}}\frac{\beta^{\frac{N-4}{2}}}{(1+\beta^{2}|x-z_{i} |^{2})^{\frac{N-4}{2}}}dx\\
			&\hspace{8mm}-\int_{\Omega_{1}}\frac{2\beta^{\frac{N+2}{2}}}{(1+\beta^{2}|x-z_{1} |^{2})^{\frac{N+6}{2}}}\frac{\beta^{\frac{N-4}{2}}}{(1+\beta^{2}|x-z_{i} |^{2})^{\frac{N-4}{2}}}dx\Big]\\
			=&-\frac{C}{\beta^{N-3}|z_{1}-z_{i}|^{N-4}}.
			\endaligned$$
			where $i=2,\cdot\cdot\cdot, m$, Consequently,
			\begin{equation}\label{5.2}
				\aligned
				\sum_{i=2}^{m}\int_{\Omega_{1}}\Big[\Big(|x|^{-\alpha}\ast |W_{z_1,\beta}|^{p}\Big)W_{z_1,\beta}^{p-2}W_{z_i,\beta}\Big]\frac{\partial W_{z_1,\beta}}{\partial \beta}dx
				=-\sum_{i=2}^{m}\frac{C}{\beta^{N-3}|z_{1}-z_{i}|^{N-4}}
				\endaligned\end{equation}

			Notice that the former term can split into  
			$$\aligned
		    &\int_{\Omega_{1}}\Big(|x|^{-\alpha}\ast |W_{z_1,\beta}^{p-1}W_{z_i,\beta}|\Big)W_{z_1,\beta}^{p-1}\frac{\partial W_{z_1,\beta}}{\partial \beta}dx\\
		    =&\frac{1}{p}\Big[\frac{\partial }{\partial \beta}\int_{\Omega_{1}}\Big(|y|^{-\alpha}\ast |W_{z_1,\beta}|^{p}\Big) W_{z_1,\beta}^{p-1}(y)W_{z_i,\beta}(y)dy
		    -\int_{\Omega_{1}}\Big(|y|^{-\alpha}\ast |W_{z_1,\beta}|^{p}\Big)\frac{\partial W_{z_1,\beta}^{p-1}(y)}{\partial \beta}W_{z_i,\beta}(y)dy\\
		    &-\int_{\Omega_{1}}\Big(|y|^{-\alpha}\ast |W_{z_1,\beta}|^{p}\Big)W_{z_1,\beta}^{p-1}(y)\frac{\partial W_{z_i,\beta}}{\partial \beta}(y)dy\Big].
			\endaligned$$
			Then we have
			$$\aligned
			&\int_{\Omega_{1}}\Big(|x|^{-\alpha}\ast |W_{z_1,\beta}^{p-1}W_{z_i,\beta}|\Big)W_{z_1,\beta}^{p-1}\frac{\partial W_{z_1,\beta}}{\partial \beta}dx\\
			=&C\frac{\partial }{\partial \beta}\int_{\Omega_{1}}\frac{\beta^{\frac{\alpha}{2}}}{(1+\beta^{2}|y- z_{1}|^{2})^{\frac{\alpha}{2}}}\frac{\beta^\frac{N+4-\alpha}{2}}{(1+\beta^{2}|y- z_{1}|^{2})^{\frac{N+4-\alpha}{2}}}\frac{\beta^{\frac{N-4}{2}}}{(1+\beta^{2}|y- z_{i}|^{2})^{\frac{N-4}{2}}}dy\\
			&+C\bigg[\int_{\Omega_{1}}\frac{\beta^{\frac{\alpha}{2}}}{(1+\beta^{2}|y- z_{1}|^{2})^{\frac{\alpha}{2}}}\frac{\beta^\frac{N+2-\alpha}{2}}{(1+\beta^{2}|y- z_{1}|^{2})^{\frac{N+4-\alpha}{2}}}\frac{\beta^{\frac{N-4}{2}}}{(1+\beta^{2}|y- z_{i}|^{2})^{\frac{N-4}{2}}}dy\\
			&\hspace{8mm}-\int_{\Omega_{1}}\frac{\beta^{\frac{\alpha}{2}}}{(1+\beta^{2}|y- z_{1}|^{2})^{\frac{\alpha}{2}}}\frac{2\beta^{\frac{N+2-\alpha}{2}}}{(1+\beta^{2}|y- z_{1}|^{2})^{\frac{N+6-\alpha}{2}}}\frac{\beta^{\frac{N-4}{2}}}{(1+\beta^{2}|y- z_{i}|^{2})^{\frac{N-4}{2}}}dy\bigg]\\
			&+C\bigg[\int_{\Omega_{1}}\frac{\beta^{\frac{\alpha}{2}}}{(1+\beta^{2}|y- z_{1}|^{2})^{\frac{\alpha}{2}}}\frac{\beta^{\frac{N+4-\alpha}{2}}}{(1+\beta^{2}|y- z_{1}|^{2})^{\frac{N+4-\alpha}{2}}}\frac{\beta^\frac{N-6}{2}}{(1+\beta^{2}|y- z_{i}|^{2})^{\frac{N-4}{2}}}dy\\
			&\hspace{8mm}-\int_{\Omega_{1}}\frac{\beta^{\frac{\alpha}{2}}}{(1+\beta^{2}|y- z_{1}|^{2})^{\frac{\alpha}{2}}}\frac{\beta^{\frac{N+4-\alpha}{2}}}{(1+\beta^{2}|y- z_{1}|^{2})^{\frac{N+4-\alpha}{2}}}\frac{2\beta^\frac{N-6}{2}}{(1+\beta^{2}|y- z_{i}|^{2})^{\frac{N-2}{2}}}dy\bigg]\\
			=&\frac{\partial }{\partial \beta}\frac{C}{\beta^{N-4}|z_{1}-z_{i}|^{N-4}}+\frac{2C}{\beta^{N-3}|z_{1}-z_{i}|^{N-4}}
			=-\frac{C}{\beta^{N-3}|z_{1}-z_{i}|^{N-4}}.
			\endaligned$$
			So,we obtain
			\begin{equation}\label{5.6}
				\sum_{i=2}^{m}\int_{\Omega_{1}}\Big(|x|^{-\alpha}\ast |W_{z_1,\beta}^{p-1}W_{z_i,\beta}|\Big)W_{z_1,\beta}^{p-1}\frac{\partial W_{z_1,\beta}}{\partial \beta}dx=- \sum_{i=2}^{m}\frac{C}{\beta^{N-3}|z_{1}-z_{i}|^{N-4}}.
			\end{equation}
			Combining \eqref{5.1}-\eqref{5.6}, we can obtain
			$$\aligned
			\int_{\mathbb{R}^{N}}&\Big(|x|^{-\alpha}\ast |Z_{\overline{r},\overline{x}'',\beta}^{\ast}|^{p}\Big)|Z_{\overline{r},\overline{x}'',\beta}^{\ast}|^{p-1}
			\frac{\partial Z_{\overline{r},\overline{x}'',\beta}^{\ast}}{\partial \beta}dx
			-\sum_{j=1}^{m}\int_{\mathbb{R}^{N}}\Big(|x|^{-\alpha}\ast |W_{z_j,\beta}|^{p}\Big)W_{z_j,\beta}^{p-1}\frac{\partial Z_{\overline{r},\overline{x}'',\beta}^{\ast}}{\partial \beta}dx\\
			&=m\Big(-\sum_{j=2}^{m}
			\frac{A_{2}}{\beta^{N-3}|z_{1}-z_{j}|^{N-4}}+O(\frac{1}{\beta^{5+\varepsilon}})\Big).
			\endaligned$$
			So, we get
			\begin{equation}\label{5.11}
				\aligned
				&\sum_{j=1}^{m}\int_{\mathbb{R}^{N}}\Big(|x|^{-\alpha}\ast |W_{z_j,\beta}|^{p}\Big)W_{z_j,\beta}^{p-1}\frac{\partial Z_{\overline{r},\overline{x}'',\beta}^{\ast}}{\partial \beta}dx-\int_{\mathbb{R}^{N}}\Big(|x|^{-\alpha}\ast |Z_{\overline{r},\overline{x}'',\beta}^{\ast}|^{p}\Big)|Z_{\overline{r},\overline{x}'',\beta}^{\ast}|^{p-1}
				\frac{\partial Z_{\overline{r},\overline{x}'',\beta}^{\ast}}{\partial \beta}dx\\
				&=m\Big(\sum_{j=2}^{m}
				\frac{A_{2}}{\beta^{N-3}|z_{1}-z_{j}|^{N-4}}+O(\frac{1}{\beta^{5+\varepsilon}})\Big).
				\endaligned
			\end{equation}
			
			Combined \eqref{5.0}, \eqref{5.11} and Lemma \ref{P3}, we obtain
			$$
			\frac{\partial J(Z_{\overline{r},\overline{x}'',\beta})}{\partial \beta}=m\Big(-\frac{A_{1}}{\beta^{5}}V(\overline{r},\overline{x}'')+\sum_{j=2}^{m}
			\frac{A_{2}}{\beta^{N-3}|z_{1}-z_{j}|^{N-4}}+O(\frac{1}{\beta^{5+\varepsilon}})\Big),
			$$
			where $A_j$, $j=1, 2$ are some positive constants.
		\end{proof}

		\begin{lem}\label{D4}
			It holds
			\begin{equation}\label{d20}
				\int_{\mathbb{R}^{N}}|\Delta \phi|^{2}dx+V(r,x'') \phi^{2}dx=O(\frac{m}{\beta^{4+\varepsilon}}).
			\end{equation}
		\end{lem}
		\begin{proof}
			It follows from \eqref{c14} that
			$$
			\aligned
			&\int_{\mathbb{R}^{N}}|\Delta \phi|^{2}dx+V(r,x'') \phi^{2}dx\\
			=&\int_{\mathbb{R}^{N}}\Big(-\Delta^2 Z_{\overline{r},\overline{x}'',\beta}- V(r,x'')Z_{\overline{r},\overline{x}'',\beta}
			+\Big(|x|^{-\alpha}\ast |Z_{\overline{r},\overline{x}'',\beta}+
			\phi|^{p}\Big)(Z_{\overline{r},\overline{x}'',\beta}+
			\phi)^{p-1}\Big)\phi dx\\
			=&\int_{\mathbb{R}^{N}}\Big(-\Delta^2 Z_{\overline{r},\overline{x}'',\beta}-V(r,x'')Z_{\overline{r},\overline{x}'',\beta}\phi+\Big(|x|^{-\alpha}\ast |Z_{\overline{r},\overline{x}'',\beta}|^{p}\Big)
			Z_{\overline{r},\overline{x}'',\beta}^{p-1}\Big)\phi dx\\
			&+\int_{\mathbb{R}^{N}}\Big(\Big(|x|^{-\alpha}\ast |Z_{\overline{r},\overline{x}'',\beta}+
			\phi|^{p}\Big)(Z_{\overline{r},\overline{x}'',\beta}+
			\phi)^{p-1}-\Big(|x|^{-\alpha}\ast |Z_{\overline{r},\overline{x}'',\beta}|^{p}\Big)
			Z_{\overline{r},\overline{x}'',\beta}^{p-1}\Big)\phi dx\\
			:=&I_{1}+I_{2}.
			\endaligned$$

			By the estimates in Lemma \ref{C5}, we obtain
			$$\aligned
			&\left|\int_{\mathbb{R}^{N}}\Big(-\Delta^2 Z_{\overline{r},\overline{x}'',\beta}-V(r,x'')Z_{\overline{r},\overline{x}'',\beta}\phi+\Big(|x|^{-\alpha}\ast |Z_{\overline{r},\overline{x}'',\beta}|^{p}\Big)
			Z_{\overline{r},\overline{x}'',\beta}^{p-1} \Big)\phi dx\right|\\
			\leq &C\frac{\|\phi\|_{\ast}}{\beta^{2+\varepsilon}}\int_{\mathbb{R}^{N}}
			\sum_{j=1}^{m}\frac{\beta^{\frac{N+4}{2}}}{(1+\beta|x-z_{j}|)^{\frac{N+4}{2}+\tau}}
			\sum_{j=1}^{m}
			\frac{\beta^{\frac{N-4}{2}}}{(1+\beta|x-z_{j}|)^{\frac{N-4}{2}+\tau}} dx
			\leq \frac{Cm}{\beta^{4+2\varepsilon}}.
			\endaligned$$
			So, we have proved
			$$
			|I_{1}|\leq  \frac{Cm}{\beta^{4+2\varepsilon}}.
			$$
			
			Referring Lemma \ref{C4} and by some calculations, we have
			$$
			|I_{2}|
			\leq C \|\phi\|_{\ast}^{3}\int_{\mathbb{R}^{N}}
			\sum_{j=1}^{m}
			\frac{\beta^{\frac{N+4}{2}}}{(1+\beta|x-z_{j}|)^{\frac{N+4}{2}+\tau}}\sum_{j=1}^{m}
			\frac{\beta^{\frac{N-4}{2}}}{(1+\beta|x-z_{j}|)^{\frac{N-4}{2}+\tau}}
			\leq \frac{Cm}{\beta^{4+2\varepsilon}}.
			$$
			So, we can deduce
			$$
			\int_{\mathbb{R}^{N}}|\Delta \phi|^{2}dx+V(r,x'') \phi^{2}dx=O(\frac{m}{\beta^{4+\varepsilon}}).
			$$
		\end{proof}

	\end{appendix}


\begin{thebibliography}{99}
		\bibitem{AH}
		\newblock K. Atkinson and W. Han,
		\newblock {Spherical Harmonics and Approximations on the Unit Sphere: an
			Introduction},
		\newblock \emph{Springer, Berlin, Heidelberg}, (2012).
		
		\bibitem{BE}
		\newblock G. Bianchi and H. Egnell,
		\newblock {A note on the Sobolev inequality}, 
		\newblock \emph{J. Funct. Anal.}, \textbf{100} (1991), 18–24.
		
		
		\bibitem{BGM}
		\newblock E. Berchio, F. Gazzola and E. Mitidieri,
		\newblock {Positivity preserving property for a class of biharmonic elliptic problems}, 
		\newblock \emph{J. Diff. Equ.}, \textbf{229} (2006), 1–23.
		
		\bibitem{BL}
		\newblock H. Brezis and E. Lieb,
		\newblock {Sobolev inequalities with remainder terms}, 
		\newblock \emph{J. Funct. Anal.}, \textbf{62} (1985), 73–86.
		
		\bibitem{BWM}
		\newblock T. Bartsch, T. Weth T. and M. Willem, 
		\newblock {A Sobolev inequality with remainder term and critical equations on domains with topology for the polyharmonic operator}, 
		\newblock \emph{Calc. Var. Partial Dif.}, \textbf{18} (2003), 253-268.
		
		
		\bibitem{CD}
		\newblock D. Cao and W. Dai,
		\newblock Classification of nonnegative solutions to a bi-harmonic equation with Hartree type nonlinearity,
		\newblock \emph{Proceedings of the Royal Society of Edinburgh Section A: Mathematics}, \textbf{149(4)}(2019), 979-994.
		
		\bibitem{CFM}
		\newblock G. Ciraolo, A. Figalli, and F. Maggi,
		\newblock A quantitative analysis of metrics on RN with almost constant positive scalar curvature, with applications to fast diffusion flows,
		\newblock \emph{Int. Math. Res. Not.}, \textbf{2017}(2018), 6780-6779.
		
		\bibitem{CFW}
		\newblock S. Chen, R. Frank, and T. Weth,
		\newblock Remainder terms in the fractional Sobolev inequality,
		\newblock \emph{Indiana Univ.
			Math. J.}, \textbf{62(4)}(2013), 1381-1397.
		
		\bibitem{CWY}
		\newblock W. Chen, J. Wei \& S. Yan,
		\newblock Infinitely many positive solutions for the Schr\"{o}dinger equations in $\R^N$ with critical growth,
		\newblock \emph{J. Differential Equations}, \textbf{252} (2012), 2425--2447.
		
		\bibitem{CY}
		\newblock S. Chang and P. Yang,
		\newblock On uniqueness of solutions of n-th order differential equations in conformal geometry,
		\newblock \emph{Math. Res. Lett.}, \textbf{4}(2019), 91–102.
		
		\bibitem{DFM}
		\newblock M. del Pino, P. Felmer \& M. Musso,
		\newblock Two-bubble solutions in the super-critical Bahri-Coron's problem,
		\newblock  \emph{Calc. Var. Partial Differential Equations}, \textbf{16} (2003), 113--145.	
		
		\bibitem{DGT}
		\newblock S. Deng, M. Grossi, X. Tian,
		\newblock On some weighted fourth-order equations,
		\newblock \emph{Journal of Differential Equations}, \textbf{364} (2023), 612--634.
		
		\bibitem{DHQWF}
		\newblock W. Dai, J. Huang, Y. Qin, B. Wang \& Y. Fang,
		\newblock Regularity and classification of solutions to static Hartree
		equations involving fractional Laplacians,
		\newblock \emph{Discrete Contin. Dyn. Syst.}, \textbf{39} (2019), 1389--1403.	
		
		\bibitem{DLY}
		\newblock Y. Deng, C.-S. Lin \& S. Yan,
		\newblock On the prescribed scalar curvature problem in $\R^N$, local uniqueness and periodicity,
		\newblock \emph{J. Math. Pures Appl.}, \textbf{104} (2015), 1013--1044.
		
		\bibitem{DSW}
		\newblock B. Deng, L. Sun \& J. Wei,
		\newblock Sharp quantitative estimates of Struwe's Decomposition,
		\newblock \emph{}arXiv:2103.15360 [math.AP].
		
		\bibitem{DT1}
		\newblock S. Deng, X. Tian,
		\newblock Caffarelli-Kohn-Nirenberg-Type inequalities related to weighted $p$-Laplace equations,
		\newblock \emph{}arXiv:2212.05459 [math.AP].
		
		\bibitem{DT2}
		\newblock S. Deng, X. Tian,
		\newblock Some weighted fourth-order Hardy-H$\acute{e}$non equations,
		\newblock \emph{Journal of Functional Analysis}, \textbf{284} (2023), 109745.
		
		\bibitem{DT3}
		\newblock S. Deng, X. Tian,
		\newblock Stability of Hardy-Sobolev inequality involving p-Laplace,
		\newblock \emph{}arXiv:2301.07442 [math.AP].
		
		
		
		\bibitem{DT6}
		\newblock S. Deng, X. Tian,
		\newblock Classification and non-degeneracy of positive radial solutions for a weighted fourth-order equation and its application,
		\newblock \emph{}arXiv:arXiv:2308.06014 [math.AP].
		
		\bibitem{DTYZ}
		\newblock S. Deng, X. Tian, M. Yang \& S. Zhao,
		\newblock Some remainder terms of Hardy-Littlewood-Sobolev type inequality,
		\newblock \emph{}arXiv:2305.16857 [math.AP].
		
		\bibitem{DX}
		\newblock F. Dai and Y. Xu,
		\newblock Approximation Theory and Harmonic Analysis on Spheres and Balls,
		\newblock \emph{Springer, New York, NY}, (2013).
		
		\bibitem{DY}
		\newblock L. Du \& M. Yang,
		\newblock Uniqueness and nondegeneracy of solutions for a
		critical nonlocal equation,
		\newblock \emph{Discrete Contin. Dyn. Syst.} \textbf{39} (2019),  5847--5866.
		
		\bibitem{EFJ}
		\newblock D. Edmunds, D. Fortunato and E. Janelli,
		\newblock Critical exponents, critical dimensions and the biharmonic operator, \newblock \emph{Arch. Rational Mech. Anal.}, \textbf{112} (1990), 269–289.
		
		
		\bibitem{FG}
		\newblock A. Figalli and F. Glaudo,
		\newblock On the Sharp Stability of Critical Points of the Sobolev Inequality,
		\newblock \emph{Arch. Ration. Mech. Anal.},  \textbf{237(1)}(2020), 201–258.
		
		\bibitem{FL}
		\newblock J. Frohlich and E. Lenzmann,
		\newblock Mean-field limit of quantum bose gases and nonlinear Hartree equation,
		\newblock \emph{ in: Sminaire E. D. P.}, (2003–2004), Expos nXVIII. 26p.
		
		\bibitem{FN}
		\newblock A. Figalli and R. Neumayer,
		\newblock Gradient stability for the Sobolev inequality: the case $p\geq2$,
		\newblock \emph{J. Eur. Math. Soc.}, \textbf{21(2)} (2019), 319-354.
		
		\bibitem{FZ}
		\newblock A. Figalli, and Y. Zhang,
		\newblock Sharp gradient stability for the Sobolev inequality,
		\newblock \emph{Duke Math. J.}, \textbf{171(12)} (2022), 2407-2459. 
		
		\bibitem{GMYZ}
		\newblock F. Gao, V. Moroz, M. Yang and S. Zhao,
		\newblock Construction of infinitely many solutions for a critical Choquard equation via local Poho$\check{z}$aev identities,
		\newblock \emph{Calc. Var. Partial Differential Equations} \textbf {61} (2022), Art. 222, 47 pp.
		
		\bibitem{GL}
		\newblock Y. Guo, B. Li,
		\newblock Infinitely many solutions for the prescribed curvature problem of polyharmonic operator,
		\newblock \emph{Calc. Var. Partial Differ. Equ.}, \textbf{46} (2013), 809--836.
		
		
		\bibitem{GPY}
		\newblock Y. Guo, S. Peng \& S. Yan,
		\newblock \emph{Local uniqueness and periodicity induced by concentration,}
		\newblock Proc. Lond. Math. Soc. (3), \textbf{114} (2017), 1005--1043.
		
		\bibitem{GLN}
		\newblock Y. Guo, T. Liu, and J. Nie, 
		\newblock {Construction of solutions for the polyharmonic equation via local Poho$\check{z}$aev identities}, 
		\newblock \emph{Calc. Var. Partial Dif.}, \textbf{58} (2019), 33p.
		
		
		\bibitem{K}
		\newblock V. Karpman,
		\newblock Stabilization of soliton instabilities by high-order dispersion: fourth order nonlinear Schr\"{o}dinger-type equations, 
		\newblock \emph{Phys. Rev. E 53}, \textbf{2} (1996), 1336–1339.
		
		\bibitem{L1}
		\newblock E. Lenzmann.,
		\newblock Uniqueness of ground states for pseudorelativistic Hartree equations,
		\newblock \emph{Anal. PDE}, \textbf{2(1)}(2009), 1–27.
		
		\bibitem{L}
		\newblock E. Lieb.,
		\newblock Sharp constants in the Hardy-Littlewood-Sobolev and related inequalities,
		\newblock \emph{Ann.
			of Math.}, \textbf{118}(1983), 349–374.
		
		\bibitem{LL}
		\newblock E. Lieb and M. Loss, 
		\newblock "Analysis,"
		\newblock \emph{Graduate Studies in Mathematics}, American Mathematical Society,
		Providence, Rhode Island, 2001.
		
		\bibitem{L2}
		\newblock C. Lin,
		\newblock A classification of solutions of a conformally invariant fourth order equation in
		$\R^n$,
		\newblock \emph{Comment. Math. Helv.}, \textbf{73}(1998), 206–231.
		
		\bibitem{L3}
		\newblock P.L. Lions,
		\newblock The concentration-compactness principle in the calculus of variations. The locally compact case. I,
		\newblock  \emph{Ann. Inst. H. Poincar$\acute{e}$Anal. Non Lin$\acute{e}$aire}, \textbf{1} (1984), 109--145.
		
		\bibitem{LWX}
		\newblock Y. Li, J. Wei \& H. Xu,
		\newblock \emph{Multi-bump solutions of $-\Delta u = K(x)u^{\frac{n+2}{n-2}}$ on lattices in $\R^n$},
		\newblock J. Reine Angew. Math. {\bf 743} (2018), 163--211.
		
		\bibitem{LW}
		\newblock G. Lu and J. Wei,
		\newblock On a Sobolev inequality with remainder terms,
		\newblock \emph{Proceedings of the American Mathematical Society}, \textbf{128(1)}(1999), 75-84.
		
		\bibitem{LXLTX}
		\newblock X. Li, C. Liu, X. Tang and G. Xu,
		\newblock Nondegeneracy of positive bubble solutions for generalized energy- critical Haetree equations,
		\newblock \emph{Preprint arXiv}: 2304.04139v1.
		
		\bibitem{PV}
		\newblock A. Pistoia, G. Vaira,
		\newblock Nondegeneracy of the bubble for the critical p-Laplace equation,
		\newblock  \emph{Proc. Roy. Soc. Edinburgh Sect. A}, \textbf{151}(1) (2021), 151--168.
		
		\bibitem{PWY}
		\newblock S. Peng, C. Wang \& S. Yan,
		\newblock Construction of solutions via local Poho$\check{z}$aev identities,
		\newblock  \emph{J. Funct. Anal.}, \textbf{274} (2018), 2606--2633.
		
		\bibitem{PYZ}
		\newblock P. Piccione, M. Yang \& S. Zhao,
		\newblock Quantitative stability of a nonlocal Sobolev inequality,
		\newblock  \emph{}arXiv:2306.16883  .
		
		\bibitem{RSW}
		\newblock V. R$\breve{a}$dulescu, D. Smets, M. Willem,
		\newblock Hardy-Sobolev inequalities with remainder terms,
		\newblock \emph{Topol. Methods Nonlinear Anal.}, \textbf{20(1)}(2002), 145-149.
		
		\bibitem{S}
		\newblock M. Struwe,
		\newblock A global compactness result for elliptic boundary value problems involving limiting nonlinearities,
		\newblock \emph{Math. Z.}, \textbf{187(4)}(1984), 511-517.
		
		
		\bibitem{SV}
		\newblock R. Servadei and E. Valdinoci,
		\newblock Variational methods for non-local operators of elliptic type,
		\newblock \emph{Discrete Contin.
			Dyn. Syst.}, \textbf{33(5)}(2013), 2105-2137.
		
		\bibitem{W}
		\newblock M. Weinstein,
		\newblock Modulational stability of ground states of nonlinear Schr\"{o}dinger
		equations,
		\newblock \emph{ SIAM J. Math. Anal.}, \textbf{16(3)}(1985), 472–491.
		
		\bibitem{WW2}
		\newblock Z. Wang and M. Willem,
		\newblock Caffarelli-Kohn-Nirenberg inequalities with remainder terms,
		\newblock \emph{J. Funct. Anal.}, \textbf{203(2)}(2003), 550-568.
		
		\bibitem{WW}
		\newblock J. Wei and M. Winter,
		\newblock Strongly interacting bumps for the Schr\"{o}dinger-Newton equations,
		\newblock \emph{J. Math. Phys.}, \textbf{50(1)}(2009), 22p.
		
		
		\bibitem{WW3}
		\newblock J. Wei and Y. Wu,
		\newblock On the stability of the Caffarelli-Kohn-Nirenberg inequality,
		\newblock \emph{Math. Ann.}, \textbf{384}(2022), 1509-1546.
		
		\bibitem{WX}
		\newblock J. Wei and X. Xu,
		\newblock Classification of solutions of higher order conformally invariant equations,
		\newblock \emph{Math. Ann.}, \textbf{313}(1999), 207–228.
		
		\bibitem{WY1}
		\newblock J. Wei \& S. Yan,
		\newblock Infinitely many solutions for the prescribed scalar curvature problem on $S^{N}$,
		\newblock \emph{J. Funct. Anal.}, \textbf{258} (2010), 3048--3081.
		
		\bibitem{X}
		\newblock X. Xu,
		\newblock Uniqueness theorem for the entire positive solutions of biharmonic equations in $\mathbb{R}^{N}$,
		\newblock \emph{Proceedings of the Royal Society of Edinburgh Section A: Mathematics}, \textbf{130}(3) (2000), 651--670.
		
		\bibitem{YYZ}
		\newblock M. Yang, W. Ye and S. Zhao,
		\newblock Existence of concentrating solutions of the Hartree type Brezis-Nirenberg problem,
		\newblock \emph{J. Diff. Equ.}, \textbf{344}(2023), 260–324.
		
		\bibitem{YGRY}
		\newblock  W. Ye, F.  Gao, V. R\u adulescu and M. Yang,
		\newblock Construction of infinitely many solutions for two-component Bose-Einstein condensates with nonlocal   critical interaction,
		\newblock \emph{J. Diff. Equ.}, \textbf{375}(2023), 415–474.
		
		
		
		
	\end{thebibliography}
\end{document}